\documentclass[11pt]{article}
\usepackage{amsthm,amsmath,amssymb}
\usepackage{natbib}
\usepackage{multirow}
\usepackage[pdftex]{graphicx}
\usepackage{subfigure}
\usepackage{makecell}
\usepackage{booktabs}
\usepackage{array}
\usepackage{fullpage}
\usepackage{url}
\usepackage{algorithm}
\usepackage{algorithmic}
\usepackage{bm}
\usepackage{smile}
\usepackage{mathtools}
\usepackage{wrapfig}
\usepackage{lipsum}
\usepackage{mathrsfs}
\usepackage{dsfont}
\usepackage{booktabs}
\usepackage{natbib}
\usepackage{multirow}
\usepackage{appendix}
\usepackage[usenames,dvipsnames,svgnames,table]{xcolor}
\usepackage[colorlinks=true,linkcolor=blue,urlcolor=blue, citecolor=blue]{hyperref}

\newcolumntype {Q}{>{$\displaystyle}l<{$}}
\newcolumntype {A}{>{$}c <{$}}

\def\tr{\mathop{\text{tr}}\kern.2ex}

\def\P{{\mathbb P}}
\def\E{{\mathbb E}}

\long
\def\comment#1{}

\def\tr{\mathop{\text{Tr}}}

\providecommand{\norm}[1]{\vvvert#1\vvvert}

\def\reals{{\mathbb{R}}}
\def\R{{\reals}}

\newtheorem{assumption}{Assumption}
\newtheorem{condition}{Condition}

\newcommand{\blue}{}

\newcommand{\hf}[1]{{{#1}}}

\theoremstyle{definition}

\numberwithin{equation}{section}
\numberwithin{theorem}{section}
\numberwithin{corollary}{section}
\numberwithin{asmp}{section}
\numberwithin{definition}{section}

\def\E{\mathbb{E}}
\def \U{\mathbb{U}}
\def\P{\mathbb{P}}
\def\Q{\mathbb{Q}}
\def\Sm{\mathbb{S}}
\def\S{\mathcal{S}}
\def\T{\mathbb{T}}
\def \U{\mathbb{U}}

\def\btheta{\bm\theta}

\def\bbeta{\bm\beta}

\def\bgamma{\bm{\gamma}}
\def\bwn{\bm{W}_n}
\def\btn{\hat{\bm{t}}_n}
\def\btsn{\bm{t}^*_n}
\def\Deltab{\mathbf{\Delta}}
\newcommand{\xlast}{\tilde{\bx}}
\newcommand{\Xlast}{\tilde{\bX}}
\newcommand{\indic}{\mathds{1}}
\newcommand{\sfH}{\mathsf{H}}
\newcommand{\sfC}{\mathsf{C}}
\newcommand{\sfK}{\mathsf{K}}
\newcommand{\sfA}{\mathsf{A}}
\newcommand{\sfc}{\mathsf{c}}
\def\F{\cF}
\newcommand*\diff{\mathop{}\!\mathrm{d}}
\allowdisplaybreaks
\begin{document}

\title{{\LARGE On Rank Estimators in Increasing Dimensions}}
\author{Yanqin Fan\thanks{%
Department of Economics, University of Washington, Seattle, WA 98195, USA;
email: \texttt{fany88@uw.edu}.}~\thanks{Corresponding author}, ~Fang Han\thanks{%
Department of Statistics, University of Washington, Seattle, WA 98195, USA;
e-mail: \texttt{fanghan@uw.edu}.}, ~Wei Li\thanks{%
School of Mathematical Sciences, Peking University, Beijing 100871, China;
e-mail: \texttt{weylpeking@pku.edu.cn}.},~~and Xiao-Hua Zhou\thanks{%
Department of Biostatistics, University of Washington, Seattle, WA 98195,
USA; e-mail: \texttt{azhou@uw.edu}.}~\thanks{%
International Center for Mathematical Research, Peking University, Beijing,
China.}}
\maketitle

\begin{abstract}
The family of rank estimators, including Han's maximum rank correlation %
\citep{han1987non} as a notable example, has been widely exploited in
studying regression problems. 
%These include Han's maximum rank correlation \citep{han1987non} and Cavanagh and Sherman's generalized maximum rank correlation \citep{cavanagh1998rank} estimators for monotonic index models, Khan and Tamer's rank estimator for semiparametric censored duration models \citep{khan2007partial}, and Abrevaya and Shin's rank estimator for generalized partially linear index models \citep{abrevaya2011rank}. 
For these estimators, although the linear index is introduced for
alleviating the impact of dimensionality, the effect of large dimension on
inference is rarely studied. This paper fills this gap via studying the
statistical properties of a larger family of M-estimators, whose objective
functions are formulated as U-processes and may be discontinuous in
increasing dimension set-up where the number of parameters, $p_{n}$, in the
model is allowed to increase with the sample size, $n$. First, we find that
often in estimation, as $p_{n}/n\rightarrow 0$, $(p_{n}/n)^{1/2}$ rate of
convergence is obtainable. Second, we establish Bahadur-type bounds and
study the validity of normal approximation, which we find often requires a
much stronger scaling requirement than $p_{n}^{2}/n\rightarrow 0$. Third, we
state conditions under which the numerical derivative estimator of
asymptotic covariance matrix is consistent, and show that the step size in
implementing the covariance estimator has to be adjusted with respect to $%
p_{n}$. 
%Lastly, technically speaking, for handling the discontinuous loss functions in these examples, we establish a new maximal inequality for degenerate U-processes of many covariates, which plays a pivotal role in our analysis and is of independent interest. 
All theoretical results are further backed up by simulation studies. 
%{FAN: Emphasize poor accuracy of normal approximation?}
\end{abstract}

\setlength{\abovedisplayskip}{5pt} \setlength{\belowdisplayskip}{5pt} %
\setlength{\abovedisplayshortskip}{5pt} \setlength{%
\belowdisplayshortskip}{5pt}

\setlength{\abovedisplayskip}{5pt} \setlength{\belowdisplayskip}{5pt} %
\setlength{\abovedisplayshortskip}{5pt} \setlength{%
\belowdisplayshortskip}{5pt}

\setlength{\abovedisplayskip}{5pt} \setlength{\belowdisplayskip}{5pt} %
\setlength{\abovedisplayshortskip}{5pt} \setlength{%
\belowdisplayshortskip}{5pt}

\setlength{\abovedisplayskip}{5pt} \setlength{\belowdisplayskip}{5pt} %
\setlength{\abovedisplayshortskip}{5pt} \setlength{%
\belowdisplayshortskip}{5pt}

\setlength{\abovedisplayskip}{5pt} \setlength{\belowdisplayskip}{5pt} %
\setlength{\abovedisplayshortskip}{5pt} \setlength{%
\belowdisplayshortskip}{5pt}

{Keywords:} Bahadur-type bounds, degenerate U-processes, maximal
inequalities, uniform bounds.

{JEL Codes}: C55, C14.

\newpage

\section{Introduction}

\label{sec:intro}

\subsection{The General Set-up, Motivation, and Main Results}

Let $\bZ_{1},\ldots ,\bZ_{n}\in {\mathbb{R}}^{m_{n}}$ denote a random sample
of size $n$ from the probability measure $\mathbb{P}$. Let $\cF:=\{f(\cdot
,\cdot ;\btheta ):\btheta \in \Theta \subset {\mathbb{R}}^{p_{n}}\}$ be
a class of real-valued, possibly asymmetric and \textit{discontinuous},
functions on ${\mathbb{R}}^{m_{n}}\times {\mathbb{R}}^{m_{n}}$. This paper
studies the following M-estimator with an objective function of a U-process
structure, 
\begin{equation}
\hat{\btheta }_{n}:=\argmax_{\btheta \in \Theta }\Gamma _{n}(\btheta )=%
\argmax_{\btheta \in \Theta }\frac{1}{n(n-1)}\sum_{i\neq j=1}^{n}f(\bZ_{i},%
\bZ_{j};\btheta ).  \label{eq:U}
\end{equation}%
Let 
\begin{equation*}
\btheta _{0}:=\argmax_{\btheta \in \Theta }\Gamma (\btheta )=\argmax_{%
\btheta \in \Theta }\mathbb{E}\Gamma _{n}(\btheta ).
\end{equation*}%
This paper aims to establish asymptotic properties of $\hat{\btheta }_{n}$
as an estimator of $\btheta _{0}$ in situations with large or increasing
dimensions $m_{n}\rightarrow \infty $ and $p_{n}\rightarrow \infty $ (with
respect to the sample size $n$), {to} which existing results do not apply.

Members of \eqref{eq:U} include the following notable examples proposed and
studied in the current literature in fixed dimension, i.e., $m_{n}\equiv m$
and $p_{n}\equiv p$ for all $n$: (1) Han's maximum rank correlation (MRC)
estimator for the generalized regression model \citep{han1987non}; (2)
Cavanagh and Sherman's rank estimator for the same model as Han's %
\citep{cavanagh1998rank}; (3) Khan and Tamer's rank estimator for the
semiparametric censored duration model \citep{khan2007partial}; and (4)
Abrevaya and Shin's rank estimator for the generalized partially linear
index model \citep{abrevaya2011rank}. One common feature of these models is
the presence of a linear index of the form $\bx^{\top }\btheta $, where $%
\bx$ represents covariates of dimension $p$ which is typically large in many
economic applications. The linear index structure is introduced to alleviate
the \textquotedblleft curse of dimensionality\textquotedblright\ associated
with fully nonparametric models. Although motivated by possibly large
dimension $p$, properties of $\hat{\btheta }_{n}$ in these examples have
only been established for fixed $p$ when $n$ approaches {infinity} (i.e., $p$
does not change with $n$). Instead, this paper models the large $p$ case by
allowing $p$ to go to infinity as $n\rightarrow \infty $, denoted as $p_{n}$%
, facilitating an explicit characterization of the effect of dimensionality
on inference in these models.

More broadly, for the general set-up \eqref{eq:U}, we allow both $m_{n}$ and 
$p_{n}$ to go to infinity as $n\rightarrow \infty $ and establish the
following properties of $\hat{\btheta }_{n}$: (i) consistency; (ii) rate
of convergence; (iii) normal approximation; and (iv) accuracy of normal
approximation. The last property is also referred to as the
\textquotedblleft Bahadur-Kiefer representation\textquotedblright\ or {simply%
} the \textquotedblleft Bahadur-type bound\textquotedblright\ %
\citep{bahadur1966note,kiefer1967bahadur, he1996general}, and is the major
focus of this paper. 
%{% FAN: reference formats different here.} \hf{ Not sure what is the difference...}
%{\ FAN: Need revision! There is }$\nu ${!} {\ }%
Specifically, in Theorems \ref{thm: consistency}, \ref{thm: general consrate}%
, and \ref{thm:generalASN}, under different scaling requirements for $n$, $%
p_{n}$, and $\nu _{n}$, where $\nu _{n}$ characterizes the function
complexity of $\cF$, we prove consistency, efficient rate of convergence,
and derive Bahadur-type bounds for the general M-estimator $\hat{\btheta }%
_{n}$ of the form \eqref{eq:U}. %These results reveal, although the scaling
%requirement $(\nu +p)/n\rightarrow 0$\ is usually sufficient to guarantee
%consistency and minimax rate of convergence {\citep{yu1997assouad}}, a
%possibly much stronger scaling requirement than $(\nu +p)^{2}/n\rightarrow 0$%
%\ is demanded for establishing $\hat{\btheta }_{n}$'s asymptotic normality
%(ASN).
To facilitate inference, we {construct consistent estimators of the
asymptotic covariance matrix of }$\hat{\btheta }_{n}$ {similar to the
numerical derivative estimators in \cite{pakes1989simulation}, \cite%
{sherman1993limiting}, and \cite{khan2007partial}. }The increasing dimension
set-up in this paper reveals that for consistent variance-covariance matrix
estimation, the step size in computing the numerical derivative should
depend not only on the sample size $n$ but also the dimensions $m_{n}$ and $%
p_{n}$.

To provide further insight on the role of the dimension $p_{n}$, we apply
our general results, Bahadur-type bounds especially, to the aforementioned
rank estimators (1)-(4). Note that for these estimators $\nu _{n}=m_{n}=p_{n}
$. Corollaries \ref{cor:H}--\ref{cor:A} provide sufficient conditions to
guarantee consistency, efficient rate of convergence, and asymptotic
normality (ASN) of the rank correlation estimators in increasing dimension.
They demonstrate that, compared to competing alternatives such as simple
linear regression, in terms of estimation, rank estimators are very
appealing, maintaining the minimax optimal $(p_{n}/n)^{1/2}$ rates %
\citep{yu1997assouad}, while enjoying an additional robustness property to
outliers and modeling assumptions. With regard to normal approximation, on
the other hand, a much stronger scaling requirement might be needed, and a
lower accuracy in normal approximation is anticipated. This observation also
echoes a common belief in robust statistics that stronger scaling
requirement than $p_{n}^{2}/n\rightarrow 0$ is needed for normal
approximation validity \citep{jurevckova2012methodology}. 
%As a side result, specific to Han's
%maximum rank correlation estimator, when fixing the dimension, our result
%implies that the Bahadur-type bound for normal approximation is of the order
%$(\log n)^{1/2}n^{-5/8}$, much slower than $n^{-1},$ which appears to echo
%Bahadur and Kiefer's insightful observations on sample quantiles. See Remark %
%\ref{remark:kiefer} for detailed discussion and relation to Kiefer's results %
%\citep{kiefer1967bahadur}.
%This echoes Kiefer's observation, which proved that, in various cases, the bound for sample quantiles is strictly slower than $n^{-1}$, of the order $(\log\log n)^{4/3}n^{-3/4}$.

All the theoretical results are further backed up by simulation studies. In
particular, using Han's MRC estimator introduced below, we have demonstrated
that for a given sample size, the accuracy of the normal approximation
deteriorates quickly as the number of parameters $p_{n}$ increases,
indicating that our theoretical bound is difficult to improve further. Also,
our simulation results suggest that for variance estimation, the step size
needs to be adjusted with respect to $p_{n}$. Practically, our results
indicate that although the linear index was introduced to alleviate the
curse of dimensionality, one must be cautious in conducting inference using
rank estimators when there are many covariates.

{\color{red} }

\subsection{The Generalized Regression Model and Han's MRC}

\label{subsec: HanMRC}

Han's MRC in Example (1) is the first rank correlation estimator proposed to
estimate the parameter $\bbeta _{0}$ in the generalized regression model: 
\begin{equation}
Y=D\circ F(\bX^{\top }\bbeta _{0},\epsilon ),  \label{GRM}
\end{equation}%
where $\bbeta _{0}\in \mathbb{R}^{p_{n}+1}$, $F(\cdot ,\cdot )$ is a
strictly increasing function of each of its arguments, and $D(\cdot )$ is a
non-degenerate monotone increasing function of its argument. Important
members of the generalized regression model in (\ref{GRM}) include many
widely known and extensively used econometrics models in diverse areas in
empirical microeconomics such as the binary choice models, the ordered
discrete response models, transformation models with unknown transformation
functions, the censored regression models, and proportional and additive
hazard models under the independence assumption and monotonicity constraints.%
\textbf{\ }

\cite{han1987non} proposed estimating $\bbeta _{0}$ in (\ref{GRM}) with 
\begin{equation}
\hat\bbeta _{n}^{\mathsf{H}}=\argmax_{\bbeta :\beta _{1}=1}\Big\{\frac{%
1}{n(n-1)}\sum_{i\neq j}\mathds{1}(Y_{i}>Y_{j})\mathds{1}(\bX_{i}^{\top }\bm%
\beta >\bX_{j}^{\top }\bbeta )\Big\}. \label{eqn:MRC}
\end{equation}%
%
%
%
%
%{\blue Without loss of generality, we assume that $\bX_i$ $(i = 1,\ldots,n)$ are independent and identically distributed with mean zero.}
For model identification, following \cite{sherman1993limiting}, we assume
the first component of $\bbeta _{0}$ is equal to 1, and express $\bbeta
_{0}$ as $\bbeta _{0}=(1,\btheta _{0}^{\top })^{\top }$. We consider
estimating $\btheta _{0}$ by $\hat\btheta _{n}^{\mathsf{H}}:=\hat\bbeta _{n,-1}^{\mathsf{H}}$, the subvector of $\hat\bbeta _{n}^{\mathsf{H%
}}$ excluding its first component. We will use the generalized regression
model (\ref{GRM}) and Han's MRC $\hat\btheta _{n}^{\mathsf{H}}$ to
illustrate our notation, assumptions, and main results in Section \ref%
{sec:general}. We defer a rigorous analysis of Han's estimator including
verification of assumptions to Section \ref{sec:application} which also
presents results for the other three rank correlation estimators. 

Empirically, consider estimating the individual demand curve for a durable
good such as a refrigerator. Let $Y_{i}$ be whether the individual $i$ buys
a refrigerator and $\bX_{i}$ be the vector of characteristics of the
individual and the refrigerator included in the model. There are many
potential candidates for the components of $\bX_{i}$ such as personal
income, marital status, the number of children, space of the kitchen, food
habits; size of the refrigerator, temperature controls, lighting, shelves,
dairy compartment, chiller, door styles. Assuming a single index form with $%
m_n=p_n+1$, this binary choice model falls into our framework with (\ref{GRM}).
Our increasing dimension set-up allows more characteristics to be included
in $\bX_{i}$ as the sample size $n$ increases and our results show that even
with the single index form, estimation and inference are possible if $p_n$
increases very mildly with $n$ but otherwise are very challenging.

\subsection{A Brief Review of Related Works and Technical Challenge}

In contrast with the fixed dimension setting, where the model is assumed
unchanged as $n$ goes to infinity, the increasing dimension triangular array
setting 
\citep{portnoy1984asymptotic, fan2015power,
chernozhukov2015valid,chernozhukov2017central} makes our analysis different
from and more challenging than most existing ones (cf. Theorem 3.2.16 and
Example 3.2.22 in \cite{wellner1996weak}, or the main theorem in \cite%
{he1996general}). Technically, this paper builds on and contributes to two
distinct literatures: the literature on estimation and inference in
increasing dimension where existing works exclude \textit{discontinuous}
loss functions and the literature on rank estimation where existing works
focus exclusively on finite dimensions. As a technical contribution, we
establish a maximal inequality, yielding a uniform bound for degenerate
U-processes in increasing dimensions which not only allows us to extend
existing results on rank estimation in finite dimension to increasing
dimensions but also establish Bahadur-type bounds. Besides the crucial role
played by our new maximal inequality for degenerate U-processes in this
paper, it should prove to be an indispensible tool in nonparametric and
semiparametric econometrics in increasing dimensions where many estimators
and test statistics are closely related to U-processes.

Since Huber's seminal paper \citep{huber1973robust}, there has been a long
history in statistics on evaluating the impact of parameter dimension on
inference. Huber himself raised questions on the scaling limits of $(n,p_n)$
for assuring M-estimation consistency and asymptotic normality in his 1973
paper \citep{huber1973robust}. For addressing them, \cite%
{portnoy1984asymptotic}, \cite{portnoy1985asymptotic}, \cite%
{mammen1989asymptotics}, and \cite{mammen1993bootstrap} studied the linear
regression model using smooth M-estimators such as the ordinary least
squares. Their results revealed that, %an interesting phenomenon that, 
in response to Huber's question, {for the simple linear regression model,}
asymptotic normality is {usually} attainable even when $p_n^{2}/n$ is large.
In contrast, \cite{portnoy1988asymptotic} studied maximum likelihood
estimators of generalized linear models, and proved that, for guaranteeing
the validity of normal approximation, the requirement $p_n^{2}/n\rightarrow
0 $ is in general unrelaxable. Different from the analysis in large $%
p_n^{2}/n$ setting, the techniques in \cite{portnoy1988asymptotic} are
applicable to more general cases. For example, focusing on the general
likelihood problem with a \textit{differentiable} likelihood function, \cite%
{spokoiny2012parametric} has provided a finite-sample analysis of normal
approximation accuracy. Related results have also been developed in \cite%
{he2000parameters}. As a direct consequence, a set of regularity conditions
could be derived for constructing Bahadur-type bounds, guaranteeing ASN
provided {some scaling requirements hold}. 
%At the core of his analysis is a new maximal inequalities chaining the empirical processes of score functions.

Extending existing works allowing for increasing parameter dimension, this
paper studies asymptotic properties of $\hat{\btheta}_{n}$ in \eqref{eq:U}%
, {allowing both} $m_n$ and $p_n$ to go to infinity as $n\rightarrow \infty $%
. The potential discontinuity and U-process structure of the objective
function $\Gamma _{n}(\btheta)$ prevent results or the proof strategy in
the current literature on increasing parameter dimension from being directly
applicable. On the other hand, for \eqref{eq:U}, the increasing dimension
set-up in this paper poses technical challenges to the proof strategy
adopted for fixed $m_n$ and $p_n$ exclusively studied in the current
literature. To see this, recall that the main argument used in the current
literature to establish asymptotic properties {for} estimators of the form %
\eqref{eq:U} for fixed $m_n$ and $p_n$ follows Sherman 
\citep{sherman1993limiting,
sherman1994maximal}, which relies on the Hoeffding decomposition, a uniform
bound for degenerate U-processes, and the classical M-estimation framework
tracing back to Huber's seminal paper, \cite{huber1967behavior}. 
%an appropriate CLT for sample averages.
Specifically, for the statistic $\Gamma _{n}(\btheta)$ in \eqref{eq:U}, 
\cite{hoeffding1948class} derived the following well-known expansion now
known as the Hoeffding decomposition: 
\begin{equation}
\Gamma _{n}(\btheta)=\Gamma (\btheta)+\mathbb{P}_{n}g(\cdot ;\btheta )+%
\mathbb{U}_{n}h(\cdot ,\cdot ;\btheta),  \label{eq:hoeffding}
\end{equation}%
where%
\begin{align}
g(\bz;\btheta)& :=\mathbb{E}f(\bz,\cdot ;\btheta)+\mathbb{E}f(\cdot ,\bz;%
\btheta)-2\Gamma (\btheta),  \notag  \label{eq:h} \\
h(\bz_{1},\bz_{2};\btheta)& :=f(\bz_{1},\bz_{2};\btheta)-\mathbb{E}f(\bz%
_{1},\cdot ;\btheta)-\mathbb{E}f(\cdot ,\bz_{2};\btheta)+\Gamma (\bm%
\theta ), \\
\mathbb{P}_{n}g(\cdot ;\btheta)& :=\sum_{i=1}^{n}g(\bZ_{i})/n,\text{ and} 
\notag \\
\mathbb{U}_{n}h(\cdot ,\cdot ;\btheta)& :=\sum_{i\neq j=1}^{n}h(\bZ_{i},\bZ%
_{j};\btheta)/\{n(n-1)\}.  \notag
\end{align}%
\cite{hoeffding1948class} further showed that for fixed $m_n$ and $p_n$, 
\begin{equation}
\Gamma _{n}(\btheta)\approx \underbrace{\Gamma (\btheta)+\mathbb{P}%
_{n}g(\cdot ;\btheta)}_{\tilde{\Gamma}_{n}(\btheta)},
\label{eq:tildeGamma}
\end{equation}%
where the remainder term $\mathbb{U}_{n}h(\cdot ,\cdot ;\btheta)$,
formulated as a degenerate U-statistic, is asymptotically negligible in
large samples. As a result, $\hat{\btheta}_{n}$ is asymptotically
equivalent to $\tilde{\btheta}_{n}$ defined below: 
\begin{equation}
\tilde{\btheta}_{n}:=\argmax_{\btheta\in \Theta }\tilde{\Gamma}_{n}(\bm%
\theta).  \label{eq:tildeBeta}
\end{equation}%
Sherman \citep{sherman1993limiting, sherman1994maximal} was the first to
notice that, by \eqref{eq:hoeffding} and the negligibility of $\mathbb{U}%
_{n}h(\cdot ,\cdot ;\btheta)$, the U-statistic formulation has
intrinsically helped smooth the loss function in \eqref{eq:U} from $\Gamma
_{n}(\btheta)$ to $\tilde{\Gamma}_{n}(\btheta)$, and hence renders an
asymptotically normal estimator $\hat{\btheta}_{n}$, even though the
original loss function $\Gamma _{n}(\btheta)$ may not be differentiable.

For increasing dimensions $m_n$ and $p_n$, the Hoeffding decomposition of $\
\Gamma _{n}(\btheta)$ takes the same form as in the case of fixed $m_n$
and $p_n$. However existing maximal inequalities or uniform bounds for
degenerate U-processes for finite dimensions crucial to Sherman 
\citep{sherman1993limiting,
sherman1994maximal} and the classical M-estimation theory for finite
dimensions are inapplicable. In response to the first challenge, this paper
develops a maximal inequality, yielding a uniform bound for degenerate
U-processes in increasing dimensions, which allows us to show that under
regularity conditions, $\hat{\btheta}_{n}$ is asymptotically equivalent to 
$\tilde{\btheta}_{n}$. Due to the smoothness of $\tilde{\Gamma}_{n}(\bm%
\theta )$, we are able to build on and improve arguments used in the proofs
of \cite{spokoiny2012parametric} on M-estimators with \textit{differentiable}
objective functions in increasing dimensions to establish asymptotic
properties of $\tilde{\btheta}_{n}$.

%{DISCUSS\ SIMULATION\ HERE.} \hf{ Will do when we see the results : )}

%{%{To Professor Fan: Do we use ``increasing dimension" or
%``increasing dimensions"? FAN: In the general set-up, both m and p diverge,
%but in the examples, m=p}}

%In the following, to better understand our contribution to literature, we first snitch all existing results in increasing dimension through a common thread via providing a general reduction scheme. Under this framework, we will then move on to introduce briefly our main statistical results, along with the newly developed probability tools, with detailed results presented in Sections \ref{} and \ref{}.

\subsection{Notation}

For a set $\mathcal{S}$, denote its binary Cartesian product as $\mathcal{S}%
\otimes \mathcal{S}$. For a probability measure $\mathbb{P}$, denote its
product measure as $\mathbb{P}\otimes\mathbb{P}$. For $q\in[1,\infty]$, the $%
L_q$-norm of a vector $\bbeta$ is denoted by $\norm{\bbeta}_{q}$. The $L_q$%
-induced matrix operator norm of a matrix $\Ab$ is denoted by $\norm{\Ab}%
_{q} $. One example is the spectral norm $\norm{\Ab}_2$, which represents
the maximal singular value of $\Ab$. In the sequel, when no confusion is
possible, we will omit the subscript in the $L_q$-norm of $\bbeta$ or $\Ab$
when $q=2$. The minimum and maximum eigenvalues of a {real symmetric} matrix
are denoted by $\lambda_{\min}(\cdot) $ and $\lambda_{\max}(\cdot)$
respectively. Let $\Ib_p$ denote the $p\times p $ identity matrix. Let $%
\mathbb{S}^{p-1}$ denote the unit-sphere of ${{\mathbb{R}}}^p$ under $%
\norm{\cdot}$. %
%{Define the norm of an operator $\bm{f}(\cdot)$
%from a normed vector space $({{\mathbb{R}}}^m,\|\cdot\|)$ to another normed
%vector space $({{\mathbb{R}}}^p,\|\cdot\|)$ as %\begin{align*}
%$\|\bm{f}(\cdot)\|:=\sup_{\bu\in\mathbb{S}^{m-1}}\|\bm{f}(\bu)\|$.}
For a twice differentiable real-valued function $\tau(\btheta)$, let $%
\nabla_1\tau(\btheta)$ denote the vector of partial derivatives $(\partial
\tau/\partial \theta_1,\ldots,\partial\tau/\partial \theta_p)^\top$ and $%
\nabla_2\tau(\btheta)$ denote the Hessian matrix of $\tau(\btheta)$. 
%\end{align*}
Let $\mathcal{B}(\btheta_0,r)=\{\btheta\in\Theta,\norm{\btheta-\btheta_0}%
<r\}$ denote an open ball of radius $r>0$ centered at $\btheta_0\in\Theta$%
, and let $\overline{\mathcal{B}}(\btheta_0,r)=\{\btheta\in\Theta,%
\norm{\btheta-\btheta_0}\leq r\}$ denote a closed ball of center $\bm%
\theta_0 $ and radius $r$. For two real numbers $a$ and $b$, we define $%
a\vee b=\max(a,b)$ and $a\wedge b=\min(a,b)$. We use $\xrightarrow{\P}$ to
denote convergence in probability with respect to $\mathbb{P} $, and $%
\Rightarrow$ to denote convergence in distribution. {{For any two real
sequences $\{a_n\}$ and $\{b_n\}$, we write $a_n=O(b_n)$ if there exists an
absolute positive constant $C$ such that $|a_n|\leq C|b_n|$ for any large
enough $n$. We write $a_n\asymp b_n$ if both $a_n=O(b_n)$ and $b_n=O(a_n)$
hold. We write $a_n=o(b_n)$ if for any absolute positive constant $C$, we
have $|a_n|\leq C|b_n|$ for any large enough $n$. We write $a_n=O_{\mathbb{P}%
}(b_n)$ and $a_n=o_{\mathbb{P}}(b_n)$ if $a_n=O(b_n)$ and $a_n=o(b_n)$ hold
stochastically.}} 
%{\blue For a set of real numbers $a_n$ and random variables $X_n$, we use $a_n=o(1)$  to denote that $a_n\rightarrow 0$, and $X_n=o_\P(1)$  to denote that $X_n\xrightarrow{\P}0$ respectively, as $n\rightarrow \infty$. We further use $X_n=O_\P(1)$ to denote that $X_n$ is stochastically bounded, that is, for any $\varepsilon>0$, there exist  finite $M>0$ and  $N>0$ such that $\P(|X_n|>M)<\varepsilon$ for any $n>N$.}
%We use $\indic(\cdot)$ to denote the indicator function.
%, and denote $H(s)=s(1+\log(1/s))$ for $s>0$.
We let $C,C^{\prime },C^{\prime \prime },c,c^{\prime },c^{\prime \prime
},\ldots$ be generic {absolute} positive constants, whose values will vary
at different locations. 
%\hf{To LW: please introduce the notation $o(), O(), o_\P(),$ and $O_\P()$.}

\subsection{Paper Organization}

The rest of this paper is organized as follows. In Section \ref{sec:general}%
, we introduce general methods for handling M-estimators of the particular
format. In particular, Section \ref{sec:U-process} gives a new U-process
bound in increasing dimensions, and Section \ref{sec:rank-general} studies
M-estimators of the form \eqref{eq:U}, whose loss functions are possibly
discontinuous. Section \ref{sec:application} applies the results in Section %
\ref{sec:general} to the four motivating rank estimators. Section \ref%
{sec:simulation} offers detailed finite-sample studies, illustrating the
impact of dimension on coverage probability and tuning parameter selection
in the asymptotic covariance estimation. Concluding remarks and possible
extensions are put in the end of the main text. 
%Section \ref{sec:discussion} discusses the optimality issue of the analysis, as well as provides connections to several other fields in studying M-estimators.
All proofs are relegated to {an appendix}. %a supplementary material.

\section{Asymptotic Theory for the M-estimator}

\label{sec:general}

Recall that $\bZ_{1},\bZ_{2},\ldots ,\bZ_{n}\in {\mathbb{R}}^{m_n}$ is a
random sample from $\mathbb{P}$, rendering an empirical measure $\mathbb{P}%
_{n}$. Let $\cF=\{f(\cdot ,\cdot ;\btheta):\btheta\in \Theta \subset {{%
\mathbb{R}}}^{p_n}\}$ be a VC-subgraph class of real-valued functions, with $%
\nu _{n}$ denoting the $VC$-dimension of $\cF$ (see Section 2.6.2 in \cite%
{wellner1996weak} for explicit definitions of VC-subgraph and VC-dimension
of a VC-subgraph class). %Though extensions to VC-hull,
%VC-major, and BUEI classes {{(see, for example, Chapter 9 in \cite
%{kosorok2007introduction} for explicit definitions)}} are straightforward,
In addition, we assume the function class $\cF$ to be uniformly bounded by
an absolute constant. The family of bounded VC-subgraph classes includes, as
subfamilies, those rank estimators proposed in \cite{han1987non}, \cite%
{cavanagh1998rank}, \cite{khan2007partial}, and \cite{abrevaya2011rank}, and
suffices for our purpose.

Without loss of generality, we assume that 
\begin{equation}
f(\bz_{1},\bz_{2};\btheta _{0})=0~~\mathrm{for~all~}(\bz_{1},\bz_{2})\in {%
\mathbb{R}}^{m_{n}}\otimes {\mathbb{R}}^{m_{n}},  \label{eq:zero}
\end{equation}%
which can always be arranged by working with $f(\bz_{1},\bz_{2};\btheta
)-f(\bz_{1},\bz_{2};\btheta _{0})$ throughout.

{{The derivation of asymptotic properties of $\hat{\btheta}_{n}$ can be
understood in two steps.}} First we show the asymptotic equivalence of $\hat{%
\btheta}_{n}$ and $\tilde{\btheta}_{n}$ by proving negligibility of $%
\mathbb{U}_{n}h(\cdot ,\cdot ;\btheta)$ and then establish asymptotic
properties of $\tilde{\btheta}_{n}$. Essential to the first step is an
increasing dimension analogue of maximal inequalities for degenerate
U-processes in finite dimensions. Because of increasing dimensions, we need
to calculate an exact order of the decaying rate of $\sup_{\btheta}|%
\mathbb{U}_{n}h(\cdot ,\cdot ;\btheta)|$ in a local neighborhood of $\bm%
\theta_{0}$, the proof of which requires a substantial amount of
modifications to the decoupling arguments in \cite{nolan1987u}. For the
second step, %we extend arguments used
%in the proofs of %\cite{portnoy1988asymptotic}, \cite{lahiri1994two}, and \cite{portnoy1988asymptotic}
{{we exploit Spokoiny's bracketing device technique (cf. Corollary 2.2 in 
\cite{spokoiny2012supp}) on M-estimators with differentiable objective
functions.}}

\subsection{A Maximal Inequality for Degenerate U-processes}

\label{sec:U-process}

For fixed dimensions, Sherman 
\citep{sherman1993limiting,
sherman1994maximal} proved a maximal inequality for degenerate U-processes
and used it to show that, when $\cF$ is $\mathbb{P}$-Donsker %
\citep{dudley1999uniform}, uniformly over a small neighborhood $\Theta _{0}$
surrounding $\btheta_{0}$, 
\begin{equation}
\sup_{\btheta\in \Theta _{0}}|\Gamma _{n}(\btheta)-\tilde{\Gamma}_{n}(\bm%
\theta)|=\sup_{\btheta\in \Theta _{0}}|\mathbb{U}_{n}h(\cdot ,\cdot ;\bm%
\theta)|=o_{\mathbb{P}}(1/n),  \label{eq:sherman}
\end{equation}%
which, combined with the fact that $g(\cdot )$ is usually a smooth function
by integration, is sufficient to guarantee that the stochastic
differentiability condition (cf. Theorem 3.2.16 in \cite%
{wellner1996weak}) holds. This suffices for establishing ASN in fixed
dimension. However, when we allow the dimension to increase with the sample
size, \eqref{eq:sherman} is no longer correct.

To account for the effect of increasing dimension, we establish a new
maximal inequality for degenerate U-processes in increasing dimensions.
Theorem \ref{thm:uniformloose} below works out an exact order of the rate of
convergence of $\sup_{\btheta\in \Theta _{0}}|\mathbb{U}_{n}h(\cdot ,\cdot
;\btheta)|$ as $\Theta _{0}$ shrinks to the true point $\btheta_{0}$ at
different rates $r_{n}\rightarrow 0$. It is formulated as two maximal
inequalities, corresponding to the {Glivenko-Cantalli and Donsker properties}%
, for a degenerate U-process.

\begin{theorem}\label{thm:uniformloose}
Suppose that  $\cF$ is uniformly bounded by an absolute constant, of VC-dimension $\nu_n$,  and $h(\cdot)$ is defined as in \eqref{eq:h}. Further recall that we have assumed $f(\cdot,\cdot;\btheta_0)$ satisfies \eqref{eq:zero}. If $\nu_n/n\rightarrow0$, then the following two claims hold.
  \begin{enumerate}
    \item[(i)]Let $r_n$ and $\epsilon_n$ be two sequences of nonnegative real numbers converging to zero. If
    \[
    \sup_{\btheta\in\overline{\mathcal{B}}(\btheta_0,r_n)}\E h^2(\cdot, \cdot;\btheta)\leq \epsilon_n,
    \]
    then there exists a sequence  of nonnegative real numbers $\delta_n$ (\hf{only} depending on $\epsilon_n,\nu,n$) converging to zero such that
        \[
       \P\bigg\{ \sup_{\btheta\in\overline{\mathcal{B}}(\btheta_0,r_n)}|\U_n h(\cdot,\cdot;\btheta)|\leq \delta_n \nu_n/n\bigg\}=1-o(1).
        \]
         \item[(ii)] Let { $r_n:=r(\nu_n,p_n,n)$} be a sequence of nonnegative real numbers converging to zero, and $\tilde\epsilon_n=\epsilon(\nu_n,p_n,n,r_n)$ be a sequence of nonnegative real numbers (\hf{only} depending on $\nu_n,p_n,n,r_n$) converging to zero. Denote $\tilde\eta_n=\eta(\nu_n,p_n,n,r_n) = \sqrt{\nu_n/n} \vee\tilde\epsilon_n$. Suppose
         \[
         \sup_{\btheta\in\overline{\mathcal{B}}(\btheta_0,r_n)}\E h^2(\cdot,\cdot;\btheta)\leq \tilde\epsilon_n.
         \]
    %where $\epsilon(v,m,n)$ is a function depending on $v,m,n$,
    %and let $\eta(v,m,n) := \sqrt{v/n} \vee\epsilon(v,m,n)$.
    We then have
        \begin{align}\label{eq:uniform-han}
       \E\sup_{\btheta\in\overline{\mathcal{B}}(\btheta_0,r_n)}|\U_nh(\cdot,\cdot;\btheta)| \leq \frac{C\log(1/\tilde\eta_n)\tilde\eta_n^{1/2}\nu_n}{n}
        \end{align}
        holds for \hf{all} sufficiently large $n$.
  \end{enumerate}
\end{theorem}

%\fbox{notice the changes of $\tilde\eta_n$ and $\tilde\epsilon_n$}

For deriving Theorem \ref{thm:uniformloose}, one might consider employing
the decoupling techniques as introduced in the proofs of the Main Corollary
in \cite{sherman1994maximal}, or Theorem 5.3.7 in \cite{de2012decoupling}.
However, since the considered U-process depends on an increasing number of
covariates, the constants in the moment inequalities therein ({e.g.,} $%
C(k,q) $ in \cite{sherman1994maximal}) are no longer finite and are
difficult to characterize in increasing dimensions. Instead, we resort to
Nolan and Pollard's original treatment of degenerate U-processes.

Specifically, denoting 
\begin{equation*}
\mathbb{S}_n f(\cdot,\cdot;\btheta)=n(n-1)\mathbb{U}_n f(\cdot,\cdot;\bm%
\theta),
\end{equation*}
a modification to Theorem 6 in \cite{nolan1987u} will give us 
\begin{align}
\mathbb{E}\Big\{\sup_{\btheta\in \overline{\mathcal{B}}(\bm%
\theta_0,r_{n})}|\mathbb{S}_{n}h(\cdot ,\cdot ;\btheta)/(n\nu )|\Big\}&
\leq CH\Big(\Big[\mathbb{E}\Big\{\sup_{\btheta\in \overline{\mathcal{B}}(%
\btheta_{0},r_{n})}\mathbb{U}_{2n}h^{2}(\cdot ,\cdot ;\btheta)\Big\}\Big]%
^{1/2}\Big)  \notag  \label{eq:uniform-han2} \\
& \leq CH\Big(\Big[\sup_{\btheta\in \overline{\mathcal{B}}(\bm%
\theta_0,r_{n})}\mathbb{E}h^{2}(\cdot ,\cdot ;\btheta)+\mathbb{E}\Big\{%
\sup_{\btheta\in \overline{\mathcal{B}}(\btheta_{0},r_{n})}|\mathbb{P}%
_{2n}h_{1}(\cdot ;\btheta)|\Big\}  \notag \\
& \hspace{1em}+\mathbb{E}\Big\{\sup_{\btheta\in \overline{\mathcal{B}}(\bm%
\theta_{0},r_{n})}|\mathbb{P}_{2n}h_{2}(\cdot ,\cdot ;\btheta)|\Big\}\Big]%
^{1/2}\Big).
\end{align}
Here $H(x):=x\{1+\log(1/x)\}$ for any $x\in(0,\infty)$, $\mathbb{U}_{2n}$
and $\mathbb{P}_{2n}$ have been introduced in \eqref{eq:h}, and $h_1(\bz,\bm%
\theta):=\mathbb{E} h^2(\bz,\cdot;\btheta)+\mathbb{E} h^2(\cdot,\bz;\bm%
\theta)-2\mathbb{E} h^2(\cdot,\cdot;\btheta)$ and $h_2(\bz_1,\bz_2;\bm%
\theta):=h^2(\bz_1,\bz_2;\btheta)-\mathbb{E} h^2(\bz_1,\cdot;\btheta) -%
\mathbb{E} h^2(\cdot,\bz_2;\btheta)+\mathbb{E} h^2(\cdot,\cdot;\btheta)$
are two functions generated from $h(\cdot,\cdot;\btheta)$. We have thus
explicitly transformed the analysis of a degenerate U-process to that of a
moment bound, and two empirical processes. Lastly, the bounds on the two
empirical processes could be derived using, for example, Theorem 9.3 in \cite%
{kosorok2007introduction}.

%the following modification to Theorem 6 in \cite{nolan1987u}
%plays a key role in our analysis: after some normalization, we have 
%\begin{align}
%\mathbb{E}\Big\{\sup_{\btheta \in \overline{\mathcal{B}}(\btheta
%_{0},r_{n})}|\mathbb{S}_{n}h(\cdot ,\cdot ;\btheta )/(n\nu )|\Big\}& \leq
%CH\Big(\Big[\mathbb{E}\Big\{\sup_{\btheta \in \overline{\mathcal{B}}(\bm%
%\theta _{0},r_{n})}\mathbb{U}_{2n}h^{2}(\cdot ,\cdot ;\btheta )\Big\}\Big]%
%^{1/2}\Big)  \notag  \label{eq:uniform-han2} \\
%& \leq CH\Big(\Big[\sup_{\btheta \in \overline{\mathcal{B}}(\btheta
%_{0},r_{n})}\mathbb{E}h^{2}(\cdot ,\cdot ;\btheta )+\mathbb{E}\Big\{\sup_{%
%\btheta \in \overline{\mathcal{B}}(\btheta_{0},r_{n})}|\mathbb{P}%
%_{2n}h_{1}(\cdot ;\btheta )|\Big\}  \notag \\
%& \hspace{1em}+\mathbb{E}\Big\{\sup_{\btheta \in \overline{\mathcal{B}}(\bm%
%\theta _{0},r_{n})}|\mathbb{P}_{2n}h_{2}(\cdot ,\cdot ;\btheta )|\Big\}%
%\Big]^{1/2}\Big),
%\end{align}%
%where the detailed definitions of $H(\cdot )$, $\mathbb{S}_{n}$, $\mathbb{U}%
%_{2n}$, $\mathbb{P}_{2n}$, $h_{1}(\cdot ;\btheta )$, and $h_{2}(\cdot
%,\cdot ;\btheta )$ are relegated to Sections \ref{sec: app-auxlemmas} and %
%\ref{sec:app-uniform} in the appendix. We thus explicitly transform the
%analysis of a degenerate U-process to that of a moment bound, and two
%empirical processes, whose analyses are more standard \citep{wellner1996weak}%
%.

%\textbf{FAN: Check expressions for }$h_{2}$\textbf{.}

%\textbf{Han: $h_{2}$ is a technical notation coming from Pollard's proof of
%degenerate U...}

%\textbf{FAN: It has two input vectors.}

\subsection{Main Results}

\label{sec:rank-general}

We are now ready to state the main results in this section. For analyzing
the statistical properties of the general M-estimator $\hat\btheta_{n}$,
three targets are in order: (i) consistency; (ii) rate of convergence; and
(iii) Bahadur-type bounds. Of note, our analysis is under the increasing
dimension triangular array setting where the true data generating process $%
\mathbb{P}$ is allowed to change with the sample size $n$.

We first establish consistency. This is via the following two assumptions. 
%
%\begin{assumption}
%\label{ass:identifiable}
%%Let $\lambda_{\min}(A)$ denote the minimum eigenvalue of a symmetric matrix $A$.
%$\Gamma(\btheta)$ is uniquely maximized at $\btheta_0$.
%\end{assumption}
%
%\begin{assumption}
%\label{ass:interior} For each fixed $p$, $\Theta$ is a compact subset of ${{%
%\mathbb{R}}}^{p}$, and there exists a small absolute constant $r_0>0$ such
%that $\mathcal{B}(\btheta_0,r_0)\subset \Theta$.
%\end{assumption}

%\hf{[To LW: for obvious reasons, we have to merge the above two assumptions to the following one and (slightly) revise the proof of Theorem 2.2 accordingly:

\begin{assumption}
\label{ass:new1} For each specified $p_n$, $\Theta$ is a compact subset of ${{%
\mathbb{R}}}^{p_n}$, and there exists an absolute constant $r_0>0$
such that $\mathcal{B}(\btheta_0,r_0)\subset \Theta$ and for any positive absolute constant $r<r_0$, there exists another absolute constant $\xi_0>0$ depending on $r$ such that 
\begin{equation}  \label{eqn: new cond}
\begin{aligned}
\Gamma(\btheta_0)-\max_{\Theta\setminus\mathcal{B}(\btheta_0,r)}\Gamma(\bm{%
\theta})\geq \xi_0. \end{aligned}
\end{equation}
\end{assumption}

%]
%}

%\hf{[To LW: this new condition is verifiable by Taylor expansion given Assumption 5(iii) if, say, $\Theta=\overline{\mathcal{B}}(\btheta_0,\bar{r}_0)$ for some absolute constant $\bar{r}_0$. We then intrinsically are analyzing a local maximizer $\hat\btheta_n$ since $\hat\btheta_n$ is an interior point in $\Theta$ by the new assumption and the conclusion in Theorem 2.2.]}

\begin{assumption}
\label{ass:continuous} $\Gamma(\btheta)$ is a continuous function at any
$\btheta\in\Theta$, and $f(\cdot,\cdot;\btheta)$ is almost
everywhere continuous at $\btheta_0$.
\end{assumption}

%Assumption \ref{ass:identifiable} is the regular identifiability condition,
%ensuring the studied problem to be well posited.
Assumption \ref{ass:new1} is the standard identifiability condition. Since $%
\Gamma(\btheta)$ as a function of $\btheta\in\mathbb{R}^{p_n}$ is also
to change with $n$, it is regulated by a constant $\xi_0$ to eliminate the
non-identifiable cases in large $n$. 
%requires $\btheta_{0}$ to be an interior point
%in $\Theta $, {\ and the condition \eqref{eqn: new cond} is verifiable by
%Taylor expansion given Assumption \ref{ass:taufunction} (ii) and (iii) below
%if, say, $\Theta =\overline{\mathcal{B}}(\btheta_{0},\bar{r}_{0})$ for
%some absolute constant $\bar{r}_{0}$. We then are intrinsically analyzing a
%local maximizer $\hat{\btheta }_{n}$, since $\hat{\btheta }_{n}$ is an
%interior point in $\Theta $ by the new assumption and the conclusion in
%Theorem \ref{thm: consistency}.} 
Assumption \ref{ass:continuous} enforces certain level of smoothness on $%
\Gamma $ and $f$. Both are regular, and in particular, verifiable for all
the considered examples of rank estimators using explicit expressions for $%
\Gamma$ and $f$ for these estimators. For example, for Han's MRC, Assumption %
\ref{ass:new1} can be established using Taylor expansion applied to $\Gamma (%
\btheta )=\Gamma ^{\mathsf{H}}(\btheta )=S^{\mathsf{H}}(\bbeta )-S^{%
\mathsf{H}}(\bbeta _{0})$ with $S^{\mathsf{H}}(\bbeta ):=\mathbb{E}\{%
\mathds{1}(Y_{1}>Y_{2})\mathds{1}(\bX_{1}^{\top }\bbeta >\bX_{2}^{\top }\bm%
\beta )\}.$

With Assumptions \ref{ass:new1} and \ref{ass:continuous}, we immediately
obtain the following theorem, establishing consistency for the studied
M-estimator $\hat{\btheta}_{n}$.

\begin{theorem}\label{thm: consistency}
Suppose that Assumptions~\ref{ass:new1}--\ref{ass:continuous} hold. If $\nu_n/n\rightarrow 0$, then $\norm{\hat\btheta_n-\btheta_0}\xrightarrow{\P} 0$.
\end{theorem}

{It is of interest to point out that consistency is established solely based
on an requirement of $\nu _{n}$ (which also intrinsically depends on $%
m_{n},p_{n}$), since the uniform consistency of $\Gamma _{n}$ to $\Gamma $
can be determined solely by the relation between $\nu _{n}$ and $n$.} 
%{FAN: but }$\nu $%
%{\ depends on }$m${\ and }$p!$
For the four examples of rank correlation estimators (1)-(4), $\nu _{n}=p_{n}
$ so consistency is ensured under Assumptions \ref{ass:new1} and \ref%
{ass:continuous} as long as the number of parameters $p_{n}$ increases at a
slower rate than the sample size $n$.

For establishing rates of convergence and Bahadur-type bounds, on the other
hand, {more assumptions} are needed. For each $\bz$ in ${\mathbb{R}}^{m_{n}}$
and for each $\btheta \in \Theta $, define 
\begin{equation*}
\tau (\bz;\btheta )=\mathbb{E}f(\bz,\cdot ;\btheta )+\mathbb{E}f(\cdot ,%
\bz;\btheta )~~\mathrm{and}~~\zeta (\bz;\btheta )=\tau (\bz;\btheta )-%
\mathbb{E}\tau (\cdot ;\btheta ).
\end{equation*}%
Here $\tau (\bz;\btheta )$ corresponds to $\tilde{\Gamma}_{n}(\btheta )$
in \eqref{eq:tildeGamma}, and is the key for establishing ASN of $\tilde\btheta_{n}$ in \eqref{eq:tildeBeta}. The following assumption regulates $%
\tau (\cdot ;\cdot )$.

\begin{assumption}
\label{ass:taufunction} For each $r\leq r_0$, the following conditions hold.

\begin{enumerate}
\item[(i)] For each $\bz$ in ${\mathbb{R}}^{m_n}$, all mixed second partial
derivatives of $\tau(\bz;\btheta)$ with respect to $\btheta$ exist on $\overline{\mathcal{B}}(%
\btheta_0,r)$.

\item[(ii)] There exist two positive absolute constants $c_{\min}, c_{\max}$
such that $0<c_{\min}\leq \lambda_{\min}(-\Vb)\leq \lambda_{\max}(-\Vb)\leq
c_{\max}$, where $2\Vb := \mathbb{E}\nabla_2\tau(\cdot;\btheta_0)$.

\item[(iii)] There exists a positive constant { $\rho(r)<\frac{%
c_{\min}}{11c_{\max}}\wedge cpr$ for some absolute constant $c>0$}, such
that $\norm{{\Ib}_p-\Vb^{-1/2}\Vb(\btheta)\Vb^{-1/2}}\leq \rho(r)$ for any $%
\btheta\in\overline{\mathcal{B}}(\btheta_0,r)$, where $2\Vb(%
\btheta) := \mathbb{E}\nabla_2\tau(\cdot;\btheta)$.

\item[(iv)] Assume $0<d_{\min}\leq\lambda_{\min}(\bDelta)\leq
\lambda_{\max}(\bDelta)\leq d_{\max}$, where $\bDelta:=%
\mathbb{E}\nabla_1\tau(\cdot;\btheta_0)\{\nabla_1\tau(\cdot;\btheta_0)\}^{\top}$ and $d_{\min}, d_{\max}$ are two positive absolute constants.

\item[(v)] There exist absolute constants $\nu_0>0$ and $\ell_0>0$ such
that, for any $\btheta\in\overline{\mathcal{B}}(\btheta_0,r)$, the
following holds:
\begin{equation*}
\sup_{\bm{\gamma}_1,\bm{\gamma}_2\in\mathbb{S}^{p_n-1}}\log\mathbb{E}%
\exp\left\{\lambda \bm{\gamma}_1 ^{\top} \nabla_2\zeta(\cdot;\btheta)%
\bm{\gamma}_2\right\}\leq \frac{\nu_0^2\lambda^2}{2}, \quad \mathrm{for~all~}%
|\lambda|\leq \ell_0.
\end{equation*}
\end{enumerate}
\end{assumption}

%Assumptions~\ref{ass:identifiable}  to \ref{ass:continuous} were used  to ensure the consistency of $\hat\btheta_n$.

Assumption \ref{ass:taufunction} is the key assumption in order to establish
Bahadur-type bounds for $\hat{\btheta }_{n}$, and is posed for the
M-estimation problem \eqref{eq:tildeGamma} of loss function $\tilde{\Gamma}%
_{n}(\btheta )$ corresponding to the function $\tau (\cdot )$. In the
following we discuss more about this assumption. In detail, Assumptions \ref%
{ass:taufunction}(i), (ii), and (iv) are regularity conditions to make sure
that the studied problem is well posited, a condition corresponding to the
local strong convexity condition in the high dimensional statistics
literature (cf. Section 2.4 in \cite{negahban2012unified}), and are
verifiable for different methods. Consider, for example, Han's MRC estimator 
$\hat\btheta _{n}^{\mathsf{H}}$ introduced in Section \ref{subsec: HanMRC}
for which $\tau =\tau ^{\mathsf{H}}$:%
\begin{equation*}
\tau ^{\mathsf{H}}(\bz;\btheta ):=\mathbb{E}f^{\mathsf{H}}(\bz,\cdot ;\bm%
\theta )+\mathbb{E}f^{\mathsf{H}}(\cdot ,\bz;\btheta ),
\end{equation*}%
where 
\begin{equation*}
f^{\mathsf{H}}(\bz_{1},\bz_{2};\btheta ):=\mathds{1}(y_{1}>y_{2})\{%
\mathds{1}(\bx_{1}^{\top }\bbeta >\bx_{2}^{\top }\bbeta )-\mathds{1}(\bx%
_{1}^{\top }\bbeta _{0}>\bx_{2}^{\top }\bbeta _{0})\}.
\end{equation*}%
{Assumptions \ref{ass:taufunction}(i), (ii), and (iv) then are immediately
ensured by Theorem 4 and subsequent discussions in \cite{sherman1993limiting}%
. Assumption \ref{ass:taufunction}(iii) requires that $\mathbb{E}\tau (\cdot
;\btheta )$ is sufficiently smooth in $\btheta $, for example, $\mathbb{E%
}\tau (\cdot ;\btheta )$ has continuous and bounded mixed partial
derivatives up to three. Assumption \ref{ass:taufunction}(v) requires the
existence of exponential moments of the errors. They correspond to the
\textquotedblleft local identifiability condition": Assumption ($\mathcal{L}%
_{0}$), and the \textquotedblleft exponential moment condition", Assumption (%
$ED_{2}$), in \cite{spokoiny2012parametric} and \cite{spokoiny2013bernstein}
separately. These conditions are often implied by subgaussian designs. }%
Particularly, in Theorem \ref{prop:H} in Section \ref{sec:Han}, we will verify Assumptions \ref%
{ass:taufunction}(iii) and (v) for $\tau ^{\mathsf{H}}$, i.e., Han's MRC
under primitive conditions.

%\begin{remark}
%Of note, Assumption \ref{ass:taufunction} resembles the corresponding assumptions in
%\cite{spokoiny2012parametric} and \cite{spokoiny2013bernstein}. \hf{Indeed, on one hand, one could verify that the two key assumptions, Assumption \ref{ass:taufunction}(iii) and (v), are the ``local identifiability condition", Assumption ($\mathcal{L}_0$), and the ``exponential moment condition", Assumption ($ED_2$), in \cite{spokoiny2012parametric} and \cite{spokoiny2013bernstein}. The remaining three conditions in Assumption \ref{ass:taufunction} have also been implicitly imposed in these papers. }
%\end{remark}

With the above assumptions, statistical properties of $\hat\btheta_n$
could then be established as follows.

\begin{theorem}\label{thm: general consrate}
If $(\nu_n\vee p_n)/n\rightarrow 0$ and Assumptions~\ref{ass:new1}--\ref{ass:taufunction} hold, we have
 \[
\norm{\hat\btheta_n-\btheta_0}^2 =O_\P\Big(\frac{\nu_n\vee p_n}{n}\Big).
  \]
\end{theorem}

For the four examples of rank correlation estimators, $\nu _{n}=p_{n}$ so
Theorem \ref{thm: general consrate} leads to the minimax optimal rate $%
\left( p_{n}/n\right) ^{1/2}$ under the condition: $p_{n}/n\rightarrow 0$.
However, Theorem \ref{thm:generalASN} below implies that much stronger
requirements on $p_{n}$ are needed to establish Bahadur-type bounds, see
Corollaries \ref{cor:H}-\ref{cor:A} for details.

\begin{theorem}\label{thm:generalASN}
 Suppose Assumptions~\ref{ass:new1}--\ref{ass:taufunction} hold,  and there exists a constant $\epsilon_n=\epsilon(\nu_n,p_n,n)$ depending on $\nu_n, p_n,n$  such that, for any $c>0$,
 \[
\sup_{\btheta\in\overline{\mathcal{B}}\{\btheta_0,c\sqrt{(\nu_n\vee p_n)/n}\}} \E  h^2(\cdot,\cdot;\btheta)\leq \tilde C\epsilon_n,
 \]
where $\tilde C$ only depends on $c$.  %Denote $\iota(v,m,n):=v^{3/2}m /n^{1/2}\vee \log(1/\eta(v,m,n))\eta^{1/2}(v,m,n)v$.
Then, the following two statements hold.
  \begin{enumerate}
    \item[(i)] Denote $\eta_n=\eta(\nu_n,p_n,n) =\sqrt{\nu_n/n} \vee\epsilon_n$. If $\eta_n=o(1)$ and $\{(\nu_n\vee p_n)^{5/2} /n^{3/2}\}\vee \{\log(1/\eta_n)\eta_n^{1/2}\nu_n/n\}=o(1)$, we have
\[
 \big\|\hat\btheta_n-\btheta_0+\Vb^{-1}\P_n \nabla_1 \tau(\cdot;\btheta_0)\big\|^2 = O_\P\Big\{\frac{(\nu_n\vee p_n)^{5/2}}{n^{3/2}}+ \frac{\log(1/\eta_n)\eta_n^{1/2}\nu_n}{n}\Big\}.
\]
  \item[(ii)]
If we further have $\{(\nu_n\vee p_n)^{5/2} /n^{1/2}\}\vee \{\log(1/\eta_n)\eta_n^{1/2}\nu_n\}=o(1)$, then for any $\bgamma\in\R^{p_n}$,
  \[
  \sqrt{n}\bgamma^\top(\hat\btheta_n-\btheta_0) / (\bgamma^\top \Vb^{-1}\Deltab \Vb^{-1}\bgamma)^{1/2}\Rightarrow N(0,1).
  \]
  \end{enumerate}
\end{theorem}

\begin{remark}
In the analysis, $p_n$ and $\nu_n$ characterize the behavior of the smoothed estimator $\tilde\btheta_n$ and the degenerate U-process $\{\mathbb{U}_{n}h(\cdot ,\cdot ;\btheta);\btheta\in\overline{\mathcal{B}}(\btheta_0,r_n) \}$ separately. On the other hand, throughout the above three theorems, the dimension of data points, $m_n$, is not present. Instead, the impact of $m_n$ on estimation and inference has been characterized by $p_n$ and $\nu_n$, both of which are usually of an order equal to or even greater than $m_n$. It is also noteworthy to point out that our analysis does allow an arbitrary subset of $(m_n,p_n,\nu_n)$ to be fixed, and the theory will directly proceed. In particular, when $m_n,p_n,\nu_n$ are all invariant with regard to $n$, we derived the conventional Bahadur representation for the studied class of M-estimators under the low-dimensional setting, which is a stronger result than asymptotic normality.
\end{remark}

%{FAN: the space for }$\theta ${\ is missing in the above remark.}

%\fbox{notice the change on $\epsilon_n$ and $\eta_n$.}

%{FAN: Some discussions may be needed here. What happens to the
%results when m and p are fixed? Do we allow the case of increasing m but
%fixed p? What is the role of m?\ }

We conclude this section with a brief discussion on consistent estimation of
the asymptotic covariance matrix in Theorem \ref{thm:generalASN}. For this,
we are focused on the covariance estimator of a numerical derivative form,
used in \cite{pakes1989simulation}, \cite{sherman1993limiting}, and \cite%
{khan2007partial}. 
%, the proposed estimator is of the numerical derivative form. }

First, for each $\bz$ in ${{\mathbb{R}}}^{m_n}$ and for each $\btheta$ in $%
\Theta $, define 
\begin{equation*}
\begin{aligned} \tau_n(\bz;\btheta)=\mathbb{P}_n f(\bz,\cdot;\btheta) +
\mathbb{P}_n f(\cdot,\bz;\btheta). \end{aligned}
\end{equation*}%
Then, we define the numerical derivative of $\tau _{n}(\bz;\btheta)$ as
follows: 
\begin{equation*}
\begin{aligned} p_{ni}(\bz;\btheta) =
\varepsilon_n^{-1}\{\tau_n(\bz;\btheta+\varepsilon_n \bu_i) -
\tau_n(\bz;\btheta)\}, \end{aligned}
\end{equation*}%
where $\varepsilon _{n}$ denotes a sequence of real numbers converging to
zero, and $\bu_{i}$ denotes the unit vector in ${{\mathbb{R}}}^{p_n}$ with
the $i$th component equal to one. Finally, we define the estimator of the
matrix $\bDelta$ as $\hat{\bDelta}=(\hat{\delta}_{ij})$ with 
\begin{equation*}
\begin{aligned} \hat\delta_{ij} :=
\mathbb{P}_n\{p_{ni}(\cdot;\hat\btheta_n)p_{nj}(\cdot;\hat\btheta_n)\}.
\end{aligned}
\end{equation*}

To estimate the matrix $\Vb$, we define the following function: 
\begin{equation*}
\begin{aligned} p_{nij}(\bz;\btheta) =
\varepsilon_n^{-2}\{\tau_n(\bz;\btheta+\varepsilon_n(\bu_i+\bu_j))-
\tau_n(\bz;\btheta+\varepsilon_n\bu_i)-\tau_n(\bz;\btheta+\varepsilon_n%
\bu_j)+\tau_n(\bz;\btheta)\}. \end{aligned}
\end{equation*}
Then, we define the estimator of the matrix $\Vb$ as $\hat\Vb=(\hat v_{ij})$
with 
\begin{equation*}
\begin{aligned} \hat v_{ij} := \frac{1}{2}\mathbb{P}_n
p_{nij}(\cdot;\hat\btheta_n). \end{aligned}
\end{equation*}

Let $\tilde{\cF}=\{f(\bz,\cdot ;\btheta)+f(\cdot ,\bz;\btheta):\bz\in {{%
\mathbb{R}}}^{m},\btheta\in \Theta \}$, and let $\tilde{\nu}_n$ denote the
VC-dimension of $\tilde{\cF}$. The following theorem establishes the
consistency of the covariance estimator. 
\begin{theorem}\label{thm: consistency cov}
  Suppose that Assumptions~\ref{ass:new1}--\ref{ass:taufunction} hold and $(\tilde \nu_n \vee \nu_n\vee p_n)^{5/2} /n^{1/2} = o(1)$. If the sequence $\varepsilon_n$ satisfies: $\varepsilon_n\sqrt{p_n}=o(1) $ and $\varepsilon_n^{-2}(\tilde \nu_n \vee \nu_n \vee p_n) /\sqrt{n}= o(1)$, then
  \begin{equation*}
    \begin{aligned}
      \norm{\hat\Vb^{-1}\hat\Deltab\hat\Vb^{-1} - \Vb^{-1}\Deltab\Vb^{-1}}\xrightarrow{\P} 0 .
    \end{aligned}
  \end{equation*}
\end{theorem}
%\hf{To LW: $\hat\Vb^{-1}\hat\Deltab\hat\Vb^{-1}\xrightarrow{\P} \Vb^{-1}\Deltab\Vb$ is an ill-defined concept. For inference validity, we need to prove $\norm{\hat\Vb^{-1}\hat\Deltab\hat\Vb^{-1}- \Vb^{-1}\Deltab\Vb}=o_\P(1)$, since only if it holds we have all linear contrasts can be correctly inferred. Here $\norm{\cdot}$ represents the matrix spectral norm. However, this property could be easily established by working out the exact order of matrix point-wise convergence rate. For example, if $|\hat\Delta_{ij}-\Delta_{ij}|=O_\P(1/\sqrt{n})$, then $\norm{\hat\Delta-\Delta}=O_\P(p/\sqrt{n})$. }

The increasing dimension set-up reveals that for consistent
variance-covariance matrix estimation, the step size in computing the
numerical derivative should depend not only on the sample size but also on
the dimensions $m_{n}$ and $p_{n}$.

\section{Asymptotic Properties of Rank Estimators}

\label{sec:application}

This section studies the four examples introduced in \hyperref[sec:intro]{%
Introduction}. In the sequel, the data points are understood to be
independent and identically drawn from the considered model. {{Of note,
throughout the following four examples, when the studied model is fixed, our
result renders the conventional Bahadur representation for the corresponding
estimator in fixed dimensions (see, for example, \cite{subbotin2008essays}
for such a bound in fixed dimensions). Hence, we recover the
asymptotic-normality-type theory in the corresponding paper, but under a
stronger moment condition in order to take the impact of increasing
dimension into consideration. In addition, it is worthwhile to point out
that, for all studied methods, the dimension of the data points $m_{n}$ and
the VC dimensions $\nu _{n}$ and $\tilde{\nu}_{n}$ of the studied function
classes are all of the same order as $p_{n}$, the number of parameters to be
estimated. Accordingly, in the following, we can use $p_{n}$ to solely
characterize the impact of dimension on inference.}}

%{FAN: I think for all these examples, we should discuss what happens
%when p is fixed and compare with existing results. Also give the expression
%for }$\nu ${\ so it's clear how the general results reduce to
%specific results in each example and what happens when p is fixed. I can't
%work on the contents of remarks and theorems directly in Scientific
%Workplace, so I'll leave this to you.}

%{Another question is: Would it involve much more work to add a
%consistent estimator of the variance-covariance matrix to the general case
%and the examples in this section to facilitate inference? I guess we could
%use the same ones developed for the fixed dimension case but of course we
%need to prove consistency. Would bootstrap work?}

\subsection{Han's Maximum Rank Correlation Estimator}\label{sec:Han}

This section studies the generalized regression model \eqref{GRM} and Han's
MRC estimator, as have been introduced in Section \ref{subsec: HanMRC}. Let $%
\mathcal{B}$ be a subset of $\{\bbeta \in {{\mathbb{R}}}^{p_{n}+1}:\beta
_{1}=1\}$. For any $\bbeta \in \mathcal{B}$, let $\bbeta =(1,\btheta
^{\top })^{\top }$, where $\btheta \in \Theta ^{\mathsf{H}}\subset {{%
\mathbb{R}}}^{p_{n}}$. For any vector $\bz=(y,\bx^{\top })^{\top }$, we
define $\zeta ^{\mathsf{H}}(\bz;\btheta )=\tau ^{\mathsf{H}}(\bz;\btheta
)-\mathbb{E}\tau ^{\mathsf{H}}(\cdot ;\btheta ),$ 
\begin{equation*}
~\bDelta^{\mathsf{H}}=\mathbb{E}\nabla _{1}\tau ^{\mathsf{H}}(\cdot ;\bm%
\theta _{0})\{\nabla _{1}\tau ^{\mathsf{H}}(\cdot ;\btheta _{0})\}^{\top
},~~\mathrm{and}~~2\Vb^{\mathsf{H}}=\mathbb{E}\nabla _{2}\tau ^{\mathsf{H}%
}(\cdot ;\btheta _{0}).
\end{equation*}%
Write $\Gamma _{n}^{\mathsf{H}}(\btheta )=S_{n}^{\mathsf{H}}(\bbeta
)-S_{n}^{\mathsf{H}}(\bbeta _{0})$ with 
\begin{equation*}
S_{n}^{\mathsf{H}}(\bbeta):=\frac{1}{n(n-1)}\sum_{i\neq j}\mathds{1}%
(Y_{i}>Y_{j})\mathds{1}(\bX_{i}^{\top }\bbeta >\bX_{j}^{\top }\bbeta ). 
\end{equation*}
Thus, Han's MRC estimator of $\btheta _{0}$, $\hat\btheta _{n}^{\mathsf{H%
}}$, can be expressed as 
\begin{equation*}
\hat\btheta _{n}^{\mathsf{H}}=\argmax_{\btheta \in \Theta ^{\mathsf{H}%
}}\Gamma _{n}^{\mathsf{H}}(\btheta ).
\end{equation*}

To conduct inference on $\btheta_{0}$ {based on $\hat\btheta_{n}^{%
\mathsf{H}}$, we further define 
\begin{align*}
& \tau _{n}^{\mathsf{H}}(\bz;\btheta )=\mathbb{P}_{n}f^{\mathsf{H}}(\bz%
,\cdot ;\btheta )+\mathbb{P}_{n}f^{\mathsf{H}}(\cdot ,\bz;\btheta
),~~p_{ni}^{\mathsf{H}}(\bz;\btheta )=\varepsilon _{n}^{-1}\{\tau _{n}^{%
\mathsf{H}}(\bz;\btheta +\varepsilon _{n}\bu_{i})-\tau _{n}^{\mathsf{H}}(%
\bz;\btheta )\},~~\mathrm{and}~~ \\
& p_{nij}^{\mathsf{H}}(\bz;\btheta )=\varepsilon _{n}^{-2}\{\tau _{n}^{%
\mathsf{H}}(\bz;\btheta +\varepsilon _{n}(\bu_{i}+\bu_{j}))-\tau _{n}^{%
\mathsf{H}}(\bz;\btheta +\varepsilon _{n}\bu_{i})-\tau _{n}^{\mathsf{H}}(%
\bz;\btheta +\varepsilon _{n}\bu_{j})+\tau _{n}^{\mathsf{H}}(\bz;\btheta
)\}.
\end{align*}%
Then, we define the estimator of the matrix $\bDelta^{\mathsf{H}}$ as $\hat{%
\bDelta}^{\mathsf{H}}=(\hat{\delta}_{ij}^{\mathsf{H}})$ and the estimator of
the matrix $\Vb^{\mathsf{H}}$ as $\hat{\Vb}^{\mathsf{H}}=(\hat{v}_{ij}^{%
\mathsf{H}})$, where 
\begin{equation*}
\begin{aligned} \hat\delta_{ij}^\sfH =
\mathbb{P}_n\{p_{ni}^\sfH(\cdot;\hat\btheta_n^\sfH)p_{nj}^\sfH(\cdot;\hat%
\btheta_n^\sfH)\} ~~\mathrm{and}~~ \hat v_{ij}^\sfH =
\frac{1}{2}\mathbb{P}_n p_{nij}^\sfH(\cdot;\hat\btheta_n^\sfH). \end{aligned}
\end{equation*}%
}

Let $\bX=(X_{1},\tilde{\bX}^{\top })^{\top }$, where $\tilde{\bX}$ denotes
the last $p$ components in $\bX$. Assume the following assumption holds 
\begin{assumption}
\label{ass:gen to han} Assume

\begin{enumerate}
\item[(i)] Assumption \ref{ass:new1} holds for $\Theta^\mathsf{H}$ and $\Gamma^\mathsf{H}(\btheta)$.

\item[(ii)] The random variables $\bX$ and $\epsilon$ are independent.

\item[(iii)]  Assume $X_1$ has an everywhere positive Lebesgue density, conditional
on $\tilde{\bX}$.

\item[(iv)] Assumption \ref{ass:taufunction} holds for $\tau^\mathsf{H}(\bz;\btheta)$ and $\zeta^\mathsf{H}(\bz;\btheta)$.
\end{enumerate}
\end{assumption}

\begin{assumption}
\label{ass:inftynormfinite} For some absolute constant $C>0$, $%
\sup_{i=2,\cdots,p+1}\mathbb{E}|X_i|^2\leq C$.
\end{assumption}

\begin{assumption}
\label{ass:conditionaldensityfinite} Let $f_0(\cdot\mid \tilde{\bx})$ denote the conditional density function
of $\bX^\top\bbeta_0$ given $\tilde{\bX}=\tilde{\bx}$. Assume $f_0(\cdot\mid \tilde{\bx})\leq C_0$
for any $\tilde{\bx}$ in the support of $\tilde{\bX}$, where $C_0>0$ is an
absolute constant.
\end{assumption}

%Assumption~\ref{ass:conditionaldensityfinite} can be easily satisfied. For example, if $\bX$ follows multivariate gaussian distribution, then this assumption automatically holds.
%
% If $p/n\rightarrow 0$, Proposition \ref{prop: fastrate H} implies that for any fixed $c>0$, $
%   \sup_{\btheta\in\overline{\mathcal{B}}(\btheta_0,c\sqrt{p/n})}\E \{h^\sfH(\cdot,\cdot;\btheta)\}^2\lesssim p/\sqrt{n}$ under Assumptions \ref{ass:inftynormfinite} and \ref{ass:conditionaldensityfinite}. Connecting this with the general method in Section \ref{sec:general}, we see that $\nu\sim p$, $m=p-1$, $\epsilon(\nu,m,n)\sim p/\sqrt{n}$, $\eta(\nu,m,n)\sim p/\sqrt{n}$. Then
We then have the following corollary. 
\begin{corollary}\label{cor:H} We have
\begin{enumerate}
  \item[(i)]   Under Assumption~\ref{ass:gen to han}(i)--(iii), if $p_n/n=o(1)$, then $\norm{\hat\theta^\sfH_n-\theta_0}\xrightarrow{\P} 0$.
  \item[(ii)]  Under Assumption~\ref{ass:gen to han},
 if $p_n/n=o(1)$, then
   \[
 \norm{\hat\btheta^\sfH_n-\btheta_0}^2 =O_\P(p_n/n).
  \]
 \item[(iii)] Under Assumptions~\ref{ass:gen to han}--\ref{ass:conditionaldensityfinite}, if $p_n^2/n=o(1)$ and $\log(n/p_n^2)p_n^{3/2}/n^{5/4}=o(1)$, we have
 \begin{align}\label{eq:kiefer-han}
 \norm{\hat\btheta_n^\sfH-\btheta_0+(\Vb^\sfH)^{-1}\P_n \nabla_1 \tau^\sfH(\cdot;\btheta_0)}^2 = O_\P\big\{\log(n/p_n^2)p_n^{3/2}/n^{5/4}\big\}.
 \end{align}
Furthermore, if $\log(n/p_n^2)p_n^{3/2}/n^{1/4}=o(1)$, then for any $\bgamma\in\R^{p_n}$,
  \[
  \sqrt{n}\bgamma^\top(\hat\btheta^\sfH_n-\btheta_0) / \{\bgamma^\top (\Vb^\sfH)^{-1}\Deltab^\sfH (\Vb^\sfH)^{-1}\bgamma\}^{1/2}\Rightarrow N(0,1).
  \]
  {\blue \item[(iv)] Under conditions in (iii), if we further have $\varepsilon_n\sqrt{p_n} = o(1)$ and $\varepsilon_n^{-2}p_n/\sqrt{n}=o(1)$, then
  \[
  \norm{(\hat\Vb^\sfH)^{-1}\hat\Deltab^\sfH(\hat\Vb^\sfH)^{-1}- (\Vb^\sfH)^{-1}\Deltab^\sfH(\Vb^\sfH)^{-1}}\xrightarrow{\P} 0.
  \]
  }
  \hf{In particular, we could choose $\epsilon_n\asymp(p_n/n)^{1/6}$, which will render a consistent covariance estimator under the same scaling condition as (iii).}
\end{enumerate}
\end{corollary}

In the following, we discuss more on the assumptions posed for Han's MRC
estimator. Since the estimator takes pairwise differences as input, without
loss of generality, the design is assumed to be zero-mean. First, Assumption %
\ref{ass:new1} can be established using Assumptions \ref{ass:taufunction}%
(ii), (iii), and Taylor expansion. Secondly, the conditions in Assumptions %
\ref{ass:continuous} and \ref{ass:taufunction}(i) are regular and can be
satisfied. Then, Theorem 4 and subsequent discussions in \cite%
{sherman1993limiting} ensure Assumptions \ref{ass:taufunction}(ii) and (iv)
hold. Lastly, we deal with Assumptions \ref{ass:taufunction}{\blue(iii) and
(v), which indeed} deserve more discussion. In the following, we give
sufficient conditions for guaranteeing Assumptions \ref{ass:taufunction}%
(iii) and (v) hold.

More notation is needed. Let $f_0(\cdot\mid \tilde \bx, y)$ denote the
conditional density function of $X_1$ given $\tilde \bX = \tilde \bx$ and $Y
= y$. Let $f_0(\cdot)$ denote the marginal density function of $\bX^\top\bm%
\beta_0$. Let 
\begin{align*}
\kappa^\mathsf{H}(y,t) = \mathbb{E}\{\mathds{1}(y>Y)-\mathds{1}(y<Y)\mid \bX%
^\top\bbeta_0=t\}, ~~\lambda^\mathsf{H}(y,t)=\kappa^\mathsf{H}(y,t)f_0(t),
\\
~~\mathrm{and}~~\lambda^\mathsf{H}_2(y,t)=\frac{\partial}{\partial t}\lambda^%
\mathsf{H}(y,t).
\end{align*}

We assume the following conditions on the design as well as the noisy hold.

\begin{condition}
\label{cond:H2} Suppose $\bX$ is multivariate subgaussian, i.e., there
exists an absolute constant $c^{\prime }>0$ such that $\sup_{\bm{\gamma}\in%
\mathbb{S}^{p}}\norm{\bgamma^\top\bX}_{\psi_2}\leq c^{\prime }$, where $%
\norm{\bgamma^\top\bX}_{\psi_2} := \sup_{q\geq 1} q^{-1/2}(\mathbb{E}|%
\bm{\gamma}^\top\bX|^q)^{1/q}$.
\end{condition}

\begin{condition}
\label{cond: cond density X1} (i) Suppose that $f_{0}(\cdot \mid \tilde{\bx}%
,y)$ has uniformly bounded derivatives up to order three, i.e., there exists
an absolute constant $C^{\prime \prime }>0$ such that $|f_{0}^{(j)}(\cdot
\mid \tilde{\bx},y)|\leq C^{\prime \prime }$ $(j=1,2,3)$ for any $\tilde{\bx}
$ and $y$ in the support of $\tilde{\bX}$ and $Y$, respectively; (ii) $%
\lim_{\left\vert t\right\vert \rightarrow \infty }f_{0}^{(2)}(t\mid \tilde{%
\bx},y)=0$ for any $\tilde{\bx}$ and $y$; (iii) Universally over the support
of $Y$ and any $\btheta\in \overline{\mathcal{B}}(\btheta_{0},r)$, $\int
|f_{0}^{(3)}(t-\tilde{\bx}^{\top }\btheta\mid s,\tilde{\bx})|G_{\tilde{\bX}%
\mid Y=s}(\diff\tilde{\bx})\leq c\{1\wedge c^{\prime}|t|^{ -(1+c^{\prime
\prime })}\}$ for some positive absolute constants $c,c^{\prime}, c^{\prime
\prime }$, where $G_{\tilde{\bX}\mid Y=s}(\cdot )$ represents the
probability measure of $\tilde{\bX}$ given $Y=s$.
\end{condition}

\begin{condition}
\label{cond:H1} Suppose that $\lambda_2^\mathsf{H}(y,t)$ is bounded, i.e.,
there exists an absolute constant $c^{\prime \prime }>0$ such that $%
|\lambda_2^\mathsf{H}(y,t)|\leq c^{\prime \prime }$ for any $y$ and $t$ in
the support of $Y$ and $\bX^\top\bbeta_0$, respectively.
\end{condition}

We then have the following theorem, which states that the above conditions
are sufficient ones to ensure Assumptions \ref{ass:taufunction}(iii) and (v)
hold.

\begin{theorem}\label{prop:H}
  Under Conditions \ref{cond:H2}--\ref{cond:H1}, Assumptions \ref{ass:taufunction}(iii) and (v) hold in this example.
\end{theorem}

%\hf{To LW: again, the matrix convergence has to be revised.}

%\begin{remark}\label{remark:kiefer}
%Equation \eqref{eq:kiefer-han} shows that the Bahadur-type bound for Han's rank estimator is of the order $(\log(n/p^2))^{1/2}p^{3/4}/n^{5/8}$, which decays to zero much slower than the desired $p/n$-rate that applies to most estimators \hf{with a smooth loss function}.% falling in the application regime of Theorem \ref{thm:han}.

%In particular, while fixing the dimension $p$, our observation is interestingly connected to Kiefer's observation on sample quantiles. In detail, Kiefer's study involves the sample median as an M-estimator:
%\[
%\hat M_n := \argmin_{\theta\in\reals}\P_n|X-\theta|,
%\]
%and shows that the Bahadur-type bound for $\hat M_n$ is exactly of the order $(\log\log n)^{4/3}/n^{3/4}$, which is not improvable. By comparing the Bahadur-type bounds for sample mean, sample median, and ours in fixed dimension, it is immediate that the order changes from $1/n$, $(\log\log n)^{4/3}/n^{3/4}$, to $(\log n)^{1/2}/n^{5/8}$. This intrinsically reflects the impact of the degree of smoothness of the  loss function, from twice differentiable, differentiable, to non-differentiable, on inference. On the other hand, the rate of convergence remains $O_\P(n^{-1/2})$, regardless of the degree of smoothness.
%\end{remark}

%The above remark further applies to Cavanagh and Sherman's rank estimator
%and Khan and Tamer's estimator for duration models.

\subsection{Cavanagh and Sherman's Rank Estimator}

\label{subsec: example C}

In contrast to Han's original proposal, \cite{cavanagh1998rank} proposed
estimating $\bbeta_{0}$ in \eqref{GRM} using 
\begin{equation*}
\hat{\bbeta}_{n}^{\mathsf{C}}=\argmax_{\bbeta:\beta _{1}=1}S_{n}^{%
\mathsf{C}}(\bbeta ),
\end{equation*}%
where 
\begin{equation*}
S_{n}^{\mathsf{C}}(\bbeta):=\frac{1}{n(n-1)}\sum_{i\neq j}M(Y_{i})%
\mathds{1}(\bX_{i}^{\top }\bbeta>\bX_{j}^{\top }\bbeta)
\end{equation*}%
and one candidate function for $M(y)$ is 
\begin{equation*}
M(y)=a\mathds{1}(y<a)+y\mathds{1}(a\leq y\leq b)+b\mathds{1}(y>b).
\end{equation*}%
Here $a$ and $b$ are two absolute constants, and hence $M(y)$ is a trimming
function for balancing the statistical efficiency and robustness to
outliers. Let $\bbeta_0=(1,\btheta_0^\top)^\top$, and we aim to estimate 
$\btheta_0$.

%{FAN: I think we should put this remark in the main text. }

%\begin{remark}
%Similar to Propostion \ref{prop:H}, all the other conditions will hold automatically given the corresponding conditions in \cite{cavanagh1998rank}, and we are focused on sufficient conditions to make Assumption \ref{ass:taufunction}(v) hold in this example. Let $\kappa^\sfC(y,t) = M(y)- \E\{M(Y)\mid \bX^\top\bbeta_0=t\}$, $\lambda^\sfC(y,t)=\kappa^\sfC(y,t)f_0(t)$, and $\lambda^\sfC_2(y,t)=\frac{\partial}{\partial t}\lambda^\sfC(y,t)$.
%\begin{proposition}\label{prop:C}
%   Under Conditions \ref{cond:H2} and assuming Condition \ref{cond:H1} holds for $\lambda^\sfC_2(y,t)$, we have Assumption \ref{ass:taufunction}(v) holds in this example.
%\end{proposition}
%\end{remark}

%If $p/n\rightarrow 0$, Proposition \ref{prop: fastrate C} implies that for any fixed $c>0$, $
%   \sup_{\btheta\in\overline{\mathcal{B}}(\btheta_0,c\sqrt{p/n})}\E \{h^\sfC(\cdot,\cdot;\btheta)\}^2\lesssim p/\sqrt{n}$ under Assumptions \ref{ass:inftynormfinite} and \ref{ass:conditionaldensityfinite}. Connecting this with the general method in Section \ref{sec:general}, we see that $\nu\sim p$, $m=p-1$, $\epsilon(\nu,m,n)\sim p/\sqrt{n}$, $\eta(\nu,m,n)\sim p/\sqrt{n}$. Then

We define the estimator $\hat\btheta_n^\mathsf{C}$ and other parameters
similarly as in Section \ref{subsec: HanMRC} and Section \ref{sec:Han}, with their explicit
definitions relegated to the appendix Section \ref{sec:app2}. Then we have
the following corollary. 
\begin{corollary}\label{cor:C} We have
\begin{enumerate}
  \item[(i)] Under Assumption~\ref{ass:gen to C}(i)--(iii) in the appendix Section \ref{sec:app2}, if $p_n/n=o(1)$, then $\norm{\hat\btheta^\sfC_n-\btheta_0}\xrightarrow{\P} 0$.
  \item[(ii)]  Suppose that Assumption~\ref{ass:gen to C} holds.
 If $p_n/n=o(1)$, then
  \[
\norm{\hat\btheta^\sfC_n-\btheta_0}^2 =O_\P(p_n/n).
  \]
  \item[(iii)] Suppose that Assumptions~\ref{ass:inftynormfinite}--\ref{ass:gen to C} hold.  If $p_n^2/n=o(1)$ and $\log(n/p_n^2)p_n^{3/2}/n^{5/4}=o(1)$, we have
 \[
 \norm{\hat\btheta_n^\sfC-\btheta_0+(\Vb^\sfC)^{-1}\P_n \nabla_1 \tau^\sfC(\cdot;\btheta_0)}^2 = O_\P\big\{\log(n/p_n^2)p_n^{3/2}/n^{5/4}\big\}.
 \]
 If further $\log(n/p_n^2)p_n^{3/2}/n^{1/4}=o(1)$, then for any $\bgamma\in\R^{p_n}$,
  \[
  \sqrt{n}\bgamma^\top(\hat\btheta^\sfC_n-\btheta_0) / \{\bgamma^\top (\Vb^\sfC)^{-1}\Deltab^\sfC (\Vb^\sfC)^{-1}\bgamma\}^{1/2}\Rightarrow N(0,1).
  \]
  {\blue \item[(iv)] Under conditions in (iii), if we further have $\varepsilon_n\sqrt{p_n} = o(1)$ and $\varepsilon_n^{-2}p_n/\sqrt{n}=o(1)$, then
  \[
  \norm{(\hat\Vb^\sfC)^{-1}\hat\Deltab^\sfC(\hat\Vb^\sfC)^{-1}- (\Vb^\sfC)^{-1}\Deltab^\sfC(\Vb^\sfC)^{-1}}\xrightarrow{\P} 0.
  \]
  }
    \hf{In particular, we could choose $\epsilon_n\asymp(p_n/n)^{1/6}$, which will render a consistent covariance estimator under the same scaling condition as (iii).}
\end{enumerate}
\end{corollary}

\subsection{Khan and Tamer's Rank Estimator for Duration Models}

\label{subsec: example K} Consider Khan and Tamer's setting %
\citep{khan2007partial}, where the data are subject to censoring and the
variable $Y$ is no longer always observed. Use $\xi$ to denote the random
censoring variable, which can be arbitrarily correlated with $\bX$. Let $R$
be a binary variable indicating whether $Y$ is uncensored or not. Let $V$
denote a scalar random variable with $V=Y$ for uncensored observations, and $%
V=\xi$ otherwise. Consider the following right censored transformation model %
\citep{khan2007partial}: 
\begin{align*}
T(V)&=\min(\bX^\top \bbeta_0+\epsilon,\xi), \\
R &=\mathds{1}(\bX^\top\bbeta_0+\epsilon\leq \xi),
\end{align*}
where $T(\cdot)$ is assumed to be strictly monotonic. The $(p_n+1)$%
-dimensional vector $\bbeta_0$ is unknown and is to be estimated.

\cite{khan2007partial} proposed estimating $\bbeta_0$ with $\hat\bbeta^%
\mathsf{K}_n =\argmax_{\bbeta:\beta _{1}=1}S^\mathsf{K}_n(\bm%
\beta)$, where 
\begin{equation*}
S^\mathsf{K}_n(\bbeta):=\frac{1}{n(n-1)}\sum_{i\neq j}R_i\mathds{1}(V_i<
V_j)\mathds{1}(\bX_i^\top\bbeta< \bX_j^\top\bbeta).
\end{equation*}
Let $\bbeta_0=(1,\btheta_0^\top)^\top$, and we consider estimation of $%
\btheta_0 $.

We define the estimator $\hat\btheta_n^\mathsf{K}$ and other parameters
similarly as in Section \ref{subsec: HanMRC} and Section \ref{sec:Han}, with their explicit
definitions relegated to the appendix Section \ref{sec:app3}. Then we have
the following corollary.

\begin{corollary}\label{cor:K} We have
\begin{enumerate}
  \item[(i)] Under Assumption~\ref{ass:gen to K}(i)--(iii) in the appendix Section \ref{sec:app3}, if $p_n/n=o(1)$, then $\norm{\hat\btheta^\sfK_n-\btheta_0}\xrightarrow{\P} 0$.
  \item[(ii)]  Under Assumption~\ref{ass:gen to K}, if $p_n/n=o(1)$, then
   \[
 \norm{\hat\btheta^\sfK_n-\btheta_0}^2 = O_\P(p_n/n).
  \]
  \item[(iii)] Suppose that Assumptions~\ref{ass:inftynormfinite}--\ref{ass:conditionaldensityfinite} and
\ref{ass:gen to K} hold.  If $p_n^2/n=o(1)$ and $\log(n/p_n^2)p_n^{3/2}/n^{5/4}=o(1)$, we have
 \[
 \norm{\hat\btheta_n^\sfK-\btheta_0+(\Vb^\sfK)^{-1}\P_n \nabla_1 \tau^\sfK(\cdot;\btheta_0)}^2 = O_\P\big\{\log(n/p_n^2)p_n^{3/2}/n^{5/4}\big\}.
 \]
 If further $\log(n/p_n^2)p_n^{3/2}/n^{1/4}=o(1)$, then for any $\bgamma\in\R^{p_n}$,
  \[
  \sqrt{n}\bgamma^\top(\hat\btheta^\sfK_n-\btheta_0) / \{\bgamma^\top (\Vb^\sfK)^{-1}\Deltab^\sfK (\Vb^\sfK)^{-1}\bgamma\}^{1/2}\Rightarrow N(0,1).
  \]
  {\blue \item[(iv)] Under conditions in (iii), if we further have $\varepsilon_n\sqrt{p_n} = o(1)$ and $\varepsilon_n^{-2}p_n/\sqrt{n}=o(1)$, then
  \[
  \norm{(\hat\Vb^\sfK)^{-1}\hat\Deltab^\sfK(\hat\Vb^\sfK)^{-1}- (\Vb^\sfK)^{-1}\Deltab^\sfK(\Vb^\sfK)^{-1}}\xrightarrow{\P} 0.
  \]
  }
    \hf{In particular, we could choose $\epsilon_n\asymp(p_n/n)^{1/6}$, which will render a consistent covariance estimator under the same scaling condition as (iii).}
\end{enumerate}
\end{corollary}

\subsection{Abrevaya and Shin's Rank Estimator for Partially Linear Index
Models}

\label{subsec: example A} Consider Abrevaya and Shin's partially linear
index model \citep{abrevaya2011rank}: 
\begin{equation*}
Y=T(\bX^{\top }\bbeta_{0}+\eta (W)+\epsilon ),
\end{equation*}%
where $\bX\in {\mathbb{R}}^{p_n+1}$, $W\in {\mathbb{R}}$, $T(\cdot )$ is a
non-degenerate monotone function, $\eta (\cdot )$ is a smooth function, and $%
\epsilon $ is a random noisy independent of $(\bX^{\top },W)^{\top }$. Our
primary interest is to estimate $\bbeta_{0}\in {{\mathbb{R}}}^{p_n+1}$.
For this, \cite{abrevaya2011rank} proposed using $\hat{\bbeta}_{n}^{%
\mathsf{A}}=\argmax_{\bbeta:\beta _{1}=1}S_{n}^{\mathsf{A}}(%
\bbeta)$, where 
\begin{equation*}
S_{n}^{\mathsf{A}}(\bbeta):=\frac{1}{n(n-1)}\sum_{i\neq j}\mathds{1}%
(Y_{i}>Y_{j})\mathds{1}(\bX_{i}^{\top }\bbeta>\bX_{j}^{\top }\bm%
\beta)K_{b}(W_{i}-W_{j}).
\end{equation*}%
Here $K_{b}(u):=b^{-1}K(u/b)$ is a function facilitating pairwise comparison %
\citep{honore1997pairwise}. It involves a kernel function $K(\cdot )$ and a
bandwidth parameter $b$. Let $\bbeta_0=(1,\btheta_0^\top)^\top$. Our aim
is to estimate $\btheta_0$.

With the estimator $\hat\btheta_n^\mathsf{A}$ and other parameters
similarly defined as in Section \ref{subsec: HanMRC} and Section \ref{sec:Han} and put in the appendix
Section \ref{sec:app4}, we have the following corollary. 
\begin{corollary}\label{cor:A} We have
\begin{enumerate}
  \item[(i)] Under Assumptions~\ref{ass:gen to A}(i)--(vii) in the appendix Section \ref{sec:app4}, if $p_n/n^{1-2\delta}=o(1)$, then $\norm{\hat\btheta^\sfA_n-\btheta_0}\xrightarrow{\P} 0$.
 \item[(ii)]  Under Assumptions
  \ref{ass:gen to A},
 if $p_n/n^{1-\delta}\rightarrow 0$, then
  \[
\norm{\hat\btheta^\sfA_n-\btheta_0}^2 =O_\P\Big(\frac{p_n}{n^{1-\delta}}\wedge\frac{p_n^{3/2}}{n} \Big).
  \]
  \item[(iii)] Under Assumptions \ref{ass:inftynormfinite} and
  \ref{ass:gen to A}-\ref{ass: cdfinite addw}, as $p_n^2/n^{1-\delta}=o(1)$ and $\log(n^{1-\delta}/p_n^2)p_n^{3/2}/n^{(5-5\delta)/4}=o(1)$, we have
 \[
 \norm{\hat\btheta_n^\sfA-\btheta_0+(\Vb^\sfA)^{-1}\P_n \nabla_1 \tau^\sfA(\cdot;\btheta_0)}^2 = O_\P\big\{n^{-\delta J}\vee \log(n^{1-\delta}/p_n^2)p_n^{3/2}/n^{(5-5\delta)/4}\big\}.
 \]
 If further $\log(n^{1-\delta}/p_n^2)p_n^{3/2}/n^{(1-5\delta)/4}=o(1)$, then for any $\bgamma\in\R^{p_n}$,
  \[
\sqrt{n}\bgamma^\top(\hat\btheta^\sfA_n-\btheta_0) / \{\bgamma^\top (\Vb^\sfA)^{-1}\Deltab^\sfA (\Vb^\sfA)^{-1}\bgamma\}^{1/2}\Rightarrow N(0,1).
  \]
  {\blue \item[(iv)] Under conditions in (iii), if we further have $\varepsilon_n\sqrt{p_n} = o(1)$ and  $\varepsilon_n^{-2}p_n/\sqrt{n^{1-2\delta}}=o(1)$,  then
  \[
  \norm{(\hat\Vb^\sfA)^{-1}\hat\Deltab^\sfA(\hat\Vb^\sfA)^{-1}- (\Vb^\sfA)^{-1}\Deltab^\sfA(\Vb^\sfA)^{-1}}\xrightarrow{\P} 0.
  \]
  }
    \hf{In particular, we could choose $\epsilon_n\asymp(p_n/n^{1-2\delta})^{1/6}$. This will render a consistent covariance estimator under the scaling condition $[p_n^4/n^{1-2\delta}\vee \{\log(n^{1-\delta}/p_n^2)\}^{4}p_n^{6}/n^{1-5\delta}] =o(1)$, which, at various cases, will be the same as the scaling condition in (iii).}
\end{enumerate}
\end{corollary}

%\hf{To LW: some discussion on why it is challenging to derive consistent variance estimators}

%{FAN: Some discussion when p is fixed and comparison with existing
%results. }

%\section{Sufficient Conditions in the Examples}

%\label{sec:sufficient}

\section{Simulation Results}

\label{sec:simulation}

This section presents results from a small simulation study to illustrate
two main implications of our theory. First for each fixed $n$, the normal
approximation to the finite sample distribution of the studied rank
correlation estimator will quickly become unreliable as $p_n$ grows,
suggesting that our theoretical bound is difficult to be improved in a
significant way. Secondly, in estimating the asymptotic covariance based on
the covariance estimator of the numerical derivative form, as $n$ fixed, the
tuning parameter that minimizes the Median Absolute Error (MAE) of the
estimator will increase with the dimension $p_n$, echoing our theoretical
observation. 
%Of note, the effective dimension (i.e., the number of parameters to be
%estimated) is $p-1$ since the first coefficient has been set to
%be 1 for identification purpose.

In the simulation study, we focus on Han's MRC estimator of the form %
\eqref{eqn:MRC} and the following binary choice model: 
\begin{equation*}
Y_{i}=\mathds{1}(\bX_{i}^{\top }\bbeta ^{\ast }+\epsilon _{i}\geq 0),\text{
}i=1,...,n,
\end{equation*}%
where $\bX_{i}\sim N(\zero,\bSigma)$ with $\bSigma_{jk}=0.5^{|j-k|}$, $%
\epsilon _{i}\sim N(0,1),$ and $\bbeta ^{\ast }=(2,4,6,\ldots
,2(p+1))^{\top }$ representing the true regression coefficient. For each $%
n=100,200,400$ and $p_n=1,2,3,4$, we simulate independent observations $%
\{Y_{i},\bX_{i}\}_{i=1}^{n} $ from the above model. Let $\bbeta_{0}^{\ast
}:=\bbeta ^{\ast }/\bbeta_{1}^{\ast }$ be the normalized regression
coefficient. We aim to estimate $\bbeta_{0}^{\ast }$ using Han's estimator 
$\hat\bbeta_{n}^{\mathsf{H}}$, which is implemented using the iterative
marginal optimization algorithm proposed by \cite{wang2007note}, with the
initial point chosen to be the truth.

Based on 1,000 independent replications and using two-sided normal
confidence interval, Tables \ref{tab:1}-\ref{tab:3} present the coverage
probability as the nominal one varies from 0.5 to 0.95 for three projections
of the same directions as $(1,1,\ldots ,1)^{\top }$, $(1,0,\ldots ,0)^{\top
} $, and $(1,2,\ldots ,p_n)^{\top }$. For calculating the confidence
intervals, we used the sample standard deviation of 1,000 replications. We
further plot the kernel estimates of the density functions of the normalized
three projected estimates against the density function of $N(0,1)$ in
Figures \ref{fig:1}-\ref{fig:3}. The normalization is based on the true mean
and the previous simulation-based standard deviation. In computing the
kernel density estimates, we used normal kernel function and the bandwidth
based on Silverman's rule-of-thumb.

Both the tables and figures reveal the same overall pattern that, for each
fixed $n$, as $p_n$ increases, the {coverage probability} will deviate more
from the nominal, and the kernel estimates of the density function of the
normalized estimator itself will deviate more from the standard normal. As
observed, the deviation from normal has become very severe even for very
small $p_n$. For example, for $p_n=2$, we need $n$ to be approximately 400
for achieving satisfactory coverage probability. This supports the
theoretical observations in Theorem \ref{thm:generalASN} and Corollary \ref%
{cor:H}(iii). We further conduct different types of normality tests
(Kolmogorov-Smirnov, Lilliefors, Jarque-Bera, Anderson-Darling,
Henze-Zirkler) on the derived projected estimates as well as the original
multi-dimensional estimates. They all reject the null hypothesis of
normality except when $p_n=1,n=400${.}

%{FAN: Sample sizes are different for coverage rates and variance
%estimation. Also for normality tests? n=500?}

We then move on to study the estimation accuracy of the asymptotic
covariance estimator discussed at the end of Section \ref{sec:rank-general}.
For this, we focus on the same setup as previously conducted. Table \ref%
{tab:4} presents the MAE of the asymptotic covariance estimator for the
projection direction $\{p_n^{-1/2},\ldots,p_n^{-1/2}\}^\top$. There, it
could be observed that, for each fixed $n$, the tuning parameter that
attains the smallest MAE will in general become larger as $p_n$ increases,
supporting our observation in Theorem \ref{thm: consistency cov} and
Corollary \ref{cor:H}(iv).

\section*{Concluding Remarks}

This paper provided a first study of asymptotic properties of a general
class of estimators defined as minimizers of possibly discontinuous
objective functions of U-process structure allowing for the dimension of the
parameter vector of interest to increase to infinity as the sample size $n$
increases to infinity. Members of this class include important rank
correlation estimators as detailed throughout this paper. Technically we
have established a maximal inequality for degenerate U-processes in
increasing dimensions which has played a critical role in deriving our
theoretical results. We have also applied our general theory to the four
motivating rank correlation estimators. Using Han's MRC estimator of the
form \eqref{eqn:MRC}, we have provided numerical support to our theoretical
findings that for a given sample size, the accuracy of the normal
approximation deteriorates quickly as the number of parameters $p_n$
increases and that for the variance estimation, the step size needs to be
adjusted with respect to $p_n$.

%In the following we briefly discuss two more lines of research related to ours.

This paper is focused on the setting that the parameter of interest itself
is of an increasing dimension and inference has to be drawn on it. On the
contrary, a growing literature studies the case that the parameter to be
inferred is of a fixed dimension, but allows for a dimension-increasing (but
still less than $n$) nuisance in the model. Substantial developments have
been made along this line. For example, \cite{cattaneo2016alternative} and 
\cite{cattaneo2017inference} studied inferring the fixed-dimension linear
component in a partially linear model, and \cite{lei2016asymptotics}
established asymptotic normality of margins of linear and robust regression
estimators in a simple linear model. Their set-up is fundamentally different
from ours due to the difference of goals.\footnote{%
We note that our set-up is also fundamentally different from works on "many
moment asymptotics" in GMM models such as \cite{han2006gmm}, \cite%
{newey2009generalized}, and \cite{caner2014near}, where the number of moment
conditions increases but the number of parameters in such models is fixed as
the sample size increases.}

%\textbf{This paper is the first attempt at studying asymptotic properties of
%rank correlation estimators when the number of parameters is allowed to
%increase with the sample size. Both the asymptotic theory and simulation
%results in this paper motivate the exploration of alternative asymptotics to
%better approximate the finite sample distribution of such estimators such as
%higher order asymptotics for fixed dimension or asymptotics for increasing
%dimension allowing for }$p$\textbf{\ to increase more slowly than required
%in this paper. Cattaneo, Jansson, and Newey (2018) and Bickel??? have taken
%the latter approach in linear models and obtained promising results. Because
%of the discontinuous nature of the sample objective function for rank
%correlation estimators, it is much more challenging to extend the results in
%Cattaneo, Jansson, and Newey (2018) and Bickel??? to our set-up.\footnote{%
%We note that our set-up is fundamentally different from works on "many
%moment asymptotics" in GMM models such as Han and Phillips (2006), Caner
%(2014), Caner, Han, and Lee (2015) where the number of moment conditions
%increases but the number of parameters in such models is fixed as the sample
%size increases.}}

We end this section with a brief discussion on further extensions. An
immediate extension is on studying \textquotedblleft penalized" rank
estimators in ultra high dimensional settings where the dimension could be
even larger than the sample size. For this much more challenging setting, to
the authors' knowledge, most literature is still focused on simple
structural statistical models (cf. \cite{zhang2014confidence}, \cite%
{van2014asymptotically}, \cite{lee2016exact}, and \cite{javanmard2015biasing}
among many others). A notable exception is the post-selection inference
framework proposed in \cite{belloni2014uniform} and \cite%
{belloni2015uniformly}, where a general set of regularization conditions has
been posed for inference validity of Z-estimation. The authors believe that,
combined with our local entropy analysis of the degenerate U-processes and
the empirical process techniques developed by Talagrand and Spokoiny and
specialized to rank estimators in this paper, the post-selection inference
framework will prove useful in extending the current study to ultra high
dimensional models. However, there are still many technical gaps, which we
believe are fundamental and related to some key challenges in high
dimensional probability in extending the scalar empirical processes to
vector and matrix ones if no further smoothing (cf. \cite{han2017provable})
is made. We will leave this for future research.

\section*{Acknowledgement}

We thank Dr. Hansheng Wang for providing the code to implement the iterative
marginal optimization algorithm, Mr. Shuo Jiang for helping conduct the
simulations, and seminar/conference participants at Emory University, Peking
University, and the 2019 Econometrics Workshop at Shanghai University of
Finance and Economics for helpful comments. The research
of Fang Han was supported in part by NSF grant DMS-1712536. We are also grateful to the
Associate Editor and two anonymous referees for instructive comments that
have greatly improved the paper. 

%\newpage

%\setcounter{lemma}{0}
%\renewcommand{\thelemma}{A\arabic{lemma}}

\appendix

\section{Appendix}

\subsection{Additional Notation}

For a vector $\bm{\alpha}\in{{\mathbb{R}}}^l$, we define $|\bm{\alpha}%
|=(|\alpha_1|,\ldots,|\alpha_l|)^\top$. For two sequences of real numbers $%
a_n$ and $b_n$, $a_n \lesssim b_n$ means that $a_n\leq b_n$ up to a
multiplicative constant. We use the symbol $a_n\sim b_n$ to denote that $a_n
\lesssim b_n$ and $b_n \lesssim a_n $. In this appendix we drop the
subscript $n$ in $m_n, \nu_n,p_n$.

\subsection{Notation and Assumptions in Section \protect\ref{sec:application}%
}

Throughout this section, let $\bX=(X_1,\tilde{\bX}^\top)^\top$, where $%
\tilde{\bX}$ denotes the last $p$ components in $\bX$.

%\subsubsection{Assumptions in Section \protect\ref{subsec: HanMRC}}

%\label{sec:app1}

%In contrast to consistency and $(p/n)^{1/2}$ rate of convergence, for
%obtaining the Bahadur-type bound for $\hat\theta^\mathsf{H}_n$, we need
%additional assumptions.

\subsubsection{Notation and Assumptions in Section \protect\ref{subsec:
example C}}

\label{sec:app2}

The following definitions are similar to those in Section \ref{subsec:
HanMRC}. We use $S^{\mathsf{C}}(\bbeta)$ to denote the expected value of $%
S_{n}^{\mathsf{C}}(\bbeta)$, and $S^{\mathsf{C}}(\bbeta)=\mathbb{E}%
\{M(Y_{1})\mathds{1}(\bX_{1}^{\top }\bbeta>\bX_{2}^{\top }\bbeta)\}$.
Let $\bz=(y,\bx^\top)^\top$. We define 
\begin{align*}
&f^\mathsf{C}(\bz_1,\bz_2;\btheta) = M(y_1) \{\mathds{1}(\bx_1^\top\bm%
\beta>\bx_2^\top\bbeta)- \mathds{1}(\bx_1^\top\bbeta_0>\bx_2^\top\bm%
\beta_0)\}, \\
& \tau^\mathsf{C}(\bz;\btheta)=\mathbb{E} f^\mathsf{C}(\bz,\cdot;\bm%
\theta) +\mathbb{E} f^\mathsf{C}(\cdot,\bz;\btheta), ~~ \zeta^\mathsf{C}(%
\bz; \btheta)=\tau^\mathsf{C}(\bz;\btheta)-\mathbb{E}\tau^\mathsf{C}%
(\cdot;\btheta), \\
& \bDelta^\mathsf{C} =\mathbb{E}\nabla_1\tau^\mathsf{C}(\cdot;\bm%
\theta_0)\{\nabla_1\tau^\mathsf{C}(\cdot;\btheta_0)\}^{\top}, ~\mathrm{and}%
~ 2\Vb^\mathsf{C} = \mathbb{E}\nabla_2\tau^\mathsf{C}(\cdot;\btheta_0).
\end{align*}
Write $\Gamma^\mathsf{C}(\btheta)$ for $S^\mathsf{C}(\bbeta)-S^\mathsf{C}%
(\bbeta_0)$ and $\Gamma^\mathsf{C}_n(\btheta)$ for $S^\mathsf{C}_n(\bm%
\beta)-S^\mathsf{C}_n(\bbeta_0)$. The estimator $\hat\btheta^\mathsf{C}%
_n $ is defined as 
\begin{equation*}
\hat\btheta^\mathsf{C}_n=\argmax_{\btheta\in\Theta^\mathsf{C}}\Gamma^%
\mathsf{C}_n(\btheta).
\end{equation*}

To conduct inference on $\btheta_{0}$ based on $\hat\btheta_{n}^{\mathsf{%
C}}$, we further define 
\begin{align*}
& \tau _{n}^{\mathsf{C}}(\bz;\btheta)=\mathbb{P}_{n}f^{\mathsf{C}}(\bz%
,\cdot ;\btheta)+\mathbb{P}_{n}f^{\mathsf{C}}(\cdot ,\bz;\btheta
),~~p_{ni}^{\mathsf{C}}(\bz;\btheta)=\varepsilon _{n}^{-1}\{\tau _{n}^{%
\mathsf{C}}(\bz;\btheta+\varepsilon _{n}\bu_{i})-\tau _{n}^{\mathsf{C}}(\bz%
;\btheta)\},~~\mathrm{and}~~ \\
& p_{nij}^{\mathsf{C}}(\bz;\btheta)=\varepsilon _{n}^{-2}\{\tau _{n}^{%
\mathsf{C}}(\bz;\btheta+\varepsilon _{n}(\bu_{i}+\bu_{j}))-\tau _{n}^{%
\mathsf{C}}(\bz;\btheta+\varepsilon _{n}\bu_{i})-\tau _{n}^{\mathsf{C}}(\bz%
;\btheta+\varepsilon _{n}\bu_{j})+\tau _{n}^{\mathsf{C}}(\bz;\btheta )\}.
\end{align*}%
Then, we define the estimator of the matrix $\bDelta^{\mathsf{C}}$ as $\hat{%
\bDelta}^{\mathsf{C}}=(\hat{\delta}_{ij}^{\mathsf{C}})$ and the estimator of
the matrix $\Vb^{\mathsf{C}}$ as $\hat{\Vb}^{\mathsf{C}}=(\hat{v}_{ij}^{%
\mathsf{C}})$, where 
\begin{equation*}
\begin{aligned} \hat\delta_{ij}^\sfC =
\mathbb{P}_n\{p_{ni}^\sfC(\cdot;\hat\btheta_n^\sfC)p_{nj}^\sfC(\cdot;\hat%
\btheta_n^\sfC)\}, ~~\mathrm{and}~~ \hat v_{ij}^\sfC =
\frac{1}{2}\mathbb{P}_n p_{nij}^\sfC(\cdot;\hat\btheta_n^\sfC). \end{aligned}
\end{equation*}

We then make the following assumptions.

\begin{assumption}
\label{ass:gen to C} Assume

\begin{enumerate}
\item[(i)] Assumption \ref{ass:new1} holds for $\Theta^\mathsf{C}$ and $%
\Gamma^\mathsf{C}(\btheta)$.

%\item[(ii)] Assumption \ref{ass:interior} holds for $\Theta^\mathsf{C}$.

\item[(ii)] The random variables $\bX$ and $\epsilon$ are independent, and $%
\mathbb{E}\{M(Y)\mid \bX\}$ depends on $\bX$ only through $\bX^\top\bbeta_0 $.

\item[(iii)] $X_1$ has an everywhere positive Lebesgue density, conditional
on $\tilde{\bX}$.

\item[(iv)] Assumption \ref{ass:taufunction} holds for $\tau^\mathsf{C}(\bz;%
\btheta)$ and $\zeta^\mathsf{C}(\bz;\btheta)$.
\end{enumerate}
\end{assumption}

\subsubsection{Notation and Assumptions in Section \protect\ref{subsec:
example K}}

\label{sec:app3}

The following definitions are similar to those in Section \ref{subsec:
HanMRC}. Let $S^\mathsf{K}(\bbeta)$ denote the expected value of $S^%
\mathsf{K}_n(\bbeta)$, and $S^\mathsf{K}(\bbeta)=\mathbb{E}\{R_1%
\mathds{1}(V_1< V_2) \mathds{1}(\bX_1^\top\bbeta< \bX_2^\top\bbeta)\}$.
Let $\bz=(r,v,\bx^\top)^\top$. We define 
\begin{align*}
&f^\mathsf{K}(\bz_1,\bz_2;\btheta) = r_1\mathds{1}(v_1<v_2) \{\mathds{1}(%
\bx_1^\top\bbeta<\bx_2^\top\bbeta)- \mathds{1}(\bx_1^\top\bbeta_0<\bx%
_2^\top\bbeta_0)\}, \\
& \tau^\mathsf{K}(\bz;\btheta)=\mathbb{E} f^\mathsf{K}(\bz,\cdot;\bm%
\theta) +\mathbb{E} f^\mathsf{K}(\cdot,\bz;\btheta),~~ \zeta^\mathsf{K}(\bz%
; \btheta)=\tau^\mathsf{K}(\bz;\btheta)-\mathbb{E}\tau^\mathsf{K}(\cdot;%
\btheta), \\
&\bDelta^\mathsf{K} =\mathbb{E}\nabla_1\tau^\mathsf{K}(\cdot;\bm%
\theta_0)\{\nabla_1\tau^\mathsf{K}(\cdot;\btheta_0)\}^{\top}, ~~\mathrm{and%
}~~ 2\Vb^\mathsf{K} = \mathbb{E}\nabla_2\tau^\mathsf{K}(\cdot;\btheta_0).
\end{align*}
Write $\Gamma^\mathsf{K}(\btheta)$ for $S^\mathsf{K}(\bbeta)-S^\mathsf{K}%
(\bbeta_0)$ and $\Gamma^\mathsf{K}_n(\btheta)$ for $S^\mathsf{K}_n(\bm%
\beta)-S^\mathsf{K}_n(\bbeta_0)$. The estimator $\hat\btheta^\mathsf{K}%
_n $ is defined as 
\begin{equation*}
\hat\btheta^\mathsf{K}_n=\argmax_{\btheta\in\Theta^\mathsf{K}}\Gamma^%
\mathsf{K}_n(\btheta).
\end{equation*}

To conduct inference on $\btheta_{0}$ based on $\hat\btheta_{n}^{\mathsf{%
K}}$, we further define 
\begin{align*}
& \tau _{n}^{\mathsf{K}}(\bz;\btheta)=\mathbb{P}_{n}f^{\mathsf{K}}(\bz%
,\cdot ;\btheta)+\mathbb{P}_{n}f^{\mathsf{K}}(\cdot ,\bz;\btheta
),~~p_{ni}^{\mathsf{K}}(\bz;\btheta)=\varepsilon _{n}^{-1}\{\tau _{n}^{%
\mathsf{K}}(\bz;\btheta+\varepsilon _{n}\bu_{i})-\tau _{n}^{\mathsf{K}}(\bz%
;\btheta)\},~~\mathrm{and}~~ \\
& p_{nij}^{\mathsf{K}}(\bz;\btheta)=\varepsilon _{n}^{-2}\{\tau _{n}^{%
\mathsf{K}}(\bz;\btheta+\varepsilon _{n}(\bu_{i}+\bu_{j}))-\tau _{n}^{%
\mathsf{K}}(\bz;\btheta+\varepsilon _{n}\bu_{i})-\tau _{n}^{\mathsf{K}}(\bz%
;\btheta+\varepsilon _{n}\bu_{j})+\tau _{n}^{\mathsf{K}}(\bz;\btheta )\}.
\end{align*}%
Then, we define the estimator of the matrix $\bDelta^{\mathsf{K}}$ as $\hat{%
\bDelta}^{\mathsf{K}}=(\hat{\delta}_{ij}^{\mathsf{K}})$ and the estimator of
the matrix $\Vb^{\mathsf{K}}$ as $\hat{\Vb}^{\mathsf{K}}=(\hat{v}_{ij}^{%
\mathsf{K}})$, where 
\begin{equation*}
\begin{aligned} \hat\delta_{ij}^\sfK =
\mathbb{P}_n\{p_{ni}^\sfK(\cdot;\hat\btheta_n^\sfK)p_{nj}^\sfK(\cdot;\hat%
\btheta_n^\sfK)\} ~~\mathrm{and}~~ \hat v_{ij}^\sfK =
\frac{1}{2}\mathbb{P}_n p_{nij}^\sfK(\cdot;\hat\btheta_n^\sfK). \end{aligned}
\end{equation*}

We then make the following assumptions.

\begin{assumption}
\label{ass:gen to K} Assume

\begin{enumerate}
\item[(i)] Assumption \ref{ass:new1} holds for $\Theta^\mathsf{K}$ and $%
\Gamma^\mathsf{K}(\btheta)$.

%\item[(ii)] Assumption \ref{ass:interior} holds for $\Theta^\mathsf{K}$.

\item[(ii)] The random variables $(\xi,\bX)$ and $\epsilon$ are independent,
and $\mathbb{E}(\xi\mid \bX)$ depends on $\bX$ only through $\bX^\top\bbeta_0$.

\item[(iii)] $X_1$ has an everywhere positive Lebesgue density, conditional
on $\tilde{\bX}$.

\item[(iv)] Assumption \ref{ass:taufunction} holds for $\tau^\mathsf{K}(\bz;%
\btheta)$ and $\zeta^\mathsf{K}(\bz;\btheta)$.
\end{enumerate}
\end{assumption}

\subsubsection{Notation and Assumptions in Section \protect\ref{subsec:
example A}}

\label{sec:app4}

Let $\phi(\cdot)$ denote the density of $W$. Let $\bz=(y,\bx^\top,w)^\top$.
We define 
\begin{align*}
&f^\mathsf{A}(\bz_1,\bz_2;\btheta) = \mathds{1}(y_1>y_2)\{\mathds{1}(\bx%
_1^\top\bbeta>\bx_2^\top\bbeta) -\mathds{1}(\bx_1^\top\bbeta_0>\bx%
_2^\top\bbeta_0)\}K\{(w_1-w_2)/b\}, \\
&m(\bz_1,\bz_2;\btheta)=\mathds{1}(y_1>y_2)\mathds{1}(\bx_1^\top\bbeta>%
\bx_2^\top\bbeta), \\
&\psi(w_1,w_2;\btheta)=\mathbb{E}\{m(\bZ_1,\bZ_2;\btheta)-m(\bZ_1,\bZ_2;%
\btheta_0)\mid W_1=w_1,W_2=w_2\}, \\
&\Gamma^\mathsf{A}(\btheta)= \mathbb{E}_W\{\psi(W,W;\btheta)\phi(W)\}, \\
&\tau^\mathsf{A}(\bz;\theta)=\mathbb{E}\{m(\bz,\bZ_2;\btheta)\mid
W_2=w\}\phi(w)+\mathbb{E}\{m(\bZ_1,\bz;\btheta)\mid W_1=w\}\phi(w), \\
&\zeta^\mathsf{A}(\bz; \btheta)=\tau^\mathsf{A}(\bz;\btheta)-\mathbb{E}%
\tau^\mathsf{A}(\cdot;\btheta),~~ \bDelta^\mathsf{A} =\mathbb{E}%
\nabla_1\tau^\mathsf{A}(\cdot;\btheta_0)\{\nabla_1\tau^\mathsf{A}(\cdot;\bm%
\theta_0)\}^{\top}, ~~\mathrm{and}~~ 2\Vb^\mathsf{A} = \mathbb{E}%
\nabla_2\tau^\mathsf{A}(\cdot;\btheta_0).
\end{align*}
Write $\Gamma_n^\mathsf{A}(\btheta)$ for $S_n^\mathsf{A}(\bbeta)-S_n^%
\mathsf{A}(\bbeta_0)$. The estimator $\hat\btheta^\mathsf{A}_n$ is
defined as 
\begin{equation*}
\hat\btheta^\mathsf{A}_n=\argmax_{\btheta\in\Theta^\mathsf{A}}\Gamma^%
\mathsf{A}_n(\btheta).
\end{equation*}

Note that $\mathbb{E}\Gamma _{n}^{\mathsf{A}}(\btheta)\neq \Gamma ^{%
\mathsf{A}}(\btheta)$. This is different from the general set-up in
Section \ref{sec:rank-general}. However, by Taylor expansion, we show that $%
\sup_{\btheta\in \Theta ^{\mathsf{A}}}\big|\mathbb{E}\Gamma _{n}^{\mathsf{A%
}}(\btheta)-\Gamma ^{\mathsf{A}}(\btheta)\big|$ is negligible under the
assumptions adopted in this section. Then, following the proof of the
general method, we can similarly establish the consistency and asymptotic
normality of $\hat{\btheta}_{n}^{\mathsf{A}}$. 
%{FAN: Elaborate on the modifications needed.} \hf{To LW: please address this comment.}

To conduct inference on $\btheta_{0}$ based on $\hat\btheta_{n}^{\mathsf{%
A}}$, we further define 
\begin{align*}
& \tau _{n}^{\mathsf{A}}(\bz;\btheta)=\mathbb{P}_{n}f^{\mathsf{A}}(\bz%
,\cdot ;\btheta)+\mathbb{P}_{n}f^{\mathsf{A}}(\cdot ,\bz;\btheta
),~~p_{ni}^{\mathsf{A}}(\bz;\btheta)=\varepsilon _{n}^{-1}\{\tau _{n}^{%
\mathsf{A}}(\bz;\btheta+\varepsilon _{n}\bu_{i})-\tau _{n}^{\mathsf{A}}(\bz%
;\btheta)\},~~\mathrm{and}~~ \\
& p_{nij}^{\mathsf{A}}(\bz;\btheta)=\varepsilon _{n}^{-2}\{\tau _{n}^{%
\mathsf{A}}(\bz;\btheta+\varepsilon _{n}(\bu_{i}+\bu_{j}))-\tau _{n}^{%
\mathsf{A}}(\bz;\btheta+\varepsilon _{n}\bu_{i})-\tau _{n}^{\mathsf{A}}(\bz%
;\btheta+\varepsilon _{n}\bu_{j})+\tau _{n}^{\mathsf{A}}(\bz;\btheta )\}.
\end{align*}%
Then, we define the estimator of the matrix $\bDelta^{\mathsf{A}}$ as $\hat{%
\bDelta}^{\mathsf{A}}=(\hat{\delta}_{ij}^{\mathsf{A}})$ and the estimator of
the matrix $\Vb^{\mathsf{A}}$ as $\hat{\Vb}^{\mathsf{A}}=(\hat{v}_{ij}^{%
\mathsf{A}})$, where 
\begin{equation*}
\begin{aligned} \hat\delta_{ij}^\sfA =
\mathbb{P}_n\{p_{ni}^\sfA(\cdot;\hat\btheta_n^\sfA)p_{nj}^\sfA(\cdot;\hat%
\btheta_n^\sfA)\}, ~~\mathrm{and}~~ \hat v_{ij}^\sfA =
\frac{1}{2}\mathbb{P}_n p_{nij}^\sfA(\cdot;\hat\btheta_n^\sfA). \end{aligned}
\end{equation*}

We make the following assumptions.

\begin{assumption}
\label{ass:gen to A} Assume

\begin{enumerate}
\item[(i)] Assumption \ref{ass:new1} holds for $\Theta^\mathsf{A}$ and $%
\Gamma^\mathsf{A}(\btheta)$;

%\item[(ii)] Assumption \ref{ass:interior} holds for $\Theta^\mathsf{A}$.

\item[(ii)] The random variables $(\bX,W)$ and $\epsilon$ are independent.

\item[(iii)] $X_1$ has an everywhere positive Lebesgue density, conditional
on $\tilde{\bX}$ and $W$.

\item[(iv)] $W$ is continuously distributed on a compact subset $\mathcal{W}$
of ${{\mathbb{R}}}$.

\item[(v)] The kernel function $K(\cdot)$ satisfies: (1) $K(\cdot)$ is twice
continuously differential with compact interval $[-C,C]\supseteq \mathcal{W}$%
; (2) $K(\cdot)$ is symmetric about 0 and integrates to 1; (3) for some
integer $J\geq 6$, $\int u^jK(u)\diff u=0$ with $j =1,\ldots,J-1$ and $\int
u^JK(u)\diff u$ is bounded.

\item[(vi)] The bandwidth $b$ is defined as $b=cn^{-\delta}$ for constants $%
c>0$ and $\frac{1}{J}<\delta<\frac{1}{5}$.

\item[(vii)] For any $w_2$, the $J$th derivative of $\psi(w_1,w_2;\btheta%
)\cdot\phi(w_1)$ with respect to $w_1$ is continuous and bounded for all $%
\btheta\in\Theta^\mathsf{A}$.

\item[(viii)] Assumption \ref{ass:taufunction} holds for $\tau^\mathsf{A}(\bz%
;\btheta)$ and $\zeta^\mathsf{A}(\bz;\btheta)$.
\end{enumerate}
\end{assumption}

\begin{assumption}
\label{ass: cdfinite addw} Let $f_{0}(\cdot\mid \tilde{\bx},w)$ denote the
conditional density function of $\bX^\top\bbeta_0$ given $(\tilde{\bX},W)=(%
\tilde{\bx}, w)$. Assume $f_{0}(\cdot\mid \tilde{\bx},w)\leq C_1$ for any $%
\tilde{\bx}$ and $w$ in the support of $\tilde{\bX}$ and $W$, respectively,
where $C_1$ is an absolute positive constant.
\end{assumption}

\subsection{Proofs in Section \protect\ref{sec:general}}

For each $\btheta\in\Theta$, define measures 
\begin{equation*}
\mathbb{S}_n f(\cdot,\cdot;\btheta)=n(n-1)\mathbb{U}_n f(\cdot,\cdot;\bm%
\theta)
\end{equation*}
and 
\begin{equation*}
\mathbb{T}_n f(\cdot,\cdot;\btheta)= \sum_{i\neq j}\{f(\bZ_{2i},\bZ_{2j};%
\newline
\btheta)+f(\bZ_{2i},\bZ_{2j-1};\btheta) +f(\bZ_{2i-1},\bZ_{2j};\bm%
\theta)+f(\bZ_{2i-1},\bZ_{2j-1};\btheta)\}.
\end{equation*}
To prove Theorems \ref{thm:uniformloose}--\ref{thm:generalASN} in Section %
\ref{sec:general}, we need several lemmas. For simplicity, we omit the
parameter $\btheta$ in each function $f(\cdot,\cdot;\btheta)\in\cF$ in
the lemmas. 
%For each $f\in\F$, define measures $\Sm_n f=n(n-1)\U_n f$ and $\T_n f= \sum_{i\neq j}\{f(\bZ_{2i},\bZ_{2j})+f(\bZ_{2i},\bZ_{2j-1}) +f(\bZ_{2i-1},\bZ_{2j})+f(\bZ_{2i-1},\bZ_{2j-1})\}$.
Let $F$ denote the envelope function of $\cF$ for which $0<\mathbb{E}
F^r<\infty$, for any $r\geq 1$. The covering number $N_r(\varepsilon,\mathbb{%
P}\otimes\mathbb{P},\cF,F)$ is defined as the smallest cardinality for a
subclass $\cF^*$ of $\cF$ such that $\min_{f^*\in\cF^*}\mathbb{E}%
|f-f^*|^r\leq \varepsilon^r\mathbb{E} F^r$, for each $f\in\cF$.

\subsubsection{Some Auxiliary Lemmas}

\label{sec: app-auxlemmas} 
\begin{lemma}\label{lem:fsquare}
  Suppose that $\F$ is $b$-uniformly bounded, then the class $
  \F^2=\{f^2:f\in \F\}$
  with envelope $b^2$ satisfies $
  N_r(2\varepsilon ,\P\otimes\P, \F^2, b^2)\leq N_r(\varepsilon, \P\otimes\P, \F, b)$.
\end{lemma}
\begin{proof}
  Find functions $f_1,\ldots,f_m$ such that
  \begin{equation*}
    \min_{i}\E|f-f_i|^r\leq \varepsilon^r b^r, \quad \text{for each $f\in\F$.}
  \end{equation*}
  Then, with the appropriate $i$,
  \begin{equation*}
    \begin{aligned}
      \E|f^2-f^2_i|^r\leq (2b)^r\E|f-f_i|^r\leq (2b)^r\varepsilon^r b^r= (2\epsilon)^r (b^2)^r.
    \end{aligned}
  \end{equation*}
  This implies that $N_r(2\varepsilon ,\P\otimes\P, \F^2, b^2)\leq N_r(\varepsilon, \P\otimes\P, \F, b)$.
\end{proof}

\begin{lemma}\label{lem:ep}
  Suppose that $\F$ is $b$-uniformly bounded. Then $\E\sup_{g\in\P\F}|\P_n g-\E g|\lesssim \sqrt{\nu/n}$,
  where $\P\F:=\{\E_{\P}f(\bz,\cdot):f\in\F\}$.
\end{lemma}
\begin{proof}
  With a little abuse of notation, let $\epsilon_1,\epsilon_2,\cdots$ be the Rademacher sequence, where $\epsilon_i\in\{-1,1\}$ is symmetric around 0. By the classic symmetrization theorem \citep[cf. Theorem 8.8 in][]{kosorok2007introduction}, we  have
  \begin{equation}\label{eqn: symmetrization}
    \E\sup_{g\in\P\F}|\P_n g-\E g|\leq \E_{\bZ}\E_{\epsilon}\sup_{g\in\P\F}\Big|\frac{1}{n}\sum_{i=1}^n \epsilon_i g(\bZ_i)\Big|.
  \end{equation}

  Next, we try to bound $\E_{\epsilon}\sup_{g\in\P\F}|\sum_{i=1}^n \epsilon_i g(\bz_i)/n|$ for fixed $\bz_i$. To that end, consider the stochastic process $\{\sum_{i=1}^n \epsilon_i g(\bz_i)/\sqrt{n}:g\in\P\F\}$. It is easy to verify that $\sum_{i=1}^n \epsilon_i \{g_1(\bz_i)-g_2(\bz_i)\}/\sqrt{n}$ is sub-gaussian with parameter $\norm{g_1-g_2}^2_{L_2(\P_n)}:=\sum_{i=1}^n\{g_1(\bz_i)-g_2(\bz_i)\}^2/n$, where $g_1,g_2\in\P\F$. Consequently, Dudley's entropy integral, combined with the fact that $\sup_{g_1,g_2\in\P\F}\norm{g_1-g_2}_{L_2(\P_n)}\leq 2b$, implies that
  \begin{equation}\label{eqn:entropybound}
    \E_{\epsilon}\sup_{g\in\P\F}\Big|\frac{1}{n}\sum_{i=1}^n \epsilon_i g(\bz_i)\Big|\leq \frac{24}{\sqrt{n}}\int_0^{2b}\sqrt{\log N_2(t/b,\P_n,\P\F,b)}\diff t.
  \end{equation}
  By  Theorem 9.3 in \cite{kosorok2007introduction} and Lemma 20 in \cite{nolan1987u}, there exists a universal constant $K$ such that $N_2(t/b,\P_n,\P\F,b)\leq K\nu(16e)^\nu(b/t)^{2(\nu-1)}$. Substituting this bound into \eqref{eqn:entropybound}, we find that there exist constants $c_0, c_1$, only depending on $K,b$ but not on $(\nu,n)$, such that
   \begin{equation*}
    \E_{\epsilon}\sup_{g\in\P\F}\Big|\frac{1}{n}\sum_{i=1}^n \epsilon_i g(\bz_i)\Big|\leq c_0 \sqrt{\frac{\nu}{n}}\left\{1+\int_0^{2b}\sqrt{\log(b/t)}\diff t\right\} \leq c_1\sqrt{\frac{\nu}{n}}.
  \end{equation*}
  Combining this with \eqref{eqn: symmetrization} implies that $\E\sup_{g\in\P\F}|\P_n g-\E g|\leq c_1\sqrt{\nu/n}$.
  This completes the proof.
\end{proof}

\begin{lemma}\label{lem:up}
  Suppose that $\F$ is $\P$-degenerate and $b$-uniformly bounded. Then $\E\sup_{f\in\F}|\U_n f|\lesssim\nu/n$.
\end{lemma}
\begin{proof}
 First, by the relationship between $\Sm_n$ and $\U_n$: $\Sm_n = n(n-1)\U_n$, we just need to show that
$\E\sup_{f\in\F}|\Sm_n f/(n\nu)|$ is bounded.
Apply Theorem~6 in \citet{nolan1987u} to get
\begin{equation}\label{eqn:maximalinequality}
  \begin{aligned}
     \E\sup_{f\in\F}|\Sm_nf|\leq C\E\bigg\{\sigma_n+\tau_nJ_n\Big(\frac{\sigma_n}{\tau_n}\Big)\bigg\},
  \end{aligned}
\end{equation}
where $C$ is a universal constant, $\sigma_n=\sup_{f\in\cF}(\T_n f^2)^{1/2}/4$, $\tau_n=(\T_n b^2)^{1/2}$, and $J_n(x)=\int_0^x\log N_2(t,\\ \T_n,\F,b)\diff t$. By  Theorem 9.3 in \cite{kosorok2007introduction}, we have $N_2(t,\T_n,\F,1)\leq K \nu(4e)^\nu(2/t)^{2(\nu-1)}$, and thus $J_n(x)\leq cH(x)\nu$ for some constant $c$ depending on $K$, where $H(x)=x\{1+\log(1/x)\}$.

Since $\F$ is $b$-uniformly bounded, it holds that $\sigma_n/\tau_n\in[0,1/4]$. Note also that
 $H(x)$ is bounded when $x\in[0,1]$.  We immediately have $H(\sigma_n/\tau_n)$ is bounded. Additionally,
by the definition of $\T_n$, we see that $\tau_n=\{4n(n-1)\}^{1/2}\lesssim n$.
Combining all these points with
 \eqref{eqn:maximalinequality} implies that there exists some constant $c'$ depending on $C,c$ such that
\begin{equation*}
  \begin{aligned}
    \frac{\E\sup_{f\in\F}|\Sm_nf|}{nv}\leq  c'\E H\Big(\frac{\sigma_n}{\tau_n}\Big) \frac{\tau_n}{n}<C'
    \end{aligned}
    \end{equation*}
for some large enough absolute constant $C'$.  This completes the proof.
\end{proof}

\begin{lemma}\label{lem:uniform as convergence}
 If for each $\varepsilon>0$, (i) $\log N_1(\varepsilon,\T_n,\F,F)=O_\P(n)$, (ii) $\log N_1(\varepsilon, \P_n\otimes \P,\F,F)=o_{\P}(n)$, (iii) $\log N_1(\varepsilon,\P\otimes \P,\F,F)=o(n)$, then $\sup_{f\in\F}|\U_n f-\E f|\rightarrow 0$ almost surely.
\end{lemma}
The proof of this lemma follows along the same lines as the proof of Theorem
7 in \cite{nolan1987u}, though the condition (iii) in this lemma is
different from there.

\subsubsection{Proof of Theorem~\protect\ref{thm:uniformloose}}

\label{sec:app-uniform} 
\begin{proof}
   (i)  It is equivalent to showing that there exists a sequence of nonnegative real numbers  $\delta_n$ converging to zero such that
\begin{equation*}
       \P\bigg\{ \sup_{\btheta\in\overline{\mathcal{B}}(\btheta_0,r_n)}|\U_n h(\cdot,\cdot;\btheta)|\geq \delta_n \nu/n\bigg\}=o(1),
\end{equation*}
or
\begin{equation*}
   \P\bigg\{ \sup_{\btheta\in\overline{\mathcal{B}}(\btheta_0,r_n)}|\Sm_n h(\cdot,\cdot;\btheta)/(n\nu)|\geq \delta_n \bigg\}=o(1).
\end{equation*}
By Chebyshev's inequality, it suffices to show that
\begin{equation*}
  \E\Big\{\sup_{\btheta\in\overline{\mathcal{B}}(\btheta_0,r_n)}|\Sm_n h(\cdot,\cdot;\btheta)/(n\nu)|\Big\}/\delta_n=o(1).
\end{equation*}

We try to bound $\E\{\sup_{\btheta\in\overline{\mathcal{B}}(\btheta_0,r_n)}|\Sm_n h(\cdot,\cdot;\btheta)/(n\nu)|\}$.
 Without loss of generality, assume $\F$ is uniformly bounded by $b=1/4$. Thus, for any $\btheta\in\Theta$, $h(\cdot,\cdot;\btheta)\leq 1$, i.e., the class of functions $\mathscr{H}:=\{h^2(\cdot,\cdot;\btheta):\btheta\in\overline{\mathcal{B}}(\btheta_0,r_n)\}$ is 1-uniformly bounded. Similar to the proof of Lemma~\ref{lem:up}, we apply Theorem 6 in \citet{nolan1987u} here to get
\begin{equation}\label{eqn:beijing}
  \begin{aligned}
   \E\Big\{\sup_{\btheta\in\overline{\mathcal{B}}(\btheta_0,r_n)}|\Sm_n h(\cdot,\cdot;\btheta)/(n\nu)|\Big\}&\leq C_1 \E H\Big(\sup_{\btheta\in\overline{\mathcal{B}}(\btheta_0,r_n)}\{\T_n h^2(\cdot,\cdot;\btheta)\}^{1/2}/(2n)\Big)\\
    &\leq C_1 H\Big(\E\Big[\sup_{\btheta\in\overline{\mathcal{B}}(\btheta_0,r_n)}\{\T_n h^2(\cdot,\cdot;\btheta)\}^{1/2}/(2n)\Big]\Big)\\
    &= C_1 H\Big(\E\Big\{\sup_{\btheta\in\overline{\mathcal{B}}(\btheta_0,r_n)}\T_n h^2(\cdot,\cdot;\btheta)/(2n)^2\Big\}^{1/2}\Big),
  \end{aligned}
\end{equation}
where $C_1$ is some constant. The second inequality holds because $H(x)$ is concave in $x$.

Note that $\T_n  h^2(\cdot,\cdot;\btheta)/(2n)^2=\T_n  h^2(\cdot,\cdot;\btheta)/\{2n(2n-1)\}\cdot \{2n(2n-1)\}/(2n)^2\leq \U_{2n}  h^2(\cdot,\cdot;\btheta)\leq 1$ and that $H(x)$ is increasing in $(0,1]$. {\blue Thus, from \eqref{eqn:beijing}, we additionally have}
%Combining these with \eqref{eqn:beijing} yields
\begin{equation}\label{eqn:shanghai}
  \begin{aligned}
     \E\Big\{\sup_{\btheta\in\overline{\mathcal{B}}(\btheta_0,r_n)}|\Sm_n h(\cdot,\cdot;\btheta)/(n\nu)|\Big\}&\leq C_1 H\Big(\E\Big\{\sup_{\btheta\in\overline{\mathcal{B}}(\btheta_0,r_n)}\U_{2n} h^2(\cdot,\cdot;\btheta)\Big\}^{1/2}\Big)\\
    &\leq C_1  H\Big(\Big[\E\Big\{\sup_{\btheta\in\overline{\mathcal{B}}(\btheta_0,r_n)}\U_{2n} h^2(\cdot,\cdot;\btheta)\Big\}\Big]^{1/2}\Big),
  \end{aligned}
\end{equation}
where the last inequality holds because $x^{1/2}$ is concave in $x$. Now, we need only to consider $ \E\{\sup_{\btheta\in\overline{\mathcal{B}}(\btheta_0,r_n)}\U_{2n} h^2(\cdot,\cdot;\btheta)\}$.

By a decomposition of $\U_{2n}h^2(\cdot,\cdot,\btheta)$ into a sum of its expected value, plus a smoothly parameterized, zero-mean empirical process, plus a degenerate $U$-process of order two, we have
\begin{equation}\label{eqn: decomposition hsquare}
  \begin{aligned}
    \E\Big\{\sup_{\btheta\in\overline{\mathcal{B}}(\btheta_0,r_n)}\U_{2n} h^2(\cdot,\cdot;\btheta)\Big\}
    \leq & \sup_{\btheta\in\overline{\mathcal{B}}(\btheta_0,r_n)} \E h^2(\cdot,\cdot;\btheta)+\E \Big\{\sup_{\btheta\in\overline{\mathcal{B}}(\btheta_0,r_n)}|\P_{2n}h_1(\cdot;\btheta)| \Big\}\\
   &\hspace{1em} +\E \Big\{\sup_{\btheta\in\overline{\mathcal{B}}(\btheta_0,r_n)}|\P_{2n}
   h_2(\cdot,\cdot;\btheta)| \Big\},
  \end{aligned}
\end{equation}
where $h_1(\bz,\btheta)=\E h^2(\bz,\cdot;\btheta)+\E h^2(\cdot,\bz;\btheta)-2\E h^2(\cdot,\cdot;\btheta)$
and $h_2(\bz_1,\bz_2;\btheta)=h^2(\bz_1,\bz_2;\btheta)-\E h^2(\bz_1,\cdot;\btheta)
-\E h^2(\cdot,\bz_2;\btheta)+\E h^2(\cdot,\cdot;\btheta)$.

By the condition in (i), it holds that $\sup_{\btheta\in\overline{\mathcal{B}}(\btheta_0,r_n)} \E h^2(\cdot,\cdot;\btheta)\leq \epsilon_n$. By Lemmas 16 and 20 in \citet{nolan1987u}, and Lemma \ref{lem:fsquare},  we have $N_r(\varepsilon,\Q,\mathscr{H},1)\leq N_r(\varepsilon/16,\Q,\cF,1/4)^4$. Then, following the proof of Lemma \ref{lem:ep}, we have $\E \{\sup_{\btheta\in\overline{\mathcal{B}}(\btheta_0,r_n)}|\P_{2n}h_1(\cdot;\btheta)| \}\leq C_2\sqrt{\nu/n}$ for some constant $C_2$. Additionally, following the proof of Lemma \ref{lem:up}, we have $\E \{\sup_{\btheta\in\overline{\mathcal{B}}(\btheta_0,r_n)}|\P_{2n}
   h_2(\cdot,\cdot;\btheta)| \}\leq C_3\nu/n$ for some constant $C_3$.

 Take $\delta_n=H^{1/2}((\epsilon_n+C_2 \sqrt{\nu/n}+C_3\nu/n)^{1/2})$. If $\epsilon_n\rightarrow 0$ and $\nu/n\rightarrow 0$, then
 \begin{equation*}
   \begin{aligned}
     \E\Big\{\sup_{\btheta\in\overline{\mathcal{B}}(\btheta_0,r_n)}|\Sm_n h(\cdot,\cdot;\btheta)/(n\nu)|\Big\}/\delta_n \leq C_1 H^{1/2}((\epsilon_n+C_2 \sqrt{\nu/n}+C_3\nu/n)^{1/2})=o(1),
   \end{aligned}
 \end{equation*}
 because $H(x)\rightarrow 0 $ as $x\rightarrow 0$. This completes proof of (i).

 (ii) The proof is based on \eqref{eqn:beijing}--\eqref{eqn: decomposition hsquare} in the proof of (i). First, by the condition in (ii), it holds that $\sup_{\btheta\in\overline{\mathcal{B}}(\btheta_0,r_n)}  \E h^2(\cdot,\cdot;\btheta)\leq \tilde\epsilon_n$. Then, similar to the proof of (i), $\E\{\sup_{\btheta\in\overline{\mathcal{B}}(\btheta_0,r_n)}|\P_{2n}h_1(\cdot;\btheta)| \}\leq c'\sqrt{\nu/n}$ for some constant $c'$, and $\E\{\sup_{\btheta\in\overline{\mathcal{B}}(\btheta_0,r_n)}|\P_{2n}
   h_2(\cdot,\cdot;\btheta)| \}\leq C'\nu/n$ for some constant $C'$. Since $\tilde \eta_n = \sqrt{\nu/n} \vee\tilde\epsilon_n$ and $\nu/n\rightarrow 0$, there exists a constant $c''$ depending on $c',C'$ such that
   \begin{equation*}
     \begin{aligned}
       \E\Big\{\sup_{\btheta\in\overline{\mathcal{B}}(\btheta_0,r_n)}\U_{2n} h^2(\cdot,\cdot;\btheta)\Big\}\leq c''\tilde\eta_n
     \end{aligned}
   \end{equation*}
   holds for sufficiently large $n$. Combining this with \eqref{eqn:shanghai} implies that
   \begin{equation*}
     \begin{aligned}
        \E\Big\{\sup_{\btheta\in\overline{\mathcal{B}}(\btheta_0,r_n)}|\Sm_n h(\cdot,\cdot;\btheta)/(n\nu)|\Big\}\leq C''\log(1/\tilde\eta_n)\tilde\eta_n^{1/2}
     \end{aligned}
   \end{equation*}
   for some constant $C''$. Finally, by the relationship between $\U_n$ and $\Sm_n$, we conclude that
   \begin{equation*}
     \begin{aligned}
        &\E \sup_{\btheta\in\overline{\mathcal{B}}(\btheta_0,r_n)}|\U_nh(\cdot,\cdot;\btheta)| \leq C''\log(1/\tilde\eta_n)\tilde\eta_n^{1/2}\nu/n
     \end{aligned}
   \end{equation*}
   holds for sufficiently large $n$.
 \end{proof}

\subsubsection{Proof of Theorem~\protect\ref{thm: consistency}}

\begin{proof}
  The proof is twofold. We first show the uniform convergence of $\Gamma_n(\btheta)$, and then establish the consistency of $\hat\btheta_n$.

  \noindent
  {\it Step 1.} By  Theorem 9.3 in \cite{kosorok2007introduction},  we have $\log N_1(\varepsilon,\mu,\F,F)\lesssim \nu$ for any $\varepsilon>0$ and any finite measure $\mu$. If $\nu/n\rightarrow 0$, then all the three conditions in Lemma
  \ref{lem:uniform as convergence} hold. {\blue Apply this lemma here to get that} $\Gamma_n(\btheta)$ converges almost surely to $\Gamma(\btheta)$ uniformly in $\btheta\in\Theta$.

  \noindent
  {\it Step 2.}
  {\blue Let $\Theta_0(r) = \mathcal{B}(\btheta_0, r)$, where $r\leq r_0$. By Assumption~\ref{ass:new1}, we see that $\Theta_1 := \Theta- \Theta_0(r)$ is compact.}
%Let $\Theta_0$ be an open set in $\R^p$ such that $\btheta_0\in \Theta_0$. By Assumption~\ref{ass:interior}, we see that $\Theta_1 := \Theta-(\Theta\cap \Theta_0)$ is compact.
By Assumption \ref{ass:continuous}, $\Gamma(\btheta)$ is continuous. Combining these two pieces yields that
 $\max_{\btheta\in\Theta_1} \Gamma(\btheta)$ exists. %By Assumption~\ref{ass:identifiable}, $\Gamma(\btheta)$ attains a unique maximum in $\Theta$ at $\btheta_0$ . Thus, there exists $\xi := \Gamma(\btheta_0)-\max_{\btheta\in\Theta_1}\Gamma(\btheta)>0$.
 {\blue Again, by Assumption \ref{ass:new1}, we know that $\Gamma(\btheta_0)-\max_{\btheta\in\Theta_1}\Gamma(\btheta)\geq \xi_0$.
 }

  By {\it Step 1}, we can find a sufficiently large $N$ such that for all $n>N$,
\begin{equation*}
  \sup_{\btheta\in\Theta}|\Gamma_n(\btheta)-\Gamma(\btheta)|<\xi_0/2
\end{equation*}
holds almost surely.
Combining this with the definition of $\hat\btheta_n$ yields that
\begin{equation*}
  \begin{aligned}
    \Gamma(\btheta_0)< \Gamma_n(\btheta_0) + \xi_0/2
    \leq \Gamma_n(\hat\btheta_n) +\xi_0/2
    < \Gamma(\hat\btheta_n) +\xi_0.
  \end{aligned}
\end{equation*}
This implies that $\hat\btheta_n\not\in \Theta_1$, i.e., $\hat\btheta_n\in \Theta_0(r)$ for all $n>N$. Since this is true for {\blue any $r < r_0$,} we have
\begin{equation*}
\norm{\hat\btheta_n -\btheta_0} \rightarrow 0 \quad \text{almost surely,}
\end{equation*}
and hence also in probability. This completes the proof.
\end{proof}

\subsubsection{Proof of Theorem~\protect\ref{thm: general consrate}}

\begin{proof}
The proof is conducted in four steps. Based on the Hoeffding decomposition of $\Gamma_n(\btheta)$, we consider $\Gamma(\btheta)$, $\P_n g(\cdot;\btheta)$ and $\U_n h(\cdot,\cdot;\btheta)$ separately in the first three steps. We finally obtain the convergence rate  of $\hat\btheta_n$ in the last step. %\hf{Of note, throughout the proofs of Theorems \ref{thm: general consrate} and \ref{thm:generalASN}, though not explicitly referenced, the general method laid out in Theorem \ref{thm:han} plays an intrinsic role. For example, the first two steps in the following proof are actually verifying the conditions in Theorem \ref{thm:han}, and hence have helped establish the ASN of $\tilde\btheta_n$. Similarly, the proof of Theorem \ref{thm:generalASN} is intrinsically proving that the gap between $\tilde\btheta_n$ and $\hat\btheta_n$ is ignorable, and hence verified the condition in Theorem \ref{thm:han}(iii).}

\noindent
{\it Step 1.} Fixing $\btheta\in\overline{\mathcal{B}}(\btheta_0,r)$, define
\begin{equation}\label{eqn:rhotheta}
\omega(\btheta) = \E \tau(\cdot;\btheta)- \E \tau(\cdot,\btheta_0) - (\btheta-\btheta_0)^\top \Vb(\btheta-\btheta_0) = 2\Gamma(\btheta)-(\btheta-\btheta_0)^\top \Vb(\btheta-\btheta_0).
\end{equation}
Additionally, expand $\omega(\btheta)$ about $\btheta_0$ to get
\begin{equation}\label{eqn:nablarhotheta}
\omega(\btheta) = (\btheta-\btheta_0)^\top \nabla_1\omega(\btheta'),
\end{equation}
where $\btheta'$ is a point on the line connecting $\btheta_0$ and $\btheta$, and $\nabla_1\omega(\btheta') = \nabla_1\E \tau(\cdot;\btheta') - 2\Vb(\btheta'-\btheta_0)$.
Expand $\nabla_1\E \tau(\cdot;\btheta')$ in $\nabla_1\omega(\btheta')$ about $\btheta_0$ to get
 \[
 \nabla_1\omega(\btheta')= 2\Vb(\btheta'')(\btheta'-\btheta_0) - 2\Vb(\btheta'-\btheta_0)=2\{\Vb(\btheta'')-\Vb\}(\btheta'-\btheta_0)
 \]
  for $\btheta''$ between $\btheta_0$ and $\btheta'$.  By Assumption~\ref{ass:taufunction}(ii) and (iii), we have
  \begin{equation}\label{eqn: rhothetabound}
  \begin{aligned}
     \sup_{\btheta'\in\overline{\mathcal{B}}(\btheta_0,r)}\norm{\nabla_1\omega(\btheta')}&\leq 2\sup_{\btheta'\in\overline{\mathcal{B}}(\btheta_0,r)}\norm{\Vb^{1/2}
     \{\Ib_p-\Vb^{-1/2}\Vb(\btheta'')\Vb^{-1/2}\}\Vb^{1/2}}\norm{\btheta'-\btheta_0}
     \\&\leq 2c_{\max}\rho(r)\norm{\btheta-\btheta_0}.
  \end{aligned}
  \end{equation}
Combining this with \eqref{eqn:rhotheta} and \eqref{eqn:nablarhotheta} yields
\begin{equation}\label{eqn:step1result}
\sup_{\btheta\in\overline{\mathcal{B}}(\btheta_0,r)}
\big|\Gamma(\btheta)-\frac{1}{2}(\btheta-\btheta_0)^\top \Vb(\btheta-\btheta_0)\big|\leq c_{\max}\rho(r) \norm{\btheta-\btheta_0}^2.
\end{equation}

\noindent
{\it Step 2.} Fixing $\bz$ in $\R^m$ and $\btheta$ in $\overline{\mathcal{B}}(\btheta_0,r)$, define
\[
\psi(\bz;\btheta) = \tau(\bz;\btheta)-\tau(\bz;\theta_0)-(\btheta-\btheta_0)^\top \nabla_1\tau(\bz;\btheta_0)-(\btheta-\btheta_0)^\top \Vb(\btheta-\btheta_0).
\]
With a little abuse of notation, we still use $\btheta'$ to denote some point between $\btheta_0$ and $\btheta$ below. Expand $\psi(\bz;\btheta)$ about $\btheta_0$ to get
\[
\psi(\bz;\btheta) = (\btheta-\btheta_0)^\top \nabla_1\psi(\bz;\btheta') = (\btheta-\btheta_0)^\top \{\nabla_1 \tau(\bz;\btheta') - \nabla_1 \tau(\bz;\btheta_0) - 2\Vb(\btheta'-\btheta_0)\}.
\]
Note that $\tau(\bz;\btheta') = \zeta(\bz;\btheta')+\E\tau(\cdot;\btheta')$. It then follows from the above equation and $\nabla_1\E \tau(\cdot;\btheta_0) = 0$ that
\begin{equation*}
\begin{aligned}
\P_n\psi(\cdot;\btheta) &= (\btheta-\btheta_0)^\top\{\P_n\nabla_1\zeta(\cdot;\btheta')-\P_n\nabla_1\zeta(\cdot;\btheta_0)+
\nabla_1\E \tau(\cdot;\btheta')- 2\Vb(\btheta'-\btheta_0)\}\\
&=(\btheta-\btheta_0)^\top \P_n\{\nabla_1\zeta(\cdot;\btheta')-\nabla_1\zeta(\cdot;\btheta_0)\} + (\btheta-\btheta_0)^\top \{\nabla_1\E\tau(\cdot;\btheta')- 2\Vb(\btheta'-\btheta_0)\}.
\end{aligned}
\end{equation*}
By {\it Step 1}, we have that
\begin{equation}\label{eqn: deterministic part}
  \sup_{\btheta\in\overline{\mathcal{B}}(\btheta_0,r)}\norm{(\btheta-\btheta_0)^\top \{\nabla_1\E\tau(\cdot;\btheta')- 2\Vb(\btheta'-\btheta_0)\}}\leq 2c_{\max}\rho(r) \norm{\btheta-\btheta_0}^2.
\end{equation}

 Next, we try to bound $ \sup_{\btheta\in\overline{\mathcal{B}}(\btheta_0,r)}\norm{(\btheta-\btheta_0)^\top \P_n\{\nabla_1\zeta(\cdot;\btheta')-\nabla_1\zeta(\cdot;\btheta_0)\}}$.
Consider the vector process
%\[
%\mathcal{U}(\theta)=\sqrt{n}P_nV^{-1/2}(\nabla\zeta(\cdot,\theta)-\nabla\zeta(\cdot,\mathbf{0})).
%\]
\[
\Lambda(\btheta)=\sqrt{n}\P_n\{\nabla_1\zeta(\cdot;\btheta)-\nabla_1\zeta(\cdot;\btheta_0)\}.
\]
%Change the variable by
%%$\upsilon=V^{1/2}\theta$
%$\upsilon=\theta$
%and consider the vector process $\mathcal{V}(\upsilon)=\mathcal{U}(\theta)$.
 According to Assumption~\ref{ass:taufunction}(v), it holds,  for any $\bgamma_1$, $\bgamma_2\in \Sm^{p-1}$, that
%\begin{equation*}
%  \begin{aligned}
%    \log \E \exp\left\{\frac{\lambda}{\omega}\gamma_1^\top\nabla \mathcal{V}(\upsilon)\gamma_2\right\}=n\log\E\exp\left\{\frac{\lambda}{\sqrt{n}\omega}\gamma_1^\top
%    V^{-1/2}\nabla_2\zeta(\cdot,\theta)V^{-1/2}\gamma_2\right\}
%    \leq \frac{\nu_0^2\lambda^2}{2}
%  \end{aligned}
%\end{equation*}
\begin{equation*}
  \begin{aligned}
    \log \E \exp\left\{\lambda\bgamma_1^\top\nabla_1 \Lambda(\btheta)\bgamma_2\right\}=n\log\E\exp\left\{\frac{\lambda}{\sqrt{n}}\bgamma_1^\top
    \nabla_2\zeta(\cdot;\btheta)\bgamma_2\right\}
    \leq \frac{\nu_0^2\lambda^2}{2}
  \end{aligned}
\end{equation*}
for any $|\lambda|\leq g_n$ with $g_n=\sqrt{n}\ell_0$. It then follows from Theorem A.3 in \cite{spokoiny2013bernstein} that for any $0<\varepsilon<1$,
%\[
%\norm{\mathcal{U}(\theta)}=\norm{\mathcal{V}(\upsilon)}=O_\P(\sqrt{p}\norm{\theta})
%\]
\begin{equation*}
  \P\bigg\{\sup_{\btheta\in\overline{\mathcal{B}}(\btheta_0,r)}\norm{\Lambda(\btheta)}>6\nu_0r d_p(\varepsilon)\bigg\}\leq\varepsilon,
\end{equation*}
where
  \begin{equation*}
      d_p(\varepsilon)=\left\{\begin{array}{ll}
      \sqrt{4p-2\log\varepsilon} \qquad &
      \text{if $4p-2\log\varepsilon\leq g_n^2$,}\\
      g_n^{-1}\log\varepsilon+\frac{1}{2}(4pg_n^{-1}+g_n) \qquad &
       \text{if $4p-2\log\varepsilon>g_n^2$.}
    \end{array}\right.
  \end{equation*}
  Thus,
  \begin{equation}\label{eqn: stochastic part}
  \P\bigg\{\sup_{\btheta\in\overline{\mathcal{B}}(\btheta_0,r)}\big|(\btheta-\btheta_0)^\top \P_n\{\nabla_1\zeta(\cdot;\btheta')-\nabla_1\zeta(\cdot;\btheta_0)\}\big|>\frac{6\nu_0r }{\sqrt{n}}d_p(\varepsilon) \norm{\btheta-\btheta_0} \bigg\}\leq\varepsilon.
\end{equation}
 This,  combined with \eqref{eqn: deterministic part}, implies that
  \begin{equation}\label{eqn:Pnpsi}
     \P\bigg\{\sup_{\btheta\in\overline{\mathcal{B}}(\btheta_0,r)}\norm{\P_n\psi(\cdot;\btheta)}>
    2 c_{\max}\rho(r) \norm{\btheta-\btheta_0}^2+\frac{6\nu_0r }{\sqrt{n}}d_p(\varepsilon) \norm{\btheta-\btheta_0} \bigg\}\leq\varepsilon.
  \end{equation}

Note that $\tau(\bz;\btheta_0)=0$ and
\begin{equation*}
\begin{aligned}
g(\bz;\btheta) &= \tau(\bz;\theta)-\tau(\bz;\btheta_0)-2\Gamma(\btheta) \\
&= (\btheta-\btheta_0)^\top \nabla_1\tau(\bz;\btheta_0) + \psi(\bz;\btheta)-\{2\Gamma(\btheta)-(\btheta-\btheta_0)^\top \Vb(\btheta-\btheta_0)\}\,.
\end{aligned}
\end{equation*}
Apply~\eqref{eqn:step1result} and~\eqref{eqn:Pnpsi} to see that
\begin{equation}\label{eqn:step2result}
  \P\bigg\{\sup_{\btheta\in\overline{\mathcal{B}}(\btheta_0,r)}\big|\P_ng(\cdot;\btheta)-\frac{1}
  {\sqrt{n}}(\btheta-\btheta_0)^\top \bwn\big|>
     4c_{\max}\rho(r) \norm{\btheta-\btheta_0}^2+\frac{6\nu_0r }{\sqrt{n}} d_p(\varepsilon) \norm{\btheta-\btheta_0} \bigg\}\leq\varepsilon,
\end{equation}
 where
$\bwn=\sqrt{n}\P_n\nabla_1 \tau(\cdot;\btheta_0)$.

\noindent
{\it Step 3.}
By Assumption \ref{ass:continuous}, $f(\bz_1,\bz_2;\btheta)$ is continuous at $\btheta_0$  almost surely. Since $\F$ is uniformly bounded, a dominated convergence argument implies that the same holds true for $h(\bz_1,\bz_2;\btheta)$. In view of $f(\bz_1,\bz_2,\btheta_0)=0$ for all $\bz_1,\bz_2$, it holds that $h(\bz_1,\bz_2;\btheta_0)=0$.
Thus, the boundedness of $h$ and the dominated convergence theorem establish
\begin{equation}\label{eqn:Qh}
\E h^2(\cdot,\cdot;\btheta)\rightarrow 0 \quad \text{as} \quad \norm{\btheta-\btheta_0}_2\rightarrow 0.
\end{equation}
Equivalently, there exists a constant $\alpha(r)>0$ such that $\sup_{\btheta\in\overline{\mathcal{B}}(\btheta_0,r)}\E h^2(\cdot,\cdot;\btheta)\leq \alpha(r)$ and $\alpha(r)\rightarrow 0$ as $r\rightarrow 0$. By Theorem \ref{thm:uniformloose}, there exists a sequence of nonnegative real numbers $\delta_n$ (depending on $\alpha(r),\nu,n$) converging to zero as $r\rightarrow 0 $ and $n\rightarrow \infty$,
 such that
\begin{equation}\label{eqn:step3result}
  \P\bigg\{ \sup_{\btheta\in\overline{\mathcal{B}}(\btheta_0,r)}|\U_n h(\cdot,\cdot;\btheta)|> \delta_n \nu/n\bigg\}\leq \epsilon
\end{equation}
holds for sufficiently large $n$.

\noindent
{\it Step 4.} The Hoeffding decomposition, combined with \eqref{eqn:step1result}, \eqref{eqn:step2result}, and \eqref{eqn:step3result} in the above three steps,  implies that
\begin{equation}\label{eqn:Gammantheta}
\begin{aligned}
  &\P\bigg\{\sup_{\btheta\in\overline{\mathcal{B}}(\btheta_0,r)}\big|\Gamma_n(\btheta)-\frac{1}{2}
  (\btheta-\btheta_0)^\top\Vb(\btheta-\btheta_0)-
  \frac{1}
  {\sqrt{n}}(\btheta-\btheta_0)^\top \bwn\big|>\\ & \hspace{8em} 5c_{\max}\rho(r) \norm{\btheta-\btheta_0}^2+\frac{6\nu_0r }{\sqrt{n}} d_p(\varepsilon) \norm{\btheta-\btheta_0}+\delta_n\frac{\nu}{n} \vphantom{\int_1^2} \bigg\}\leq2\varepsilon.
   \end{aligned}
\end{equation}
In view of $\bwn=\sqrt{n}\P_n\nabla_1 \tau(\cdot;\btheta_0)$, it holds that $(\btheta-\btheta_0)^\top \bwn/\{(\btheta-\btheta_0)^\top\Deltab(\btheta-\btheta_0)\}^{1/2}\Rightarrow N(0,1)$. This, combined with Assumption \ref{ass:taufunction}(iv), implies that there exists a constant $b_{\varepsilon}$ depending on $d_{\max}$ such that
\begin{equation*}
  \P\big\{|(\btheta-\btheta_0)^\top W_n|>b_{\varepsilon}\norm{\btheta-\btheta_0}\big\}\leq \varepsilon
\end{equation*}
holds for sufficiently large $n$.
Define the set
\begin{equation*}
\begin{aligned}
  \mathcal{A}_{n,\varepsilon} = &\bigg\{\bZ: \sup_{\btheta\in\overline{\mathcal{B}}(\btheta_0,r)}\big|\Gamma_n(\btheta)-\frac{1}{2}
  (\btheta-\btheta_0)^\top\Vb(\btheta-\btheta_0)\big|\leq
  \frac{b_{\varepsilon}}{\sqrt{n}}\norm{\btheta-\btheta_0} +\\&\hspace{8em} 5c_{\max}\rho(r) \norm{\btheta-\btheta_0}^2+\frac{6\nu_0r }{\sqrt{n}} d_p(\varepsilon) \norm{\btheta-\btheta_0}+\delta_n\frac{\nu}{n} \vphantom{\int_1^2} \bigg\},
\end{aligned}
\end{equation*}
then $\P(\mathcal{A}_{n,\varepsilon})\geq 1-3\varepsilon$ holds for sufficiently large $n$. The following analysis is on the set $\mathcal{A}_{n,\varepsilon}$.

By Theorem \ref{thm: consistency}, $\norm{\hat\btheta_n-\btheta_0}\rightarrow 0$ almost surely. Thus, for sufficiently large $n$, $\hat\btheta_n\in\overline{\mathcal{B}}(\btheta_0,r)$. This implies that
\begin{equation*}
  \begin{aligned}
    \Gamma_n(\hat\btheta_n)\leq& \frac{1}{2}
  (\hat\btheta_n-\btheta_0)^\top\Vb(\hat\btheta_n-\btheta_0)+   \frac{b_{\varepsilon}}{\sqrt{n}}\norm{\hat\btheta_n-\btheta_0}+5c_{\max}\rho(r) \norm{\hat\btheta_n-\btheta_0}^2
  \\ &~~~~~~~~~ +\frac{6\nu_0r }{\sqrt{n}} d_p(\varepsilon) \norm{\hat\btheta_n-\btheta_0}+\delta_n\frac{\nu}{n}.
  \end{aligned}
\end{equation*}
In view of $\Gamma_n(\hat\btheta_n)\geq \Gamma_n(\btheta_0)=0$, it holds that
\begin{equation*}
  \begin{aligned}
    0\leq \frac{1}{2}
  (\hat\btheta_n-\btheta_0)^\top\Vb(\hat\btheta_n-\btheta_0)+   \frac{b_{\varepsilon}}{\sqrt{n}}\norm{\hat\btheta_n-\btheta_0}+5c_{\max}\rho(r) \norm{\hat\btheta_n-\btheta_0}^2
  +\frac{6\nu_0r }{\sqrt{n}} d_p(\varepsilon) \norm{\hat\btheta_n-\btheta_0}+\delta_n\frac{\nu}{n}.
    \end{aligned}
\end{equation*}
This, combined with Assumption \ref{ass:taufunction}(ii) and $\rho(r)<\frac{c_{\min}}{11c_{\max}}$, implies that
\begin{equation}\label{eqn:kappa}
  \begin{aligned}
     \frac{1}{2}\kappa
  \norm{\hat\btheta_n-\btheta_0}^2\leq   \frac{b_{\varepsilon}}{\sqrt{n}}\norm{\hat\btheta_n-\btheta_0}
  +\frac{6\nu_0r }{\sqrt{n}} d_p(\varepsilon) \norm{\hat\btheta_n-\btheta_0}+\delta_n\frac{\nu}{n},
    \end{aligned}
\end{equation}
where $\kappa = c_{\min}-10c_{\max}\rho(r)>0$. By the definition of $d_p(\varepsilon)$ and $p/n\rightarrow 0$, there exists a constant $c_{\varepsilon}$ such that $d_p(\varepsilon)\leq c_{\varepsilon} \sqrt{p}$ for sufficiently large $n$. Combining this with \eqref{eqn:kappa} yields that
\begin{equation*}\label{eqn:kappa square}
  \begin{aligned}
     \frac{1}{2}\kappa\Big(
  \norm{\hat\btheta_n-\btheta_0}-\frac{b_{\varepsilon}+6\nu_0rc_{\varepsilon} \sqrt{p}}{\kappa\sqrt{n}}\Big)^2
  \leq  \frac{(b_{\varepsilon}+6\nu_0rc_{\varepsilon} \sqrt{p})^2}{2\kappa n}+\delta_n\frac{\nu}{n}.
    \end{aligned}
\end{equation*}
Solving the above equation establishes that
\begin{equation*}
   \norm{\hat\btheta_n-\btheta_0}\leq C_\varepsilon \sqrt{\frac{\nu\vee p}{n}}
\end{equation*}
holds for sufficiently large $n$, where
 $C_{\varepsilon}$ is some constant  depending only on $c_{\min},c_{\max},\rho(r),d_{\max},\varepsilon$, but not depending on $\nu,p,n$. Thus,
 \begin{equation*}
   \P\bigg\{\norm{\hat\btheta_n-\btheta_0}\leq C_\varepsilon \sqrt{\frac{\nu\vee p}{n}}\bigg\}\geq 1-3\varepsilon
\end{equation*}
holds for sufficiently large $n$. This completes the proof.
\end{proof}

\subsubsection{Proof of Theorem~\protect\ref{thm:generalASN}}

\begin{proof}
\hf{The proof is based on the proof of Theorem \ref{thm: general consrate}.} We first define $\btn=\sqrt{n}(\hat\btheta_n-\btheta_0)$ and $\btsn=-\Vb^{-1}\bwn$.
By Theorem \ref{thm: general consrate}, for any $\varepsilon>0$, there exists a constant $C'_\varepsilon>0$ such that
\begin{equation}\label{eqn: consrate result}
  \P\bigg\{\norm{\hat\btheta_n-\btheta_0}> C'_\varepsilon \sqrt{\frac{\nu\vee p}{n}}\bigg\}\leq \varepsilon
\end{equation}
holds for sufficiently large $n$. By the definition of $\bwn$, $\bwn=\sqrt{n}\P_n\nabla_1 \tau(\cdot;\btheta_0)$, it holds that for any $\bgamma\in\R^p$, $\bgamma^\top \btsn/(\bgamma^\top \Vb^{-1}\Deltab\Vb^{-1}\bgamma)^{1/2}\Rightarrow N(0,1)$. This, combined with Assumption \ref{ass:taufunction}(ii) and (iv), implies that there exists a constant $C''_\varepsilon$ such that \begin{equation}\label{eqn: normal result}
  \P\big\{\norm{\btsn}\geq C''_\varepsilon\big\}\leq \varepsilon
\end{equation}
holds for sufficiently large $n$.
Thus, by \eqref{eqn: consrate result}  and  \eqref{eqn: normal result}, there exists a constant $c'_\varepsilon$ depending on $C'_\varepsilon,C''_\varepsilon$ such that
\begin{equation}\label{eqn: normal consrate result}
  \P\big(\mathcal{A}'_{n,\varepsilon}
  \big)\geq 1-2 \varepsilon
\end{equation}
holds for sufficiently large $n$, where   $\mathcal{A}'_{n,\varepsilon} :=\{\bZ:\hat\btheta_n\in\overline{\mathcal{B}}(\btheta_0,r_n), \btsn/\sqrt{n}+\btheta_0\in\overline{\mathcal{B}}(\btheta_0,r_n)\}$ and $r_n:=c'_\varepsilon\sqrt{(\nu\vee p)/n}$.

Fix $\btheta\in\overline{\mathcal{B}}(\btheta_0,r_n)$. Then  following the proofs of {\it Steps 1--2} in Theorem \ref{thm: general consrate}, we have
\begin{equation}\label{eqn:step1result asn}
\sup_{\btheta\in\overline{\mathcal{B}}(\btheta_0,r_n)}
\big|\Gamma(\btheta)-\frac{1}{2}(\btheta-\btheta_0)^\top \Vb(\btheta-\btheta_0)\big|\leq c_{\max}\rho(r_n) \norm{\btheta-\btheta_0}^2,
\end{equation}
and
\begin{equation}\label{eqn:step2result asn}
\begin{aligned}
  &\P\bigg\{\sup_{\btheta\in\overline{\mathcal{B}}(\btheta_0,r_n)}
  \big|\P_ng(\cdot;\btheta)-\frac{1}
  {\sqrt{n}}(\btheta-\btheta_0)^\top \bwn\big|>
     4c_{\max}\rho(r_n) \norm{\btheta-\btheta_0}^2+
     \\ &\hspace{23em} \frac{6\nu_0r_n }{\sqrt{n}} d_p(\varepsilon) \norm{\btheta-\btheta_0} \bigg\}\leq\varepsilon.
\end{aligned}
\end{equation}

Since $\F$ is uniformly bounded and $\sup_{\overline{\mathcal{B}}(\btheta_0,r_n)}\E  h^2(\cdot,\cdot;\btheta)\leq \tilde C\epsilon_n$, Theorem \ref{thm:uniformloose}(ii) implies that there exists a   constant $C_{\varepsilon}$  such that
        \begin{equation}\label{eqn:step3result asn}
       \P\bigg\{ \sup_{\btheta\in\overline{\mathcal{B}}(\btheta_0,r_n)}\big|\U_nh(\cdot,\cdot;\btheta)\big| > C_{\varepsilon}\log(1/\eta_n)\eta_n^{1/2}\frac{\nu}{n}\bigg\}\leq \varepsilon
        \end{equation}
        holds for sufficiently large $n$. This, together with \eqref{eqn:step1result asn} and \eqref{eqn:step2result asn}, implies that
 \begin{equation}\label{eqn:Gammantheta asn}
\begin{aligned}
  &\P\bigg\{\sup_{\btheta\in\overline{\mathcal{B}}(\btheta_0,r_n)}\big|\Gamma_n(\btheta)-\frac{1}{2}
  (\btheta-\btheta_0)^\top\Vb(\btheta-\btheta_0)-
  \frac{1}
  {\sqrt{n}}(\btheta-\btheta_0)^\top \bwn\big|>5c_{\max}\rho(r_n) \norm{\btheta-\btheta_0}^2+\\ & \hspace{9em} \frac{6\nu_0r_n }{\sqrt{n}} d_p(\varepsilon) \norm{\btheta-\btheta_0}+C_{\varepsilon}\log(1/\eta_n)\eta_n^{1/2}\frac{\nu}{n} \vphantom{\int_1^2} \bigg\}\leq2\varepsilon.
   \end{aligned}
\end{equation}
In view of $\norm{\btheta-\btheta_0}\leq r_n$, $\rho(r_n)\leq pr_n$ and $d_p(\varepsilon)\leq c_\varepsilon\sqrt{p}$, it holds that
\begin{equation}\label{eqn:Gammantheta asn}
\begin{aligned}
  &\P\bigg\{\sup_{\btheta\in\overline{\mathcal{B}}(\btheta_0,r_n)}\big|\Gamma_n(\btheta)-\frac{1}{2}
  (\btheta-\btheta_0)^\top\Vb(\btheta-\btheta_0)-
  \frac{1}
  {\sqrt{n}}(\btheta-\btheta_0)^\top \bwn\big|>5c_{\max}pr_n^3+\\ & \hspace{14em} \frac{6\nu_0c_\varepsilon r_n^2\sqrt{p} }{\sqrt{n}} +C_{\varepsilon}\log(1/\eta_n)\eta_n^{1/2}\frac{\nu}{n} \vphantom{\int_1^2} \bigg\}\leq2\varepsilon
   \end{aligned}
\end{equation}
for sufficiently large $n$. Define the set
\begin{equation}\label{eqn:Gammantheta asn set}
\begin{aligned}
  &\mathcal{A}''_{n,\varepsilon}=\bigg\{\bZ:
  \sup_{\btheta\in\overline{\mathcal{B}}(\btheta_0,r_n)}\big|\Gamma_n(\btheta)-\frac{1}{2}
  (\btheta-\btheta_0)^\top\Vb(\btheta-\btheta_0)-
  \frac{1}
  {\sqrt{n}}(\btheta-\btheta_0)^\top \bwn\big|\leq \phi_\varepsilon(\nu,p,n) \bigg\},
   \end{aligned}
\end{equation}
where
\[
\phi_\varepsilon(\nu,p,n):=5c_{\max}pr_n^3+ \frac{6\nu_0c_\varepsilon r_n^2\sqrt{p} }{\sqrt{n}} +C_{\varepsilon}\log(1/\eta_n)\eta_n^{1/2}\frac{\nu}{n}.
\]
Then, $\P(\mathcal{A}''_{n,\varepsilon})\geq 1-2\varepsilon$. Additionally, $\P(\mathcal{A}'_{n,\varepsilon} \cap \mathcal{A}''_{n,\varepsilon})\geq 1-4\varepsilon$. The following analysis is on the set $\mathcal{A}'_{n,\varepsilon} \cap \mathcal{A}''_{n,\varepsilon}$.

By definition, $\Gamma_n(\hat\btheta_n)=\Gamma_n(\btn/\sqrt{n}+\btheta_0)
\geq\Gamma_n(\btsn/\sqrt{n}+\btheta_0)$. Apply the inequality in \eqref{eqn:Gammantheta asn set} twice, then multiply through by $n$, consolidate terms, and use the fact that $\Vb$ is negative definite to get that
\begin{equation}\label{eqn: close}
  0\leq -\frac{1}{2}( \btn-\btsn)^\top\Vb( \btn-\btsn)\leq 2n\phi_\varepsilon(\nu,p,n).
\end{equation}
Note that $\phi_\varepsilon(\nu,p,n)\lesssim (\nu\vee p)^{5/2}/n^{3/2} +\log(1/\eta_n)\eta_n^{1/2}\nu/n$. This, combined with \eqref{eqn: close} and Assumption \ref{ass:taufunction}(ii), implies that
\begin{equation*}
  \norm{\btn-\btsn} \lesssim \bigg\{\frac{(\nu\vee p)^{5/2}}{n^{1/2}}\vee \log(1/\eta_n)\eta_n^{1/2}\nu\bigg\}^{1/2}.
\end{equation*}
Recall the definition of $\btn$ and $\btsn$, we immediately have
\begin{equation*}
 \big\|\hat\btheta_n-\btheta_0+\Vb^{-1}\P_n \nabla_1 \tau(\cdot;\btheta_0)\big\|^2 = O_\P\Big\{\frac{(\nu\vee p)^{5/2}}{n^{3/2}}+ \frac{\log(1/\eta_n)\eta_n^{1/2}\nu}{n}\Big\}.
\end{equation*}
Furthermore,
if $\{(\nu\vee p)^{5/2} /n^{1/2}\}\vee \{\log(1/\eta_n)\eta_n^{1/2}\nu\}\rightarrow 0$, then by Assumption \ref{ass:taufunction}(iv) and Slutsky's Theorem, it hold that for any $\bgamma\in\R^p$, $\bgamma^\top \btn/\{\bgamma^\top \Vb^{-1}\Deltab\Vb^{-1}\bgamma\}^{1/2}\Rightarrow N(0,1)$. This completes the proof.
\end{proof}

\subsubsection{Proof of Theorem \protect\ref{thm: consistency cov}}

\begin{proof}
  {\blue Note that the function class $\F$ is uniformly bounded by an absolute constant. We immediately have that $\tilde \F$ is also uniformly bounded by an absolute constant. In addition, $\E\tau_n(\bz;\btheta)=\tau(\bz;\btheta)$.
  It then follows from Lemma \ref{lem:ep} that
  \begin{equation}\label{eqn: uniform taun}
    \begin{aligned}
      \sup_{\R^m \otimes \Theta}|\tau_n(\bz;\btheta)-\tau(\bz;\btheta)| = O_\P\big(\sqrt{\tilde\nu/n}\big).
    \end{aligned}
  \end{equation}
  Since $\varepsilon_n^{-1}\sqrt{\tilde\nu/n}\rightarrow 0$, we just need to consider
  \begin{equation*}
    \begin{aligned}
      \tilde \delta_{ij} := \P_n\{\tilde p_{ni}(\cdot;\hat\btheta_n) \tilde p_{nj}(\cdot;\hat\btheta_n)\},
    \end{aligned}
  \end{equation*}
  where
  \begin{equation*}
    \begin{aligned}
      \tilde p_{ni}(\bz;\btheta) := \varepsilon_n^{-1}\{\tau(\bz;\btheta+\varepsilon_n\bu_i) - \tau(\bz;\btheta)\}.
    \end{aligned}
  \end{equation*}
  Expand $\tilde p_{ni}(\bz;\hat\btheta_n)$ about $\btheta_0$ to get
  \begin{equation}\label{eqn: tylor tilde p}
    \begin{aligned}
      \tilde p_{ni}(\bz;\hat\btheta_n) = \tilde p_{ni}(\bz;\btheta_0) + \varepsilon_n^{-1}(\hat\btheta_n - \btheta_0)^\top \underbrace{\{\nabla_1 \tau(\bz;\btheta^*+\varepsilon_n\bu_i) - \nabla_1\tau(\bz;\btheta^*)\}}_{R_n},
    \end{aligned}
  \end{equation}
  where $\btheta^*$ denotes some point between $\hat\btheta_n$ and $\btheta_0$.
  Note that $\tau(\bz;\btheta) = \zeta(\bz;\btheta) + \E\tau(\cdot;\btheta)$. We can rewrite $R_n$ in the above equation as follows:
  \begin{equation*}
    \begin{aligned}
      R_n = \underbrace{\{\nabla_1 \zeta(\bz;\btheta^* + \varepsilon_n\bu_i) - \nabla_1 \zeta(\bz;\btheta^*)\}}_{R_{n1}} +\underbrace{\{\nabla_1\E \tau(\cdot; \btheta^* + \varepsilon_n\bu_i) - \nabla_1 \E\tau(\cdot;\btheta^*)
      \}}_{R_{n2}}.
    \end{aligned}
  \end{equation*}
  We discuss $R_{n1}$ and $R_{n2}$ separately. First, following the calculations in {\it Step 2} of the proof of Theorem \ref{thm: general consrate}, we have
  \begin{equation*}
    \begin{aligned}
      \sup_{\btheta^*\in\overline{\mathcal{B}}(\btheta_0,r_n)}\norm{\P_n \{\nabla_1 \zeta(\cdot;\btheta^* + \varepsilon_n\bu_i) - \nabla_1 \zeta(\cdot;\btheta^*)\}} = O_{\P}\big(r_n\sqrt{p/n}\big),
    \end{aligned}
  \end{equation*}
  where $r_n:=\sqrt{(\nu \vee p)/n}$. In view of $ R_{n1} = \P_n \{\nabla_1 \zeta(\cdot;\btheta^* + \varepsilon_n\bu_i) - \nabla_1 \zeta(\cdot;\btheta^*)\}$, it then holds that
  \begin{equation}\label{eqn: Rn1}
  \begin{aligned}
  \sup_{\btheta^*\in\overline{\mathcal{B}}(\btheta_0,r_n)} \norm{R_{n1}} = O_{\P}\big(r_n\sqrt{p/n}\big).
  \end{aligned}
  \end{equation}
  We now turn to consider $R_{n2}$. Following similar arguments as in {\it Step 1} of the proof of Theorem \ref{thm: general consrate}, we have
  \begin{equation*}
    \begin{aligned}
      \sup_{\btheta^*\in\overline{\mathcal{B}}(\btheta_0,r_n)} \norm{R_{n2}} = O\big(r_n\rho(r_n)+\varepsilon_n\big)=O\big(pr_n^3 + \varepsilon_n\big).
    \end{aligned}
  \end{equation*}
  Combining this with \eqref{eqn: Rn1} and \eqref{eqn: tylor tilde p} implies that
  \begin{equation}\label{eqn: tilde p}
    \begin{aligned}
      \tilde p_{ni}(\bz;\hat\btheta_n) = \tilde p_{ni}(\bz;\btheta_0) + O_{\P}\big(\varepsilon_n^{-1}r_n\big) O_{\P}\big(r_n\sqrt{p/n}+ pr_n^3 + \varepsilon_n\big).
    \end{aligned}
  \end{equation}
%  Because $\varepsilon_n^{-1}r_n = o(1)$, $\varepsilon_n =o(1)$, and $(\nu\vee p)^{5/2}/n^{1/2}=o(1)$ by assumption, we then have that
%  \begin{equation}\label{eqn: tilde p}
%    \begin{aligned}
%      \tilde p_{ni}(\bz;\hat\btheta_n) = \tilde p_{ni}(\bz;\btheta_0) + o_{\P}(1).
%    \end{aligned}
%  \end{equation}
%  Thus, it suffices to
  Next, we consider
$\tilde p_{ni}(\cdot;\btheta_0) = \varepsilon_n^{-1}\{\tau(\bz;\btheta_0+\varepsilon_n\bu_i)-\tau(\bz;\btheta_0)\}$. Expand $\tau(\bz;\btheta_0+\varepsilon_n\bu_i)-\tau(\bz;\btheta_0)$ about $\varepsilon_n =0$ to get
\begin{equation}\label{eqn: tau epsilon}
  \begin{aligned}
    \tau(\bz;\btheta_0+\varepsilon_n\bu_i)-\tau(\bz;\btheta_0) = \varepsilon_n \bu_i^\top \nabla_1\tau(\bz;\btheta_0) +\varepsilon_n^2 \bu_i^\top \nabla_2\tau(\bz;\btheta_0 +\alpha\varepsilon_n\bu_i)\bu_i,
  \end{aligned}
\end{equation}
where $\alpha\in(0,1)$.
Again using the equality $\tau(\bz;\btheta) = \zeta(\bz;\btheta) + \E\tau(\cdot;\btheta)$, we have
\begin{equation}\label{eqn: Tn12}
  \begin{aligned}
    \bu_i^\top \nabla_2\tau(\bz;\btheta_0 +\alpha\varepsilon_n\bu_i)\bu_i = \underbrace{\bu_i^\top \Vb(\btheta_0 +\alpha\varepsilon_n\bu_i)\bu_i}_{T_{n1}} + \underbrace{\bu_i^\top \nabla_2\zeta(\bz;\btheta_0 +\alpha\varepsilon_n\bu_i)\bu_i}_{T_{n2}}.
  \end{aligned}
\end{equation}
By Assumption \ref{ass:taufunction}(ii) and (iii), we have
\begin{equation}\label{eqn: Tn1}
  \begin{aligned}
    \sup_{\alpha\in(0,1)}|T_{n1}| = \bu_i^\top \{ \Vb(\btheta_0 +\alpha\varepsilon_n\bu_i) - \Vb\}\bu_i + \bu_i^\top \Vb\bu_i=O(1).
  \end{aligned}
\end{equation}
By Assumption \ref{ass:taufunction}(v), we know that $T_{n2}$ is zero-mean subexponential. Thus, by the equivalent definitions of  zero-mean subexponential variables, it holds that
\begin{equation}\label{eqn: Tn2}
  \begin{aligned}
    \sup_{\alpha\in(0,1)}\E |T_{n2}| \leq \sup_{\alpha\in(0,1)}(\E T_{n2}^2)^{1/2}
  \end{aligned}
\end{equation}
is bounded. That is, $\sup_{\alpha\in(0,1)}|T_{n2}|=O_{\P}(1)$. Put \eqref{eqn: tau epsilon}--\eqref{eqn: Tn2} together. We then have
\begin{equation*}\label{eqn: tilde pni}
  \begin{aligned}
   \tilde p_{ni}(\bz;\btheta_0) = \bu_i^\top \nabla_1 \tau(\bz;\btheta_0) + O_{\P}(\varepsilon_n).
  \end{aligned}
\end{equation*}
This, combined with \eqref{eqn: tilde p}, implies that
\begin{equation*}
  \begin{aligned}
    \tilde p_{ni}(\bz;\hat\btheta_n) = \bu_i^\top \nabla_1 \tau(\bz;\btheta_0) + O_{\P}\big\{\varepsilon_n^{-1}\sqrt{\tilde v / n}+\varepsilon_n^{-1}r_n(r_n\sqrt{p/n}+ pr_n^3 + \varepsilon_n)+\varepsilon_n\big\}.
  \end{aligned}
\end{equation*}
Additionally, combining this with \eqref{eqn: uniform taun} implies that
\begin{equation*}
  \begin{aligned}
    \hat\delta_{ij} &= \P_n \{\bu_i^\top \nabla_1 \tau(\cdot;\btheta_0) \bu_j^\top \nabla_1 \tau(\cdot;\btheta_0)\} + O_{\P}\big[\{\varepsilon_n^{-1}\sqrt{\tilde v / n}+\varepsilon_n^{-1}r_n(r_n\sqrt{p/n}+ pr_n^3 + \varepsilon_n)+\varepsilon_n\}^2\big]\\
    &=\delta_{ij}+O_{\P}\big(1/\sqrt{n}\big)+O_{\P}\big[\{\varepsilon_n+\varepsilon_n^{-1}(r_n^2\sqrt{p/n}+ pr_n^4+\sqrt{\tilde v / n}) +r_n\}^2\big].
  \end{aligned}
\end{equation*}
Thus,
\begin{equation*}
  \begin{aligned}
    \norm{\hat\Deltab - \Deltab} = O_{\P}\big(p/\sqrt{n}\big)+O_{\P}\big[p\{\varepsilon_n+\varepsilon_n^{-1}(r_n^2\sqrt{p/n}+ pr_n^4 +\sqrt{\tilde v / n}) +r_n\}^2\big].
  \end{aligned}
\end{equation*}
Similarly,
\begin{equation*}
  \begin{aligned}
    \norm{\hat\Vb - \Vb} =O_{\P}\big(p/\sqrt{n}\big)+O_{\P}\big[p\{\varepsilon_n+\varepsilon_n^{-2}(r_n^2\sqrt{p/n}+ pr_n^4 +\sqrt{\tilde v / n}) +\varepsilon_n^{-1} r_n\}^2\big].
  \end{aligned}
\end{equation*}

By assumption, $(\tilde \nu \vee \nu\vee p)^{5/2} /n^{1/2} = o(1)$, $\varepsilon_n\sqrt{p}=o(1) $, and $\varepsilon_n^{-2}(\tilde \nu \vee \nu \vee p) /\sqrt{n}= o(1)$. It can then be easy to verify that
\begin{equation*}
  \begin{aligned}
    \norm{\hat\Deltab - \Deltab} = o_{\P}(1) ~~ \text{and} ~~
    \norm{\hat\Vb - \Vb}  = o_{\P}(1).
  \end{aligned}
\end{equation*}
This, combined with Assumption \ref{ass:taufunction}(ii) and (iv), implies that $\norm{\hat\Deltab} = O_{\P}(1)$, $\norm{\hat\Vb} = O_{\P}(1)$, and $\norm{\hat \Vb^{-1}} = O_{\P}(1)$. Note that $\hat\Vb^{-1} - \Vb^{-1} = \hat\Vb^{-1}(\Vb - \hat\Vb)\Vb^{-1}$. Then, we have
\begin{equation*}
  \begin{aligned}
    \norm{\hat\Vb^{-1}-\Vb^{-1}} \leq \norm{\hat\Vb^{-1}}\cdot\norm{\Vb - \hat\Vb}\cdot\norm{\Vb^{-1}} = o_{\P}(1).
  \end{aligned}
\end{equation*}
Note also that
\begin{equation*}
  \begin{aligned}
    &\hat\Vb^{-1}\hat\Deltab \hat\Vb^{-1} - \Vb^{-1} \Deltab\Vb^{-1} \\
    =  &(\hat\Vb^{-1} -\Vb^{-1})(\hat\Deltab - \Deltab)(\hat\Vb^{-1}-\Vb^{-1}) - \Vb^{-1}(\Deltab - \hat\Deltab)\hat\Vb^{-1}
    - \hat\Vb^{-1}\Deltab(\Vb^{-1} - \hat\Vb^{-1}) \\ ~~~ &- (\Vb^{-1} - \hat\Vb^{-1})\hat\Deltab\Vb^{-1}.
  \end{aligned}
\end{equation*}
Apply the triangle inequality to the above equation to get that
\begin{equation*}
  \begin{aligned}
    \norm{\hat\Vb^{-1}\hat\Deltab \hat\Vb^{-1} - \Vb^{-1} \Deltab\Vb^{-1}}= o_{\P}(1).
  \end{aligned}
\end{equation*}
This completes the proof.
%by the law of large numbers, we conclude that $\hat\delta_{ij}$ converges to the $(i,j)$ component $\delta_{ij}$ of $\Deltab$ in probability. Similarly, we can show that $p_{nij}(\bz;\btheta)$ converges in probability to the $(ij)$ second derivation of $\tau(\bz;\btheta)$ at $\btheta=\btheta_0$. Then, $\hat v_{ij}$ converges  to the $(i,j)$ component $v_{ij}$ of $\Vb$ in probability. This establishes the consistency of $\hat\Vb$. Finally, by the continuous mapping theorem, we complete the proof of this theorem.
  }
\end{proof}

\subsection{Proofs in Section \protect\ref{sec:application}}

For the example in Section \ref{subsec: HanMRC}, we define $\cF^\mathsf{H}%
=\{f^\mathsf{H}(\bz_1,\bz_2;\btheta):\btheta\in\Theta^\mathsf{H}\}$,
where $f^\mathsf{H}(\bz_1,\bz_2;\btheta)$ is defined in the main text, and
define $h^\mathsf{H}(\bz_1,\bz_2;\btheta) = f^\mathsf{H}(\bz_1,\bz_2;\bm%
\theta)-\mathbb{E} f^\mathsf{H}(\bz_1,\cdot;\btheta)- \mathbb{E} f^\mathsf{%
H}(\cdot,\bz_2;\btheta)+\Gamma^\mathsf{H}(\btheta)$. For the example in
Section \ref{subsec: example C}, we define $\cF^\mathsf{C}=\{f^\mathsf{C}(\bz%
_1,\bz_2;\btheta):\btheta\in\Theta^\mathsf{C}\}$ and $h^\mathsf{C}(\bz_1,%
\bz_2;\btheta)=f^\mathsf{C}(\bz_1,\bz_2;\btheta)-\mathbb{E} f^\mathsf{C}(%
\bz_1,\cdot;\btheta)-\mathbb{E} f^\mathsf{C}(\cdot,\bz_2;\btheta)+\Gamma^%
\mathsf{C}(\btheta)$. For the example in Section \ref{subsec: example K},
we define $\cF^\mathsf{K}=\{f^\mathsf{K}(\bz_1,\bz_2;\btheta):\bm%
\theta\in\Theta^\mathsf{K}\}$ and $h^\mathsf{K}(\bz_1,\bz_2;\btheta)=f^%
\mathsf{K}(\bz_1,\bz_2;\btheta)-\mathbb{E} f^\mathsf{K}(\bz_1,\cdot;\bm%
\theta)-\mathbb{E} f^\mathsf{K}(\cdot,\bz_2;\btheta)+\Gamma^\mathsf{K}(\bm%
\theta)$. For the example in Section \ref{subsec: example A}, we define $\cF^%
\mathsf{A}=\{f^\mathsf{A}(\bz_1,\bz_2;\btheta): \btheta\in\Theta^\mathsf{%
A}\big\}$ and $h^\mathsf{A}(\bz_1,\bz_2;\btheta)=f^\mathsf{A}(\bz_1,\bz_2;%
\btheta)-\mathbb{E} f^\mathsf{A}(\bz_1,\cdot;\btheta)-\mathbb{E} f^%
\mathsf{A}(\cdot,\bz_2;\btheta)+\mathbb{E} f^\mathsf{A}(\cdot,\cdot;\bm%
\theta)$.

\subsubsection{Some Additional Lemmas}

\begin{lemma}\label{lem:multsubgaussian}
  Suppose that Condition \ref{cond:H2} in the main text holds. Then $\Xlast-\E(\Xlast \mid \bX^\top\bbeta_0)$ is multivariate subgaussion.
\end{lemma}
\begin{proof}
  Fix $\bu\in \Sm^{p-1}$. Applying the triangle inequality yields that
  \begin{equation*}
    \norm{\bu^\top \{\Xlast-\E(\Xlast\mid \bX^\top\bbeta_0)\}}_r\leq \underbrace{\norm{\bu^\top\Xlast}_r}_{B_1}+\underbrace{\norm{\bu^\top \E(\Xlast\mid \bX^\top\bbeta_0)}_r}_{B_2}.
  \end{equation*}
In what follows, we discuss $B_1$ and $B_2$ separately. We first consider $B_1$:
\begin{equation}\label{eqn: resultforB1}
  B_1=\norm{(0,\bu^\top)\bX}_r\leq\sup_{\bv\in\Sm^{p}}\norm{\bv^\top \bX}_r.
\end{equation}
We then consider $B_2$:
\begin{equation*}
  \begin{aligned}
    B_2=\{\E|\E(\bu^\top \Xlast\big| \bX^\top\bbeta_0)|^r\}^{\frac{1}{r}}
    %\\&
    &\leq [\E\{\E(|\bu^\top \Xlast|\big| \bX^\top\bbeta_0)\}^r]^{\frac{1}{r}}\\
    %&\quad &\text{(By convexity of $|\cdot|$)}\\
    %&
    &\leq \{\E\E(|\bu^\top \Xlast|^r\big|  \bX^\top\bbeta_0)\}^{\frac{1}{r}}
    %&\quad &\text{(By convexity of $|\cdot|^r$ for $r\geq 1$)}\\
    %&
    = \norm{\bu^\top\Xlast}_r
    \leq \sup_{\bv\in\Sm^{p}}\norm{\bv^\top \bX}_r,
  \end{aligned}
\end{equation*}
where the second and third inequalities hold because of the convexity of $|\cdot|^r$ for $r\geq 1$.
This, combined with \eqref{eqn: resultforB1} and Condition \ref{cond:H2}, implies that
\begin{equation*}
\begin{aligned}
  \norm{\bu^\top (\Xlast-\E(\Xlast\mid \bX^\top\bbeta_0))}_{\psi_2} =&
  \sup_{r\geq 1}r^{-1/2}\E \norm{\bu^\top \{\Xlast-\E(\Xlast\mid \bX^\top\bbeta_0)\}}_r\\
  \leq& 2\sup_{\bv\in\Sm^p}\sup_{r\geq 1}r^{-1/2}\E\norm{\bv^\top \bX}_r
  = 2\sup_{\bv\in\Sm^{p}}\norm{\bv^\top \bX}_{\psi_2}\leq 2c'',
  \end{aligned}
\end{equation*}
which completes the proof.
\end{proof}
Next, we give the following lemma which establishes the upper bound for $%
\sup_{\btheta\in\overline{\mathcal{B}}(\btheta_0,r)}\mathbb{E} \{h^%
\mathsf{H}(\cdot,\cdot;\btheta)\}^2$. 
\begin{lemma}\label{lem: fastrate H}
 Under Assumptions~\ref{ass:inftynormfinite} and \ref{ass:conditionaldensityfinite} in the main text, then  for any small $r>0$ with $\overline{\mathcal{B}}(\btheta_0,r)\subset \Theta^\sfH$, $\sup_{\btheta\in\overline{\mathcal{B}}(\btheta_0,r)}\E \{h^\sfH(\cdot,\cdot;\btheta)\}^2\lesssim \sqrt{p}\norm{\btheta-\btheta_0}_2$.
 %where $C$ is some constant depending on $\sup_{i=2,\cdots,p}\E|X_i|^2$ and $C_0$, but not on $r$.
  \end{lemma}
\begin{proof}
    Write $H(\btheta)=\E\{h^\sfH(\cdot,\cdot;\btheta)\}^2$. Substitute the equation for $h^\sfH(\bz_1,\bz_2;\btheta)$ into  $H(\btheta)$ and consolidate terms to get that
    \begin{equation*}
    \begin{aligned}
    H(\btheta) =& \underbrace{\E \{f^\sfH(\cdot,\cdot;\btheta)\}^2}_{H_1(\btheta)}-
    \underbrace{\E\{\E_\P f^\sfH(\bZ_1,\cdot;\btheta)\}^2}_{H_2(\btheta)}-
    \underbrace{\E\{\E_\P f^\sfH(\cdot,\bZ_2;\theta)\}^2}_{H_3(\btheta)}+
    2\underbrace{\E\{\E_\P f^\sfH(\bZ_1,\cdot;\btheta)\E_\P f^\sfH(\cdot,\bZ_2;\btheta)\}}_{H_4(\btheta)}\\
    &\quad -
    \{\Gamma^\sfH(\btheta)\}^2.
    \end{aligned}
    \end{equation*}
    Fix $\btheta\in\overline{\mathcal{B}}(\btheta_0,r)$. Expand $H(\btheta)$ about $\btheta_0$ to get
    \begin{equation*}
      H(\btheta)=(\btheta-\btheta_0)^\top \nabla_1 H(\btheta'),
    \end{equation*}
    where $\btheta'$ is between $\btheta_0$ and $\btheta$. We wish to bound $\norm{\nabla_1 H(\btheta')}_{\infty}$. To that end, we discuss $H_j(\btheta)$ separately   for $j=1,\cdots,4$.  With a little abuse of notation, we still use $\btheta$ instead of $\btheta'$ below.

    We first consider $\nabla_1 H_1(\btheta)$. By the property of exchangeability between integration and derivation {\blue with $\nabla_1 H_1(\btheta)$}, we have
    \begin{equation*}
      \nabla_1 H_1(\btheta) = \nabla_1 \E\{f^\sfH(\cdot,\cdot;\btheta)\}^2 = \E [\nabla_1 \E_\P\{f^\sfH(\bZ_1,\cdot;\btheta)\}^2] =  \E\{\nabla_1 h_1(\bZ_1;\btheta)\},
    \end{equation*}
    where
    \begin{equation*}
      \begin{aligned}
         h_1(\bz;\btheta) =\E_\P\{f^\sfH(\bz,\bZ;\btheta)\}^2
         &= \E\{\indic(y>Y)\indic(\bx^\top \bbeta>\bX^\top \bbeta)\} +
        \E\{\indic(y>Y)\indic(\bx^\top \bbeta_0>\bX^\top \bbeta_0)\}\\
        &\quad - 2\E\{\indic(y>Y)\indic(\bx^\top \bbeta>\bX^\top \bbeta)\indic(\bx^\top \bbeta_0>\bX^\top \bbeta_0) \}.
      \end{aligned}
    \end{equation*}
    Similarly, we can write $\nabla_1 H_2(\btheta)$, $\nabla_1 H_3(\btheta)$ and $\nabla_1 H_4(\btheta)$ respectively as
    \begin{equation*}
      \nabla_1 H_2(\btheta) = \E\{\nabla_1 h_2(\cdot;\btheta)\}, \ \nabla_1 H_3(\btheta) = \E\{\nabla_1 h_3(\cdot;\btheta)\}, \ \nabla_1 H_4(\btheta) = \E\{\nabla_1 h_4(\cdot,\cdot;\btheta)\},
    \end{equation*}
    where
    \begin{equation*}
      \begin{aligned}
        h_2(\bz;\btheta) = &[\E\{\indic(y>Y)\indic(\bx^\top \bbeta>\bX^\top \bbeta)\}]^2 +
        [\E\{\indic(y>Y)\indic(\bx^\top \bbeta_0>\bX^\top \bbeta_0)\}]^2\\
        &\ - 2\E\{\indic(y>Y)\indic(\bx^\top \bbeta>\bX^\top \bbeta)\}\cdot
        \E\{\indic(y>Y)\indic(\bx^\top \bbeta_0>\bX^\top \bbeta_0) \},\\
        h_3(\bz;\btheta) = &[\E\{\indic(Y>y)\indic(\bX^\top \bbeta>\bx^\top \bbeta)\}]^2 +
        [\E\{\indic(Y>y)\indic(\bX^\top \bbeta_0>\bx^\top \bbeta_0)\}]^2\\
        &\ - 2\E\{\indic(Y>y)\indic(\bX^\top \bbeta>\bx^\top \bbeta)\}\cdot
        \E\{\indic(Y>y)\indic(\bX^\top \bbeta_0>\bx^\top \bbeta_0) \},\\
      \end{aligned}
    \end{equation*}
    and
    \begin{equation*}
      \begin{aligned}
        h_4(\bz_1,\bz_2;\btheta) = &\E\{\indic(y_1>Y)I(\bx_1^\top \bbeta>\bX^\top \bbeta)\}\cdot
        \E\{\indic(Y>y_2)\indic(\bX^\top \bbeta>\bx_2^\top \bbeta)\}\\
        &\ - \E\{\indic(y_1>Y)\indic(\bx_1^\top \bbeta>\bX^\top \bbeta)\}\cdot
        \E\{\indic(Y>y_2)\indic(\bX^\top \bbeta_0>\bx_2^\top \bbeta_0) \}\\
        &\ - \E\{\indic(y_1>Y)\indic(\bx_1^\top \bbeta_0>\bX^\top \bbeta_0)\} \cdot \E\{\indic(Y>y_2)\indic(\bX^\top \bbeta>\bx_2^\top \bbeta)\}\\
        &\ + \E\{\indic(y_1>Y)\indic(\bx_1^\top \bbeta_0>\bX^\top \bbeta_0)\} \cdot\E\{\indic(Y>y_2)\indic(\bX^\top \bbeta_0>\bx_2^\top \bbeta_0) \}.
      \end{aligned}
    \end{equation*}
   Thus,  we can rewrite
    $\nabla_1 H(\btheta)$  as
    \begin{equation*}
      \nabla_1 H(\theta) = \E\{\nabla_1 h_1(\cdot;\btheta)\} - \E\{\nabla_1 h_2(\cdot;\btheta)\} - \E\{\nabla_1 h_3(\cdot;\btheta)\} + 2\E\{\nabla_1 h_4(\cdot,\cdot;\btheta)\} - \nabla_1 \{\Gamma^\sfH(\btheta)\}^2.
    \end{equation*}
   To simplify the {\blue expression forms of} the functions $h_j$ with $j=1,\cdots,4$, we introduce the following notations:
    \begin{equation*}
      \begin{aligned}
        \varphi_1(\bz;\btheta) &:= \E\{\indic(y>Y)\indic(\bx^\top \bbeta>\bX^\top \bbeta)\},\\
        \varphi_2(\bz,\btheta) &:=\E\{\indic(y>Y)\indic(\bx^\top \bbeta>\bX^\top \bbeta)\indic(\bx^\top \bbeta_0>\bX^\top \bbeta_0)\},\\
        \varphi_3(\bz) &:=\E\{\indic(y>Y)\indic(\bx^\top \bbeta_0>\bX^\top \bbeta_0)\},
      \end{aligned}
    \end{equation*}
    and
    \begin{equation*}
      \begin{aligned}
        \omega_1(\bz;\btheta)&:=\E\{\indic(Y>y)\indic(\bX^\top \bbeta>\bx^\top \bbeta)\},\\
        \omega_2(\bz;\btheta)&:=\E\{\indic(Y>y)\indic(\bX^\top \bbeta>\bx^\top \bbeta)\indic(\bX^\top \bbeta_0>\bx^\top \bbeta_0)\},\\
        \omega_3(\bz)&:=\E\{\indic(Y>y)\indic(\bX^\top \bbeta_0>\bx^\top \bbeta_0)\}.
      \end{aligned}
    \end{equation*}
    This, combined with that $\Gamma^\sfH(\btheta) = \E\{\varphi_1(\cdot;\btheta)-\varphi_3(\cdot)\}$, allows us to rewrite $\nabla_1 H(\btheta)$ as
    \begin{equation*}
      \begin{aligned}
        \nabla_1 H(\btheta) =& \E \nabla_1\{\varphi_1(\cdot;\btheta) - 2\varphi_2(\cdot;\btheta) + \varphi_3(\cdot) -\varphi_1^2(\cdot;\btheta)+2\varphi_1(\cdot;\btheta)\varphi_3(\cdot)-\varphi_3^2(\cdot)\\
        & \ -\omega_1^2(\cdot;\btheta)+2\omega_2(\cdot;\btheta)\omega_3(\cdot)-\omega_3^2(\cdot) + 2\varphi_1(\cdot;\btheta)\omega_1(\cdot;\btheta) - 2\varphi_1(\cdot;\btheta)
        \omega_3(\cdot)\\  & \ - 2\varphi_3(\cdot)\omega_1(\cdot;\btheta)+ 2\varphi_3(\cdot)\omega_3(\cdot)\} - 2\Gamma^\sfH(\btheta)\E\nabla_1\{\varphi_1(\cdot;\btheta)-\varphi_3(\cdot)\}.
      \end{aligned}
    \end{equation*}
    Since the functions $\varphi_j$, $\omega_j$ are all bounded, we just need to bound $\norm{\E|\nabla_1 \varphi_j|}_{\infty}$ and $\norm{\E|\nabla_1 \omega_j|}_{\infty}$  for $j=1,2,3$.

    We  first consider $\E|\nabla_1 \varphi_1(\cdot;\btheta)|$ and rewrite $\varphi_1(\bZ;\btheta)$ as follows:
    \begin{equation*}
      \begin{aligned}
        \varphi_1(\bZ;\btheta) = \int_{\bx^\top \bbeta<\bX^\top \bbeta} \rho_1(Y,\bx^\top \bbeta_0)G(\diff \bx),
      \end{aligned}
    \end{equation*}
    where $\rho_1(y,t) = \E\{\indic(y>Y)\mid \bX^\top\bbeta_0 = t\}$, and $G(\cdot)$ denotes the probability distribution of $\bX$.
    %For simplicity, we will use $\beta$ and $\beta_0$ to represent $\beta(\theta)$ and $\beta(0)$ respectively. Correspondingly, we write $\pi_1(Z,\beta) =\pi_1(Z,\theta)$.

    Let $\bu_i$ denote the unit vector in $\R^{p+1}$ with the $i$th component equal to one and let $\nabla_1^i$ denote the $i$th component of $\nabla_1$, where $i=2,\cdots,p+1$. By definition,
    \begin{equation*}
      \nabla_1^i \varphi_1(\bZ;\bbeta)=\lim_{\varepsilon\rightarrow 0} \varepsilon^{-1}\{\varphi_1(\bZ;\bbeta+\varepsilon \bu_i)-\varphi_1(\bZ;\bbeta)\}.
    \end{equation*}
    The term in brackets equals
    \begin{equation*}
      \int_{\bx^\top \bbeta<\bX^\top \bbeta +\varepsilon(X_i-x_i)}\rho_1(Y,\bx^\top\bbeta_0)G(\diff \bx)-
      \int_{\bx^\top\bbeta<\bX^\top \bbeta}\rho_1(Y,\bx^\top\bbeta_0)G(\diff \bx).
    \end{equation*}
    Change variables from $\bx=(x_1,\xlast)$ to $(\bx^\top\bbeta_0,\xlast)$, rearrange the terms in the fields of integration to get that
    \begin{equation*}
    \begin{aligned}
      &\int_{\bx^\top\bbeta_0<\bX^\top \bbeta - \xlast^\top (\btheta-\btheta_0) +\varepsilon(X_i-x_i)}\rho_1(Y,\bx^\top\bbeta_0)G(\diff \bx)-
      \int_{\bx^\top\bbeta_0<\bX^\top \bbeta- \xlast^\top (\btheta-\btheta_0)}\rho_1(Y,\bx^\top\bbeta_0)G(\diff \bx)\\
      =&\int\left\{\int_{\bX^\top\bbeta - \xlast^\top (\btheta-\btheta_0)}^{\bX^\top \bbeta - \xlast^\top (\btheta-\btheta_0) +\varepsilon(X_i-x_i)}\rho_1(Y,t)g_0(t\mid \xlast)\diff t\right\}G_{\Xlast}(\diff \xlast),
      \end{aligned}
    \end{equation*}
    where $G_{\Xlast}(\cdot)$ denotes the distribution of $\Xlast$. The inner integral equals
    \begin{equation*}
      \varepsilon(X_i-x_i)\rho_1\{Y,
      \bX^\top \bbeta - \xlast^\top (\btheta-\btheta_0)\}g_0\{\bX^\top \bbeta - \xlast^\top (\btheta-\btheta_0)\mid \xlast\}+|X_i-x_i|o(|\varepsilon|) \quad \text{as $\varepsilon\rightarrow 0 $.}
    \end{equation*}
    Integrate, then apply the moment condition $\sup_{i=2,\cdots,p+1}\E|X_i|\leq \sqrt{C}$ in Assumption \ref{ass:inftynormfinite} to see that
    \begin{equation*}
      \begin{aligned}
        \nabla_1^i\varphi_1(\bZ;\bbeta) = \int (X_i-x_i)\rho_1\{Y,\bX^\top \bbeta-\xlast^\top (\btheta-\btheta_0)\}g_0\{\bX^\top \bbeta-\xlast^\top (\btheta-\btheta_0)\mid \xlast\}G_{\Xlast}(\diff \xlast).
      \end{aligned}
    \end{equation*}
    Since $|\rho_1(y,t)|\leq 1$ and $g_0(\cdot\mid \xlast)\leq C_0$ by Assumption \ref{ass:conditionaldensityfinite}, it then holds that
    \begin{equation*}
      \begin{aligned}
        |\nabla_1^i \varphi_1(\bZ;\bbeta)|\leq C_0\int |X_i-x_i|G_{\Xlast}(\diff \xlast)\leq C_0(|X_i|+\E|X_i|).
      \end{aligned}
    \end{equation*}
Thus,
\begin{equation*}
  \sup_{i=2,\cdots,p+1}\E|\nabla_1^i \varphi_1(\bZ;\bbeta)|\leq 2C_0\sup_{i=2,\cdots,p+1}\E|X_i|\leq 2C_0\sqrt{C}.
\end{equation*}
Similarly,
\begin{equation*}
\begin{aligned}
  \sup_{i=2,\cdots,p+1}\E|\nabla_1^i \varphi_2(\bZ;\bbeta)|&\leq 2C_0\sup_{i=2,\cdots,p+1}\E|X_i|\leq 2C_0\sqrt{C},\\
  \sup_{i=2,\cdots,p+1}\E|\nabla_1^i \omega_1(\bZ;\bbeta)|&\leq 2C_0\sup_{i=2,\cdots,p+1}\E|X_i|\leq 2C_0\sqrt{C},\\
  \sup_{i=2,\cdots,p+1}\E|\nabla_1^i \omega_2(\bZ;\bbeta)|&\leq 2C_0\sup_{i=2,\cdots,p+1}\E|X_i|\leq 2C_0\sqrt{C}.
  \end{aligned}
\end{equation*}
Put all results together, and we have that
\begin{equation*}
  \norm{\nabla_1 H(\btheta)}_{\infty}\leq C_1 \sup_{i=2,\cdots,p+1}\E|X_i|\leq C_1\sqrt{C}
\end{equation*}
for some constant $C_1$ depending only on $C_0$.
Then
\begin{equation*}
  H(\btheta)=(\btheta-\btheta_0)^\top H(\btheta')\leq \norm{\btheta-\btheta_0}_1\norm{\nabla H(\btheta')}_{\infty}\leq C_2\norm{\btheta-\btheta_0}_1\leq C_2\sqrt{p}\norm{\btheta-\btheta_0}_2.
\end{equation*}
That is, $\sup_{\btheta\in\overline{\mathcal{B}}(\btheta_0,r)}\E \{h^\sfH(\cdot,\cdot;\btheta)\}^2\lesssim \sqrt{p}\norm{\btheta-\btheta_0}_2$. This completes the proof.
  \end{proof}

The next three lemmas give the upper bound for $\sup_{\btheta\in\overline{%
\mathcal{B}}(\btheta_0,r)}\mathbb{E} \{h^\mathsf{C}(\cdot,\cdot;\bm%
\theta)\}^2$, $\sup_{\btheta\in\overline{\mathcal{B}}(\btheta_0,r)}%
\mathbb{E} \{h^\mathsf{K}(\cdot,\cdot;\btheta)\}^2$, and $\sup_{\bm%
\theta\in\overline{\mathcal{B}}(\btheta_0,r)}\mathbb{E} \{h^\mathsf{A}%
(\cdot,\cdot;\btheta)\}^2$, respectively. Since the proofs of these lemmas
are similar to the proof of Lemma \ref{lem: fastrate H}, we omit the proofs
for simplicity. 
\begin{lemma}\label{lem: fastrate C}
 Suppose that Assumptions~\ref{ass:inftynormfinite}--\ref{ass:conditionaldensityfinite} in the main text hold. Then  for any small $r>0$ with $\overline{\mathcal{B}}(\btheta_0,r)\subset \Theta^\sfC$, $\sup_{\btheta\in\overline{\mathcal{B}}(\btheta_0,r)}\E \{h^\sfC(\cdot,\cdot;\btheta)\}^2\lesssim \sqrt{p}\norm{\btheta-\btheta_0}_2$.
 %where $C$ is some constant depending on $\sup_{i=2,\cdots,p}\E|X_i|^2$ and $C_0$, but not on $r$.
  \end{lemma}

\begin{lemma}\label{lem: fastrate K}
 Suppose that Assumptions~\ref{ass:inftynormfinite}--\ref{ass:conditionaldensityfinite} in the main text hold. Then  for any small $r>0$ with $\overline{\mathcal{B}}(\btheta_0,r)\subset \Theta^\sfK$, $\sup_{\btheta\in\overline{\mathcal{B}}(\btheta_0,r)}\E \{h^\sfK(\cdot,\cdot;\btheta)\}^2\lesssim \sqrt{p}\norm{\btheta-\btheta_0}_2$.
 %where $C$ is some constant depending on $\sup_{i=2,\cdots,p}\E|X_i|^2$ and $C_0$, but not on $r$.
  \end{lemma}
\begin{lemma}\label{lemma: fastrate A}
 Suppose that Assumptions~\ref{ass:inftynormfinite} and \ref{ass: cdfinite addw} in the main text hold. Then  for any small $r>0$ with $\overline{\mathcal{B}}(\btheta_0,r)\subset \Theta^\sfA$, $\sup_{\btheta\in\overline{\mathcal{B}}(\btheta_0,r)}\E \{h^\sfA(\cdot,\cdot;\btheta)\}^2\lesssim \sqrt{p}\norm{\btheta-\btheta_0}_2$.
 %where $C$ is some constant depending on $\sup_{i=2,\cdots,p}\E|X_i|^2$ and $C_0$, but not on $r$.
  \end{lemma}

\subsubsection{Proof of Corollary \protect\ref{cor:H}}

\begin{proof}
Note that $\F^\sfH$ is uniformly bounded.
  To prove Corollary \ref{cor:H}(i) and (ii), it suffices to show that the VC-dimension of $\F^\sfH$ is $\sim p$ by Theorems \ref{thm: consistency} and \ref{thm: general consrate}.

  To see this, define the following function:
  \begin{equation*}
    g(\bz_1,\bz_2,t;\gamma,\gamma_1,\gamma_2,\bdelta_1,\bdelta_2)=\gamma t+\gamma_1 y_1+\gamma_2 y_2+\bdelta_1^\top\bx+\bdelta_2^\top\bx,
  \end{equation*}
  and the following function class:
   \begin{equation*}
   \mathscr{G}=\{ g(\bz_1,\bz_2,t;\gamma,\gamma_1,\gamma_2,\bdelta_1,\bdelta_2):\gamma,\gamma_1,\gamma_2\in\R,
   \bdelta_1,\bdelta_2\in\R^{p+1}\}.
  \end{equation*}
  Note that $\mathscr{G}$ is a $(2p+5)$-dimensional vector space of real-valued functions. By Lemma 18 in \cite{pollard2012convergence} and Lemma 2.4 in \cite{pakes1989simulation}, $\{\mathscr{G}\geq s\}$ and $\{\mathscr{G}>s\}$ are VC-classes of VC-dimensions $2p+5$ for any $s\in\R$. We further have, for any $\btheta\in \Theta^\sfH$, $\bbeta=(1,\btheta^\top)^\top$, and $\bbeta_0=(1,\btheta_0^\top)^\top$,
  %\allowdisplaybreaks
  \begin{equation*}
    \begin{aligned}
      \textnormal{subgraph}\{f^\sfH(\cdot,\cdot;\btheta)\}=&\{(\bz_1,\bz_2,t)\in\S \otimes\S \otimes\R:
      t<f^\sfH(\bz_1,\bz_2;\btheta)\}\\
      =&\Big\{\{y_1-y_2>0\}\cap\{\bx_1^\top\bbeta-\bx_2^\top\bbeta>0\}\cap
      \{\bx_1^\top\bbeta_0-\bx_2^\top\bbeta_0>0\}^\sfc\cap\{t\geq 1\}^\sfc\Big\}
      \\&~~ \cup\Big\{\{y_1-y_2>0\}^\sfc\cap\{t\geq 0\}^\sfc\Big\} \cup
       \Big\{\{y_1-y_2>0\}\cap\{\bx_1^\top\bbeta-\bx_2^\top\bbeta>0\}^\sfc\cap \\&~~
      \{\bx_1^\top\bbeta_0-\bx_2^\top\bbeta_0>0\}\cap\{t\geq -1\}^\sfc\Big\}
      \\=&\Big\{\{g_1>0\}\cap\{g_2>0\}\cap
      \{g_3>0\}^\sfc\cap \{g_4\geq 1\}^\sfc\Big\} \cup \Big\{\{g_1>0\}^\sfc\cap\{g_4\geq 0\}^\sfc\Big\}\\&~~ \cup\Big\{\{g_1>0\}\cap\{g_2>0\}^\sfc\cap
      \{g_3>0\}\cap\{g_4\geq -1\}^\sfc\Big\}
    \end{aligned}
  \end{equation*}
  for $g_1,\ldots,g_4\in\mathscr{G}$. This, combined with Lemma 9.7 in \cite{kosorok2007introduction}, implies that $\F^\sfH$ is a VC-class of VC-dimension $\sim p$. Then,
  apply Theorems \ref{thm: consistency} and \ref{thm: general consrate} to complete the proof of Corollary \ref{cor:H}(i) and (ii).

  Next, we prove Corollary \ref{cor:H}(iii). By Lemma \ref{lem: fastrate H}, we see that for any $c>0$, $\sup_{\btheta\in\overline{\mathcal{B}}(\btheta_0,c\sqrt{p/n})}\E \\ \{h^\sfH(\cdot,\cdot;\btheta)\}^2\lesssim p/\sqrt{n}$ if $p/n\rightarrow 0$. Connecting this with Theorem \ref{thm:generalASN} implies that $\epsilon_n\sim p/\sqrt{n}$ and $\eta_n\sim p/\sqrt{n}$. Thus, by Theorem \ref{thm:generalASN}, we conclude that
  if $\log(n/p^2)p^{3/2}/n^{5/4}\rightarrow 0$, we have
 \begin{align}\label{eq:kiefer}
 \norm{\hat\btheta_n^\sfH-\btheta_0+(\Vb^\sfH)^{-1}\P_n \nabla_1 \tau^\sfH(\cdot;\btheta_0)}^2 = O_\P\big\{\log(n/p^2)p^{3/2}/n^{5/4}\big\}.
 \end{align}
 In particular,
  if $\log(n/p^2)p^{3/2}/n^{1/4}\rightarrow 0$, then for any $\bgamma\in\R^{p}$,
  \[
  \sqrt{n}\bgamma^\top(\hat\btheta^\sfH_n-\btheta_0) / \{\bgamma^\top (\Vb^\sfH)^{-1}\Deltab^\sfH (\Vb^\sfH)^{-1}\bgamma\}^{1/2}\Rightarrow N(0,1).
  \]
  This completes the proof.

  {\blue To prove Corollary \ref{cor:H}(iv), we only need to evaluate the order of $\tilde \nu^\sfH$, the VC-dimension of $\tilde\F^\sfH:=\{f^\sfH(\bz,\cdot;\btheta)+f^\sfH(\cdot,\bz;\btheta):\bz\in\R^{p+1},
  \btheta\in\Theta^\sfH\}$. Following similar arguments above, we can know that $\tilde \nu^\sfH$ is also of order $p$. Then, the claim in Corollary \ref{cor:H}(iv) follows from Theorem \ref{thm: consistency cov}.
  }
\end{proof}

\subsubsection{Proof of Corollary \protect\ref{cor:C}}

\begin{proof}
  Similar to the proof of Corollary \ref{cor:H}, it can be easy to show that the VC-dimensions of $\F^\sfC$ {\blue and $\tilde\F^\sfC := \{f^\sfC(\bz,\cdot;\btheta)+f^\sfC(\cdot,\bz;\btheta):\bz\in\R^{p+1},
  \btheta\in\Theta^\sfC\}$
  are}  both of order $ p$. This, combined with that  $\F^\sfC$ is uniformly bounded, proves Corollary \ref{cor:H}(i) and (ii) by Theorems \ref{thm: consistency} and
  \ref{thm: general consrate}. Corollary \ref{cor:H}(iii) follows from Lemma
  \ref{lem: fastrate C} and Theorem \ref{thm:generalASN}. {\blue Corollary \ref{cor:H}(iv) follows from Theorem \ref{thm: consistency cov}.}
\end{proof}

\subsubsection{Proof of Corollary \protect\ref{cor:K}}

\begin{proof}
  Similar to the proof of Corollary \ref{cor:H}, one could show that the VC-dimension of $\F^\sfK$ {\blue and $\tilde\F^\sfK := \{f^\sfK(\bz,\cdot;\btheta)+f^\sfK(\cdot,\bz;\btheta):\bz\in\R^{p+1},
  \btheta\in\Theta^\sfK\}$
  are}  both of order $ p$. Then, the proofs of Corollary \ref{cor:K} (i) and (ii) follow directly from the proof of Corollary \ref{cor:H}. Finally, Lemma \ref{lem: fastrate K}, together with Theorem
  \ref{thm:generalASN} imply Corollary \ref{cor:K}(iii). {\blue Corollary \ref{cor:K}(iv) follows from Theorem \ref{thm: consistency cov}.}
\end{proof}

\subsubsection{Proof of Corollary \protect\ref{cor:A}}

\begin{proof}
  (i) Similar to the proof of Theorem \ref{thm: consistency}, the proof is twofold. We first show that   $\Gamma_n^\sfA(\btheta)$ converges in probability to $\Gamma^\sfA(\btheta)$ uniformly in $\btheta\in\Theta^\sfA$, and then establish the consistency of $\hat\btheta_n^\sfA$.

  \noindent
  {\it Step 1.}
  %Let $\F^\sfA=\big\{\indic(Y_1>Y_2)\indic(\bX_1^\top\bbeta>\bX_2^\top\bbeta)K((W_1-W_2)/b):
  %\btheta\in\Theta^\sfA\big\}$.
  Since $K(\cdot)$ is continuously differential with compact support by Assumption \ref{ass:gen to A}(vi), $K(\cdot)$ is bounded and is also a function of bounded variation. Thus, $K(\cdot)$ can be written as $K(\cdot)=K_1(\cdot)-K_2(\cdot)$ with appropriate bounded and monotone functions $K_1(\cdot)$ and $K_2(\cdot)$. Let $C_1$ and $C_2$ denote the upper bounds of $|K_1(\cdot)|$ and $|K_2(\cdot)|$ respectively.

  Let $\F_1^\sfA =\big\{\indic(Y_1>Y_2)\indic(\bX_1^\top\bbeta>\bX_2^\top\bbeta)K_1\{(W_1-W_2)/b\}:
  \btheta\in\Theta^\sfA\big\}$ and  $\F_2^\sfA =\big\{\indic(Y_1>Y_2)\indic(\bX_1^\top\bbeta>\bX_2^\top\bbeta)K_2\{(W_1-W_2)/b\}:
  \btheta\in\Theta^\sfA\big\}$.  Then, $\F^\sfA = \F_1^\sfA-\F_2^\sfA$. Similar to the proof of Corollary \ref{cor:H},
  it can be easy to verify that  the VC-dimensions of $\F_1^\sfA$ and $\F_2^\sfA$ are both $\sim p$ by considering the class of  subgraphs of all functions in $\F_1^\sfA$ and $\F_2^\sfA$ separately.
  By Lemma 16 in \cite{nolan1987u}, the covering number of $\F^\sfA$ is bounded through
  $N_r(\varepsilon,\P\otimes\P,\F^\sfA,C_1+C_2)\leq N_r(\varepsilon/4, \P\otimes\P,\F_1^\sfA,C_1)N_r(\varepsilon/4,\P\otimes\P,\F_2^\sfA,C_2)$.
This, combined  with Theorem 9.3 in \cite{kosorok2007introduction}, Lemma \ref{lem:ep}, \ref{lem:uniform as convergence}, and Hoeffding decomposition implies that
  \begin{equation}\label{eqn:first result A}
    \sup_{\btheta\in\Theta^\sfA}\big |\Gamma_n^\sfA(\btheta)-\E\Gamma_n^\sfA(\btheta)\big |=
    O_\P\Big(\frac{\sqrt{p}}{b\sqrt{n}}\Big)=O_\P\Big(\sqrt{\frac{p}{n^{1-2\delta}}}\Big) = o_\P(1).
  \end{equation}

Next, we try to bound $\sup_{\btheta\in\Theta^\sfA}\big|\E\{\Gamma_n^\sfA(\btheta)\} -\Gamma^\sfA(\btheta)\big |$. Note that
\begin{equation}\label{eqn:middle result A}
  \begin{aligned}
    \E\Gamma_n^\sfA(\btheta)&=\E[\{m(\bZ_1,\bZ_2;\btheta)-m(\bZ_1,\bZ_2;\btheta_0)\}K_b(W_1-W_2)]\\
    &=\E[\E\{m(\bZ_1,\bZ_2;\btheta)-m(\bZ_1,\bZ_2;\btheta_0)\mid W_1, W_2\}K_b(W_1-W_2)]\\
    &=\frac{1}{b}\int \int \psi(w_1,w_2;\btheta)K\Big(\frac{w_1-w_2}{b}\Big)\phi(w_1)\phi(w_2)\diff w_1 \diff w_2\\
    &=\int \int \psi(bu+w_2,w_2;\btheta)\phi(bu+w_2)K(u)\phi(w_2)\diff u \diff w_2.
  \end{aligned}
\end{equation}
A $J$th-order Tylor expansion of $\E\{\Gamma_n^\sfA(\btheta)\}$ with respect to $b$ at 0 and Assumptions \ref{ass:gen to A}(vi)--(viii) imply that
\begin{equation}\label{eqn:second result A}
  \sup_{\btheta\in\Theta^\sfA}\big|\E\Gamma_n^\sfA(\btheta) -\Gamma^\sfA(\btheta)\big | \lesssim b^J =o(1).
\end{equation}
This, combined with \eqref{eqn:first result A} and the triangular inequality, implies that
\begin{equation*}
  \sup_{\btheta\in\Theta^\sfA}\big|\Gamma_n^\sfA(\btheta) -\Gamma^\sfA(\btheta)\big |  =o_\P(1).
\end{equation*}
Thus, the uniform convergence of $\Gamma_n^\sfA(\btheta)$ is shown.

\noindent
{\it Step 2.} Following {\it Step 2} in the proof of Theorem \ref{thm: consistency}, it can be easy to show that $\norm{\hat\btheta_n^\sfA-\btheta_0}\xrightarrow{\P} 0$. This completes  proof of Corollary \ref{cor:A}(i).

(ii) Similar to the proof of Theorem \ref{thm: general consrate}, the proof is conducted in four steps. We first define $f_n^\sfA(\bz_1,\bz_2;\btheta) = f^\sfA(\bz_1,\bz_2;\btheta)/b$. Thus, $\E\Gamma_n^\sfA(\btheta) = \E f_n^\sfA(\cdot,\cdot;\btheta)$.
By a Hoeffding decomposition of $\Gamma_n^\sfA(\btheta)$, we have
\begin{equation*}
  \Gamma_n^\sfA(\btheta) = \E\Gamma_n^\sfA(\btheta)+\P_n g_n^\sfA(\cdot;\btheta)+\U_n h_n^\sfA(\cdot,\cdot;\btheta),
\end{equation*}
where
\begin{equation*}
 g_n^\sfA(\bz;\btheta)=\E f_n^\sfA(\bz,\cdot;\btheta) + \E f_n^\sfA(\cdot,\bz;\btheta)-2\E\Gamma_n^\sfA(\btheta),
\end{equation*}
and
\begin{equation*}
  h_n^\sfA(\bz_1,\bz_2;\btheta)=f_n^\sfA(\bz_1,\bz_2;\btheta)-\E f_n^\sfA(\bz_1,\cdot;\btheta)-\E f_n^\sfA(\cdot,\bz_2;\btheta)+\E\Gamma_n^\sfA(\btheta).
\end{equation*}
The first three steps aim to establish  bounds {\blue that are} similar to \eqref{eqn:step1result}, \eqref{eqn:step2result} and \eqref{eqn:step3result}, respectively. The last step establishes the rate of convergence of $\hat\btheta_n^\sfA$.

\noindent
{\it Step 1.} We first consider $\E\Gamma_n^\sfA(\btheta)$. By \eqref{eqn:first result A}, there exists a constant $C>0$ such that
\begin{equation}\label{eqn: first result AA}
  \sup_{\btheta\in\Theta^\sfA}\big|\E\Gamma_n^\sfA(\btheta) -\Gamma^\sfA(\btheta)\big |\leq C b^J.
\end{equation}
Fix $\btheta\in\overline{\mathcal{B}}(\btheta_0,r)\subset \Theta^\sfA$. Similar to {\it Step 1} in the proof of Theorem \ref{thm: general consrate}, we have
\begin{equation}
\sup_{\btheta\in\overline{\mathcal{B}}(\btheta_0,r)}
\big|\Gamma(\btheta)-\frac{1}{2}(\btheta-\btheta_0)^\top \Vb^\sfA(\btheta-\btheta_0)\big|\leq c_{\max}\rho(r) \norm{\btheta-\btheta_0}^2.
\end{equation}
This, combined with \eqref{eqn: first result AA}, implies that
\begin{equation}\label{eqn:step1result A}
\sup_{\btheta\in\overline{\mathcal{B}}(\btheta_0,r)}
\big|\E\Gamma_n^\sfA(\btheta)-\frac{1}{2}(\btheta-\btheta_0)^\top \Vb^\sfA(\btheta-\btheta_0)\big|\leq c_{\max}\rho(r) \norm{\btheta-\btheta_0}^2 + C b^J.
\end{equation}

\noindent
{\it Step 2.} Similar to \eqref{eqn:middle result A}, a change of variables and a $J$th-order Tylor expansion imply that
\begin{equation*}
  |\P_n g_n^\sfA(\cdot;\btheta)-\P_n\{\tau^\sfA(\cdot;\btheta)-\tau^\sfA(\cdot;\btheta_0)-
  2\E\Gamma_n^\sfA(\btheta)\}|\leq C' b^J
\end{equation*}
for some constant $C'>0$. This, combined with \eqref{eqn: first result AA}, implies that
\begin{equation}\label{eqn:step2 firstresult A}
\begin{aligned}
  |\P_n g_n^\sfA(\cdot;\btheta)-\P_n\{\tau^\sfA(\cdot;\btheta)-\tau^\sfA(\cdot;\btheta_0)-
  2\Gamma^\sfA(\btheta)\}|\leq (C+C') b^J.
  \end{aligned}
\end{equation}
Following the proof  of Theorem \ref{thm: general consrate} in {\it Step 2}, we additionally have
\begin{equation*}
\begin{aligned}
  \P\bigg\{\sup_{\btheta\in\overline{\mathcal{B}}(\btheta_0,r)}\big|\P_n\{\tau^\sfA(\cdot;\btheta)-\tau^\sfA(\cdot;\btheta_0)-
  2\Gamma^\sfA(\btheta)\}-\frac{1}
  {\sqrt{n}}(\btheta-\btheta_0)^\top \bwn^\sfA\big|>
     4c_{\max}\rho(r) \norm{\btheta-\btheta_0}^2+\\
     \frac{6\nu_0r }{\sqrt{n}} d_p(\varepsilon) \norm{\btheta-\btheta_0} \bigg\}\leq\varepsilon,
     \end{aligned}
\end{equation*}
 where
$\bwn^\sfA=\sqrt{n}\P_n\nabla_1 \tau^\sfA(\cdot;\btheta_0)$. Combining this with \eqref{eqn:step2 firstresult A} implies that
\begin{equation}\label{eqn:step2result A}
\begin{aligned}
  \P\bigg\{\sup_{\btheta\in\overline{\mathcal{B}}(\btheta_0,r)}\big|\P_ng_n^\sfA(\cdot;\btheta)-\frac{1}
  {\sqrt{n}}(\btheta-\btheta_0)^\top \bwn^\sfA\big|>(C+C') b^J+
     2c_{\max}\rho(r) \norm{\btheta-\btheta_0}^2+\\
     \frac{6\nu_0r }{\sqrt{n}} d_p(\varepsilon) \norm{\btheta-\btheta_0} \bigg\}\leq\varepsilon,
     \end{aligned}
\end{equation}

\noindent
{\it Step 3.} Following the proof of Theorem \ref{thm:uniformloose}(i), one can get
\begin{equation*}
  \P\bigg\{ \sup_{\btheta\in\overline{\mathcal{B}}(\btheta_0,r)}|\U_n h^\sfA(\cdot,\cdot;\btheta)|> \delta_n p/n\bigg\}\leq \epsilon,
\end{equation*}
where $\delta_n$ is a sequence of nonnegative real numbers converging to zero.
Thus,
\begin{equation}\label{eqn:step3result A}
  \P\bigg\{ \sup_{\btheta\in\overline{\mathcal{B}}(\btheta_0,r)}|\U_n h_n^\sfA(\cdot,\cdot;\btheta)|> \delta_n p/(bn)\bigg\}\leq \epsilon.
\end{equation}

\noindent
{\it Step 4.} By the Hoeffding decomposition of $\Gamma_n^\sfA(\btheta)$ and the results in \eqref{eqn:step1result A}, \eqref{eqn:step2result A} and \eqref{eqn:step3result A}, we have
\begin{equation}\label{eqn:Gammantheta}
\begin{aligned}
  &\P\bigg\{\sup_{\btheta\in\overline{\mathcal{B}}(\btheta_0,r)}\big|\Gamma_n^\sfA(\btheta)-\frac{1}{2}
  (\btheta-\btheta_0)^\top\Vb^\sfA(\btheta-\btheta_0)-
  \frac{1}
  {\sqrt{n}}(\btheta-\btheta_0)^\top \bwn^\sfA\big|>(2C+C') b^J+ \\ & \hspace{14em} 5c_{\max}\rho(r) \norm{\btheta-\btheta_0}^2+\frac{6\nu_0r }{\sqrt{n}} d_p(\varepsilon) \norm{\btheta-\btheta_0}+\delta_n\frac{p}{bn} \vphantom{\int_1^2} \bigg\}\leq2\varepsilon.
   \end{aligned}
\end{equation}
Then, following the proof of Theorem \ref{thm: general consrate} in {\it Step 4}, we conclude that there exists a sufficiently large constant $C'_\varepsilon>0$ such that
 \begin{equation}\label{eqn: step3 middle result A}
   \P\bigg\{\norm{\hat\btheta_n^\sfA-\btheta_0}\leq C'_\varepsilon \sqrt{\frac{p}{n^{1-\delta}}}\bigg\}\geq 1-3\varepsilon
\end{equation}
holds for sufficiently large $n$.

Since $\F^\sfA$ is uniformly bounded and $\sup_{\overline{\mathcal{B}}(\btheta_0,C'_\varepsilon\sqrt{p/n^{1-\delta}})}\E  \{h^\sfA(\cdot,\cdot;\btheta)\}^2\lesssim p/\sqrt{n^{1-\delta}}$ by Lemma \ref{lemma: fastrate A}, Theorem \ref{thm:uniformloose}(ii) implies that there exists a   constant $C''_{\varepsilon}$  such that
\begin{equation*}
    \P\bigg\{ \sup_{\btheta\in\overline{\mathcal{B}}(\btheta_0,C'_\varepsilon\sqrt{p/n^{1-\delta}})}
    \big|\U_nh^\sfA(\cdot,\cdot;\btheta)\big| > C''_{\varepsilon}\log(n^{1-\delta}/p^2)p^{\frac{3}{2}}/n^{\frac{5-\delta}{4}}\bigg\}\leq \varepsilon
 \end{equation*}
        holds for sufficiently large $n$. Thus,
 \begin{equation}\label{eqn:step3result fast A}
       \P\bigg\{ \sup_{\btheta\in\overline{\mathcal{B}}(\btheta_0,C'_\varepsilon\sqrt{p/n^{1-\delta}})}\big|\U_nh_n^\sfA(\cdot,\cdot;\btheta)\big| > C''_{\varepsilon}\log(n^{1-\delta}/p^2)p^{\frac{3}{2}}/n^{\frac{5(1-\delta)}{4}}\bigg\}\leq \varepsilon.
 \end{equation}
In view of $\delta<\frac{1}{5}$, it holds that $\{\log(n^{1-\delta}/p^2)/n^{\frac{5(1-\delta)}{4}}\}/(1/n)\rightarrow 0$ as $n\rightarrow \infty$. This, combined with \eqref{eqn:step3result fast A}, implies that
 \begin{equation}\label{eqn:step3result fast A}
       \P\bigg\{ \sup_{\btheta\in\overline{\mathcal{B}}(\btheta_0,C'_\varepsilon\sqrt{p/n^{1-\delta}})}\big|\U_nh_n^\sfA(\cdot,\cdot;\btheta)\big| > C''_{\varepsilon}p^{\frac{3}{2}}/n\bigg\}\leq \varepsilon.
 \end{equation}
 Based on similar analyses at the beginning of this step, we conclude that, there exists a sufficiently large constant $c'_\varepsilon>0$ such that
 \begin{equation*}
   \P\bigg\{\norm{\hat\btheta_n^\sfA-\btheta_0}\leq c'_\varepsilon \sqrt{\frac{p^{3/2}}{n}}\bigg\}\geq 1-6\varepsilon
\end{equation*}
holds for sufficiently large $n$. This, combined with \eqref{eqn: step3 middle result A}, implies that there exists a sufficiently large constant $C_\varepsilon>0$ such that
\begin{equation*}
   \P\bigg\{\norm{\hat\btheta_n^\sfA-\btheta_0}\leq C_\varepsilon \sqrt{\frac{p}{n^{1-\delta}}\wedge\frac{p^{3/2}}{n}}\bigg\}\geq 1-9\varepsilon.
\end{equation*}
This completes  proof of (ii).

(iii) Similar to the proof of Theorem \ref{thm:generalASN}, we first define  ${\btsn}^\sfA=-(\Vb^\sfA)^{-1}\bwn^\sfA$. Similarly, there exists a constant $c'_\varepsilon$ such that
\begin{equation}\label{eqn: normal consrate result A}
  \P\big(\mathcal{A}'_{n,\varepsilon}
  \big)\geq 1-2 \varepsilon
\end{equation}
holds for sufficiently large $n$, where  $\mathcal{A}'_{n,\varepsilon} :=\{\bZ:\hat\btheta^\sfA_n\in\overline{\mathcal{B}}(\btheta_0,r_n), {\btsn}^\sfA/\sqrt{n}+\btheta_0\in\overline{\mathcal{B}}(\btheta_0,r_n)\}$ and $r_n:=c'_\varepsilon\sqrt{p/n^{1-\delta}}$.

Fix $\btheta\in\overline{\mathcal{B}}(\btheta_0,r_n)$. Then  following the proofs of Corollary \ref{cor:A}(ii) in {\it Step 1--2}, we have
\begin{equation}\label{eqn:step1result asn A}
\sup_{\btheta\in\overline{\mathcal{B}}(\btheta_0,r_n)}
\big|\E\Gamma_n^\sfA(\btheta)-\frac{1}{2}(\btheta-\btheta_0)^\top \Vb^\sfA(\btheta-\btheta_0)\big|\leq c_{\max}\rho(r_n) \norm{\btheta-\btheta_0}^2 + C b^J,
\end{equation}
and
\begin{equation}\label{eqn:step2result asn A}
\begin{aligned}
  &\P\bigg\{\sup_{\btheta\in\overline{\mathcal{B}}(\btheta_0,r_n)}
 \big|\P_ng_n^\sfA(\cdot;\btheta)-\frac{1}
  {\sqrt{n}}(\btheta-\btheta_0)^\top \bwn^\sfA\big|>(C+C') b^J+
     4c_{\max}\rho(r_n) \norm{\btheta-\btheta_0}^2+\\ &\hspace{25em}
     \frac{6\nu_0r_n}{\sqrt{n}} d_p(\varepsilon) \norm{\btheta-\btheta_0} \bigg\}\leq\varepsilon.
\end{aligned}
\end{equation}
Similar to \eqref{eqn:step3result fast A}, we have
 \begin{equation}\label{eqn:step3result asn A}
       \P\bigg\{ \sup_{\btheta\in\overline{\mathcal{B}}(\btheta_0,r_n)}\big|\U_nh_n^\sfA(\cdot,\cdot;\btheta)\big| > C'_{\varepsilon}\log(n^{1-\delta}/p^2)p^{\frac{3}{2}}/n^{\frac{5(1-\delta)}{4}}\bigg\}\leq \varepsilon.
 \end{equation}
This, together with \eqref{eqn:step1result asn A} and \eqref{eqn:step2result asn A}, implies that
 \begin{equation}\label{eqn:Gammantheta asn A}
\begin{aligned}
  &\P\bigg\{\sup_{\btheta\in\overline{\mathcal{B}}(\btheta_0,r_n)}\big|\Gamma^\sfA_n(\btheta)-
  \frac{1}{2}
  (\btheta-\btheta_0)^\top\Vb^\sfA(\btheta-\btheta_0)-
   \frac{1}
  {\sqrt{n}}(\btheta-\btheta_0)^\top \bwn^\sfA\big|> (2C+C') b^J+\\ & \hspace{4em}
  5c_{\max}\rho(r_n) \norm{\btheta-\btheta_0}^2+ \frac{6\nu_0r_n }{\sqrt{n}} d_p(\varepsilon) \norm{\btheta-\btheta_0}+
  C'_{\varepsilon}\log(n^{1-\delta}/p^2)p^{\frac{3}{2}}/n^{\frac{5(1-\delta)}{4}}  \bigg\}\leq2\varepsilon.
   \end{aligned}
\end{equation}
The remaining proofs are straightforward and follow the proof of Theorem \ref{thm:generalASN}. In conclusion, if $\log(n^{1-\delta}/p^2)p^{3/2}/n^{(5-5\delta)/4}\to 0$, we have
 \[
 \norm{\hat\btheta_n^\sfA-\btheta_0+(\Vb^\sfA)^{-1}\P_n \nabla_1 \tau^\sfA(\cdot;\btheta_0)}^2 = O_\P\big\{n^{-\delta J}\vee \log(n^{1-\delta}/p^2)p^{3/2}/n^{(5-5\delta)/4}\big\}.
 \]
 In addition,
if $\log(n^{1-\delta}/p^2)p^{3/2}/n^{(1-5\delta)/4}\rightarrow 0$, then for any $\bgamma\in\R^{p}$,
  \[
\sqrt{n}\bgamma^\top(\hat\btheta^\sfA_n-\btheta_0) / \{\bgamma^\top (\Vb^\sfA)^{-1}\Deltab^\sfA (\Vb^\sfA)^{-1}\bgamma\}^{1/2}\Rightarrow N(0,1).
  \]
This completes the proof of (iii).

{\blue (iv) Note that $\E\tau_n^\sfA(\bz;\btheta)\neq\tau^\sfA(\bz;\btheta)$. The proof is a little different from that in proving Theorem \ref{thm: consistency cov}. To see this, define $\tilde \F^\sfA=\{ f^\sfA(\bz,\cdot;\btheta)+f^\sfA(\cdot,\bz;\btheta):\bz\in\R^{p+1},
  \btheta\in\Theta^\sfA\}$. Following the similar arguments in proof of (i), one can show that the VC-dimension of $\F^\sfA$ is of order $p$. It then follows from Lemma \ref{lem:ep} that \begin{equation*}%\label{eqn: uniform taun}
    \begin{aligned}
      \sup_{\R^m \otimes \Theta^\sfA}|\tau_n^\sfA(\bz;\btheta)-\E\tau_n^\sfA(\bz;\btheta)| = O_\P\big\{\sqrt{p}/(b\sqrt{n})\big\} = O_{\P}\big(p/n^{1-2\delta}\big).
    \end{aligned}
  \end{equation*}
  Similar to the derivations in \eqref{eqn:middle result A} and \eqref{eqn:second result A},  we have
  \begin{equation*}
    \begin{aligned}
      \sup_{\R^m \otimes \Theta^\sfA}|\E\tau_n^\sfA(\bz;\btheta)-\tau^\sfA(\bz;\btheta)| = O(b^J)=O(n^{-\delta J}).
    \end{aligned}
  \end{equation*}
  Then, following from the proof of Theorem \ref{thm: consistency cov}, we  get that
  \begin{equation*}
    \begin{aligned}
      \norm{\hat\Deltab^\sfA - \Deltab^\sfA} &= O_{\P}\big(p/\sqrt{n}\big)+O_{\P}\big[p\{\varepsilon_n+\varepsilon_n^{-1}(r_n^2\sqrt{p/n}+ pr_n^4 +\sqrt{p / n^{1-2\delta}} +n^{-\delta J}) +r_n\}^2\big],\\
      \norm{\hat\Vb^\sfA - \Vb^\sfA} &=O_{\P}\big(p/\sqrt{n}\big)+O_{\P}\big[p\{\varepsilon_n+\varepsilon_n^{-2}(r_n^2\sqrt{p/n}+ pr_n^4 +\sqrt{p / n^{1-2\delta}} +n^{-\delta J}) +\varepsilon_n^{-1} r_n\}^2\big],
    \end{aligned}
  \end{equation*}
  where $r_n = \sqrt{p/n^{1-\delta}}\wedge \sqrt{p^{3/2}/n}$.
  By assumption,  $\log(n^{1-\delta}/p^2)p^{3/2}/n^{(1-5\delta)/4}=o(1)$, $\varepsilon_n\sqrt{p} = o(1)$, and  $\varepsilon_n^{-2}p/\sqrt{n^{1-2\delta}}=o(1)$, one can show that
      \begin{equation*}
        \begin{aligned}
              \norm{\hat\Deltab^\sfA - \Deltab^\sfA} = o_{\P}(1), ~~ \text{and} ~~
    \norm{\hat\Vb^\sfA - \Vb^\sfA}  = o_{\P}(1).
        \end{aligned}
      \end{equation*}
      The remaining proof follows exactly from that in the proof of Theorem \ref{thm: consistency cov}.
}
\end{proof}

\subsection{Proof of Theorem \protect\ref{prop:H}}

\begin{proof}
We check Assumption \ref{ass:taufunction}(iii) and (v) separately under Conditions \ref{cond:H2}--\ref{cond:H1}.
\begin{enumerate}
%  \item[(i)] The claim in Assumption \ref{ass:taufunction}(i) is trivial.
%  \item[(ii)] By Theorem 4 in \cite{sherman1993limiting}, we have that
%  \begin{equation*}
%    \begin{aligned}
%      -\Vb^\sfH = -\frac{1}{2}\E\big[\{\Xlast-\E(\Xlast\mid \bX^\top \bbeta_0)\}\{\Xlast-\E(\Xlast\mid \bX^\top \bbeta_0)\}^\top \lambda_2^\sfH(Y,
%\bX^\top\bbeta_0)\big].
%    \end{aligned}
%  \end{equation*}
  \item[(iii)] The proof proceeds in two steps. We first calculate the third order mixed partial derivatives of $\E\tau^\sfH(\cdot;\cdot)$. Then we establish the bound of $\norm{\Vb^\sfH(\btheta) - \Vb^\sfH}$ for any $\btheta\in\overline{\mathcal{B}}(\btheta_0,r)$.

   \hf{  {\it Step 1.} Fix $\bz=(\bx,y)^T\in\R^m$ and $\btheta\in \overline{\mathcal{B}}(\btheta_0,r)$. Note that
  \begin{equation*}%\label{eqn: tau explicit}
    \begin{aligned}
      \tau^\sfH(\bz;\btheta)=\int_{-\infty}^{\bx^\top\bbeta}\int_{-\infty}^{y}g_0(t\mid s; \btheta)G_Y(\diff s)\diff t + \int_{\bx^\top\bbeta}^{\infty}\int_y^{\infty}g_0(t\mid s;\btheta)G_Y(\diff s)\diff t+C(\btheta_0),
    \end{aligned}
  \end{equation*}
  where  $G_Y(\cdot)$ denotes the marginal distribution of $Y$ and $g_0(\cdot\mid s;\btheta)$ denotes the conditional density function of $\bX^\top\bbeta$ given $Y =s$, $C(\btheta_0)$ is a term that does not depend on $\btheta$, and
  %Note also that
  \begin{equation*}
    \begin{aligned}
      g_0(t\mid s;\btheta) = \int g_0(t\mid s, \tilde \bx;\btheta)G_{\tilde \bX\mid Y=s}(\diff \tilde \bx)
       = \int f_0(t-\tilde\bx^\top\btheta\mid s, \tilde \bx;\btheta)G_{\tilde \bX\mid Y=s}(\diff \tilde \bx).
    \end{aligned}
  \end{equation*}
   For simplicity, we consider only the first part of $\tau^\sfH(\bz;\btheta)$ and denote
     \begin{equation*}%\label{eqn: tau explicit}
    \begin{aligned}
      \tau_1^\sfH(\bz;\btheta)=\int_{-\infty}^{\bx^\top\bbeta}\int_{-\infty}^{y}g_0(t\mid s; \btheta)G_Y(\diff s)\diff t.
    \end{aligned}
  \end{equation*}
   After some simple calculations, we have
   \begin{equation*}
     \begin{aligned}
       \frac{\partial \tau_1^\sfH(\bz;\btheta)}{\partial \theta_i} = \tilde x_i\bigg\{\int^y_{-\infty}g_0(\bx^\top\bbeta\mid s;\btheta)G_Y(\diff s) - \int_{-\infty}^{\bx^\top\bbeta}\int_{-\infty}^y g_{0,1}(t\mid s;\btheta) G_Y(\diff s)\diff t\bigg\},
     \end{aligned}
   \end{equation*}
   where
   \[
   g_{0,1}(t\mid s;\btheta) = \int f_0^{(1)}(t-\tilde \bx^\top\btheta\mid s,\tilde \bx)G_{\tilde\bX\mid Y=s}(\diff \tilde \bx).
   \]
   Additionally, we have
   \begin{equation*}
     \begin{aligned}
       \frac{\partial^3 \tau_1^\sfH(\bz;\btheta)}{\partial \theta_i\partial\theta_j\partial \theta_k} = \tilde x_i\tilde x_j\tilde x_k\bigg\{\underbrace{\int_{-\infty}^y g_{0,2}(\bx^\top\bbeta\mid s;\btheta)
       G_Y(\diff s) - \int_{-\infty}^{\bx^\top\bbeta} \int_{-\infty}^y g_{0,3}(t\mid s;\btheta) G_Y(\diff s)\diff t}_{A_1(\bx,y;\btheta)}\bigg\},
     \end{aligned}
   \end{equation*}
   where
   \[
   g_{0,m}(t\mid s;\btheta) = \int f_0^{(m)}(t-\tilde \bx^\top\btheta\mid s,\tilde \bx)G_{\tilde\bX\mid Y=s}(\diff \tilde \bx), \quad m = 2, 3.
   \]
   According to Condition \ref{cond: cond density X1}, we know that $A_1(\bx,y;\btheta)$ is uniformly upper bounded: $|A_1(\bx,y;\btheta)|\leq K$ for some absolute constant $K>0$. We could then similarly define $A_2(\bx,y;\btheta)$ for the second part and write $A(\bx,y;\btheta)=A_1(\bx,y;\btheta)+A_2(\bx,y;\btheta)$.
}

   {\it Step 2.} For any $\bgamma\in \Sm^{p-1}$, we consider $\bgamma^\top \{\Vb^\sfH(\btheta)-\Vb^\sfH\}\bgamma$. Expand $\bgamma^\top \{\Vb^\sfH(\btheta)-\Vb^\sfH\}\bgamma$ about $\btheta_0$ to get
  \begin{equation*}%\label{eqn: wl}
    \begin{aligned}
      \bgamma^\top \{\Vb^\sfH(\btheta)-\Vb^\sfH\}\bgamma = \sum_{i,j,k} \gamma_i\gamma_j(\theta_k-\theta_{0,k})\frac{\partial^3\E\tau^\sfH(\cdot;\btheta^*)}{\partial \theta_i\partial\theta_j\partial\theta_k}=\bgamma^\top \E\big\{A(\bX,Y;\btheta) \tilde \bX^\top(\btheta-\btheta_0) \tilde \bX^\top\tilde \bX\big\} \bgamma.
    \end{aligned}
  \end{equation*}
 Then,
\begin{align*}
      \sup_{\bgamma\in\Sm^{p-1}}|\bgamma^\top \{\Vb^\sfH(\btheta)-\Vb^\sfH\}\bgamma|\leq& 2K \big[\E\{\tilde \bX^\top (\btheta-\btheta_0)\}^2\big]^{1/2}\big\{\E(\bgamma^\top\tilde\bX)^4\big\}^{1/2}.
\end{align*}
  By Condition \ref{cond:H2}, we know that there exists an absolute constant $C$ such that
   \begin{equation*}
    \begin{aligned}
      \sup_{\bgamma\in\Sm^{p-1}}|\bgamma^\top \{\Vb^\sfH(\btheta)-\Vb^\sfH\}\bgamma|\leq KC\norm{\btheta-\btheta_0}\leq KCr_0.
    \end{aligned}
  \end{equation*}
   Then, we can choose $r_0$ small enough such that $KCr_0\leq c_{\min}/(11c_{\max})$. The first part of Assumption \ref{ass:taufunction}(iii) has been verified.

   Next, we try to verify the second part of Assumption \ref{ass:taufunction}(iii). According to the results in {\it Step 1}, we
   expand $\Vb^\sfH_{ij}(\btheta)-\Vb^\sfH_{ij}$ about $\btheta_0$ to get
  that
  \begin{equation*}
    \begin{aligned}
      \sup_{i,j}|\Vb_{ij}^\sfH(\btheta)-\Vb^\sfH_{ij}| \leq cr,
    \end{aligned}
  \end{equation*}
  where $c$ depends only on the absolute constants $K$ and $C'$.  Then, by the relationship between different matrix norms, we have that
  \begin{equation*}
    \begin{aligned}
      \norm{\Vb^\sfH(\btheta)-\Vb^\sfH}\leq \norm{\Vb^\sfH(\btheta)-\Vb^\sfH}_1\leq p \sup_{i,j}|\Vb^\sfH_{ij}(\btheta)-\Vb^\sfH_{ij}(\btheta)|\leq c pr.
    \end{aligned}
  \end{equation*}
  Finally,
  \begin{equation*}
    \begin{aligned}
      \norm{{\Ib}_p-(\Vb^\sfH)^{-1/2}\Vb^\sfH(\btheta)(\Vb^\sfH)^{-1/2}}\leq \norm{(\Vb^\sfH)^{-1/2}}\norm{\Vb^\sfH(\btheta)-\Vb^\sfH}\norm{(\Vb^\sfH)^{-1/2}}\leq \frac{cpr}{c_{\min}}.
    \end{aligned}
  \end{equation*}
  This completes the verification of Assumption \ref{ass:taufunction}(iii).
  %\item[(iv)]
  \item[(v)] We first consider  $\btheta=\btheta_0$. Since $\bX$ is multivariate subgaussian by Condition \ref{cond:H2}, it holds that $\sup_{i=1,\ldots,p+1}\E|X_i|^2\leq c_0$.
According to calculations in the proof of Theorem 4 in \cite{sherman1993limiting}, we have
\begin{equation*}
\nabla_2 \tau^\sfH(\bZ,\btheta_0)=\{\Xlast-\E(\Xlast\mid \bX^\top \bbeta_0)\}\{\Xlast-\E(\Xlast\mid \bX^\top \bbeta_0)\}^\top \lambda_2^\sfH(Y,
\bX^\top\bbeta_0).
\end{equation*}
 For any $\bgamma_1,\bgamma_2\in \Sm^{p-1}$, Lemma~\ref{lem:multsubgaussian} implies that under Conditions \ref{cond:H1} and \ref{cond:H2}, $\bgamma_1^\top\{\Xlast-\E(\Xlast\mid \bX^\top \bbeta_0)\}$ and $\bgamma_2^\top\{\Xlast-\E(\Xlast\mid \bX^\top \bbeta_0)\}\lambda_2^\sfH(Y,\bX^\top\bbeta_0)$ are both subgaussian with  subgaussian norms $2c'$ and $2c'c''$, respectively. Because the product of two subgaussian random variables is subexponential, $\bgamma_1^\top \nabla_2\tau(\bZ,\btheta_0)\bgamma_2$ is subexponential with a subexponential norm that depends only on $c'$ and $c''$. By the definition of subexponential variables and $\zeta^\sfH(\bz;\btheta_0)=\tau^\sfH(\bz;\btheta_0)-\E\{\tau^\sfH(\cdot;\btheta_0)\} = \tau^\sfH(\bz;\btheta_0)-\Vb$, we have
 \begin{equation}\label{eqn: subexponential}
 \begin{aligned}
   \E\exp\{\lambda\bgamma_1^\top \nabla_2\zeta^\sfH(\cdot,\btheta_0)\bgamma_2\}&= \E\exp[\lambda\bgamma_1^\top \{\nabla_2\tau^\sfH(\cdot,\btheta_0)-\Vb\}\bgamma_2]\\
   &\leq \exp[C_0\lambda^2\norm{\bgamma_1^\top \{\nabla_2\tau^\sfH(\bZ,\btheta_0)-\Vb\}\bgamma_2}_{\psi_1}^2]\\
   &\leq \exp\{4C_0\lambda^2\norm{\bgamma_1^\top \nabla_2\tau^\sfH(\bZ,\btheta_0)\bgamma_2}_{\psi_1}^2\}\\
   &\leq \exp(\nu_0^2\lambda^2/2), \quad \text{for $|\lambda|\leq \ell_0$,}
   \end{aligned}
 \end{equation}
where $\nu_0$ and $\ell_0$ are constants depend on constants $c_0,c',c''$. This shows that Assumption~\ref{ass:taufunction}(v) holds at $\btheta=\btheta_0$.

{\blue Note that there are several equivalent definitions for a generic zero-mean subexponential variable $U$. One of them is defined as follows:}
 there is a constant $c_1>0$ such that $\E\exp(\lambda U)$ is bounded for all $|\lambda|\leq c_1$. %$\bgamma_1^\top \nabla_2\zeta(\bZ,\btheta_0)\gamma_2$ is subexponential,
 %According to this definition and that $\bgamma_1^\top \nabla_2\zeta^\sfH(\bZ,\btheta_0)\gamma_2$ is subexponential,
 {\blue This definition implies that, for the subexponential variable $\bgamma_1^\top \nabla_2\zeta^\sfH(\bZ,\btheta_0)\gamma_2$,} there is a constant $c_2>0$ such that  $\E\exp\{\lambda\bgamma_1^\top \nabla_2\zeta^\sfH(\cdot,\btheta_0)\bgamma_2\}$ is bounded  for all $|\lambda|\leq c_2$.
Because $\E\exp\{\lambda\bgamma_1^\top \nabla_2\zeta^\sfH(\cdot,\btheta)\bgamma_2\}$ is a continuous function in $(\lambda,\btheta^\top)\in[-c_2,c_2]\otimes\overline{\mathcal{B}}(\btheta_0,r)$, and in addition that the domain of this function  is a compact set, it then holds  \[
\sup_{|\lambda|\leq c_2}\sup_{\btheta\in\overline{\mathcal{B}}(\btheta_0,r)} \E\exp\{\lambda\bgamma_1^\top \nabla_2\zeta^\sfH(\cdot,\btheta)\bgamma_2\}< C.
\]
Thus,
 $\bgamma_1^\top \nabla_2\zeta^\sfH(\cdot,\btheta)\bgamma_2$ is subexponential for any $\btheta\in\overline{\mathcal{B}}(\btheta_0,r)$. Similar to \eqref{eqn: subexponential}, we can establish the bound in Assumption~\ref{ass:taufunction}(v).
\end{enumerate}

This completes the proof.
\end{proof}

\bibliographystyle{apalike}
\bibliography{mybib}

\newpage{}

{\ \renewcommand{\tabcolsep}{5pt} \renewcommand{\arraystretch}{1.1} 
\begin{table}[]
\caption{Coverage probability under the first projection direction.}
\label{tab:1}\vspace{0.2cm} \centering
\begin{tabular}{llllllllllll}
\hline
\multicolumn{1}{l|}{\multirow{2}{*}{$n$}} & \multicolumn{1}{l|}{%
\multirow{2}{*}{$p$}} & \multicolumn{10}{c}{nominal coverage probability} \\ 
\cline{3-12}
\multicolumn{1}{l|}{} & \multicolumn{1}{l|}{} & 0.5 & 0.55 & 0.6 & 0.65 & 0.7
& 0.75 & 0.8 & 0.85 & 0.9 & 0.95 \\ \hline
\multicolumn{1}{l|}{\multirow{4}{*}{100}} & \multicolumn{1}{l|}{1} & 0.606 & 
0.644 & 0.692 & 0.731 & 0.781 & 0.822 & 0.860 & 0.890 & 0.914 & 0.932 \\ 
\multicolumn{1}{l|}{} & \multicolumn{1}{l|}{2} & 0.806 & 0.829 & 0.844 & 
0.862 & 0.881 & 0.895 & 0.907 & 0.920 & 0.930 & 0.948 \\ 
\multicolumn{1}{l|}{} & \multicolumn{1}{l|}{3} & 0.923 & 0.930 & 0.938 & 
0.947 & 0.953 & 0.957 & 0.963 & 0.964 & 0.970 & 0.973 \\ 
\multicolumn{1}{l|}{} & \multicolumn{1}{l|}{4} & 0.877 & 0.892 & 0.905 & 
0.920 & 0.926 & 0.939 & 0.945 & 0.948 & 0.956 & 0.964 \\ \hline
\multicolumn{1}{l|}{\multirow{4}{*}{200}} & \multicolumn{1}{l|}{1} & 0.518 & 
0.561 & 0.619 & 0.672 & 0.719 & 0.763 & 0.809 & 0.861 & 0.903 & 0.939 \\ 
\multicolumn{1}{l|}{} & \multicolumn{1}{l|}{2} & 0.598 & 0.655 & 0.704 & 
0.754 & 0.801 & 0.826 & 0.863 & 0.890 & 0.912 & 0.938 \\ 
\multicolumn{1}{l|}{} & \multicolumn{1}{l|}{3} & 0.702 & 0.746 & 0.788 & 
0.820 & 0.846 & 0.874 & 0.893 & 0.911 & 0.930 & 0.953 \\ 
\multicolumn{1}{l|}{} & \multicolumn{1}{l|}{4} & 0.852 & 0.871 & 0.887 & 
0.902 & 0.920 & 0.923 & 0.934 & 0.940 & 0.952 & 0.960 \\ \hline
\multicolumn{1}{l|}{\multirow{4}{*}{400}} & \multicolumn{1}{l|}{1} & 0.502 & 
0.552 & 0.588 & 0.648 & 0.699 & 0.749 & 0.797 & 0.857 & 0.900 & 0.946 \\ 
\multicolumn{1}{l|}{} & \multicolumn{1}{l|}{2} & 0.500 & 0.555 & 0.604 & 
0.663 & 0.724 & 0.766 & 0.819 & 0.858 & 0.905 & 0.945 \\ 
\multicolumn{1}{l|}{} & \multicolumn{1}{l|}{3} & 0.576 & 0.627 & 0.672 & 
0.715 & 0.765 & 0.809 & 0.844 & 0.882 & 0.900 & 0.929 \\ 
\multicolumn{1}{l|}{} & \multicolumn{1}{l|}{4} & 0.613 & 0.672 & 0.711 & 
0.737 & 0.782 & 0.833 & 0.870 & 0.890 & 0.920 & 0.944 \\ \hline
\end{tabular}%
\end{table}
}

{\ \renewcommand{\tabcolsep}{5pt} \renewcommand{\arraystretch}{1.1} 
\begin{table}[]
\caption{Coverage probability under the second projection direction.}
\label{tab:2}\vspace{0.2cm} \centering
\begin{tabular}{llllllllllll}
\hline
\multicolumn{1}{l|}{\multirow{2}{*}{$n$}} & \multicolumn{1}{l|}{%
\multirow{2}{*}{$p$}} & \multicolumn{10}{c}{nominal coverage probability} \\ 
\cline{3-12}
\multicolumn{1}{l|}{} & \multicolumn{1}{l|}{} & 0.5 & 0.55 & 0.6 & 0.65 & 0.7
& 0.75 & 0.8 & 0.85 & 0.9 & 0.95 \\ \hline
\multicolumn{1}{l|}{\multirow{4}{*}{100}} & \multicolumn{1}{l|}{1} & 0.606 & 
0.644 & 0.692 & 0.731 & 0.781 & 0.822 & 0.860 & 0.890 & 0.914 & 0.932 \\ 
\multicolumn{1}{l|}{} & \multicolumn{1}{l|}{2} & 0.790 & 0.820 & 0.839 & 
0.858 & 0.875 & 0.890 & 0.905 & 0.914 & 0.929 & 0.945 \\ 
\multicolumn{1}{l|}{} & \multicolumn{1}{l|}{3} & 0.920 & 0.928 & 0.938 & 
0.947 & 0.952 & 0.956 & 0.963 & 0.965 & 0.970 & 0.973 \\ 
\multicolumn{1}{l|}{} & \multicolumn{1}{l|}{4} & 0.876 & 0.890 & 0.903 & 
0.918 & 0.926 & 0.939 & 0.944 & 0.949 & 0.956 & 0.965 \\ \hline
\multicolumn{1}{l|}{\multirow{4}{*}{200}} & \multicolumn{1}{l|}{1} & 0.518 & 
0.561 & 0.619 & 0.672 & 0.719 & 0.763 & 0.809 & 0.861 & 0.903 & 0.939 \\ 
\multicolumn{1}{l|}{} & \multicolumn{1}{l|}{2} & 0.578 & 0.638 & 0.691 & 
0.732 & 0.773 & 0.810 & 0.857 & 0.883 & 0.909 & 0.934 \\ 
\multicolumn{1}{l|}{} & \multicolumn{1}{l|}{3} & 0.699 & 0.735 & 0.770 & 
0.801 & 0.831 & 0.869 & 0.889 & 0.912 & 0.929 & 0.947 \\ 
\multicolumn{1}{l|}{} & \multicolumn{1}{l|}{4} & 0.841 & 0.865 & 0.883 & 
0.900 & 0.911 & 0.919 & 0.932 & 0.943 & 0.952 & 0.958 \\ \hline
\multicolumn{1}{l|}{\multirow{2}{*}{400}} & \multicolumn{1}{l|}{1} & 0.502 & 
0.552 & 0.588 & 0.648 & 0.699 & 0.749 & 0.797 & 0.857 & 0.900 & 0.946 \\ 
\multicolumn{1}{l|}{} & \multicolumn{1}{l|}{2} & 0.519 & 0.573 & 0.623 & 
0.661 & 0.701 & 0.754 & 0.810 & 0.861 & 0.901 & 0.947 \\ 
\multicolumn{1}{l|}{} & \multicolumn{1}{l|}{3} & 0.568 & 0.615 & 0.673 & 
0.717 & 0.760 & 0.800 & 0.837 & 0.868 & 0.903 & 0.929 \\ 
\multicolumn{1}{l|}{} & \multicolumn{1}{l|}{4} & 0.592 & 0.637 & 0.675 & 
0.732 & 0.774 & 0.817 & 0.856 & 0.881 & 0.911 & 0.937 \\ \hline
\end{tabular}%
\end{table}
}

{\ \renewcommand{\tabcolsep}{5pt} \renewcommand{\arraystretch}{1.1} 
% \usepackage{multirow}
\begin{table}[]
\caption{Coverage probability under the third projection direction.}
\label{tab:3}\vspace{0.2cm} \centering
\begin{tabular}{llllllllllll}
\hline
\multicolumn{1}{l|}{\multirow{2}{*}{$n$}} & \multicolumn{1}{l|}{%
\multirow{2}{*}{$p$}} & \multicolumn{10}{c}{nominal coverage probability} \\ 
\cline{3-12}
\multicolumn{1}{l|}{} & \multicolumn{1}{l|}{} & 0.5 & 0.55 & 0.6 & 0.65 & 0.7
& 0.75 & 0.8 & 0.85 & 0.9 & 0.95 \\ \hline
\multicolumn{1}{l|}{\multirow{4}{*}{100}} & \multicolumn{1}{l|}{1} & 0.606 & 
0.644 & 0.692 & 0.731 & 0.781 & 0.822 & 0.860 & 0.890 & 0.914 & 0.932 \\ 
\multicolumn{1}{l|}{} & \multicolumn{1}{l|}{2} & 0.804 & 0.828 & 0.846 & 
0.861 & 0.880 & 0.897 & 0.907 & 0.921 & 0.931 & 0.948 \\ 
\multicolumn{1}{l|}{} & \multicolumn{1}{l|}{3} & 0.923 & 0.929 & 0.938 & 
0.947 & 0.953 & 0.957 & 0.963 & 0.964 & 0.970 & 0.974 \\ 
\multicolumn{1}{l|}{} & \multicolumn{1}{l|}{4} & 0.877 & 0.892 & 0.904 & 
0.920 & 0.926 & 0.939 & 0.945 & 0.948 & 0.956 & 0.964 \\ \hline
\multicolumn{1}{l|}{\multirow{4}{*}{200}} & \multicolumn{1}{l|}{1} & 0.518 & 
0.561 & 0.619 & 0.672 & 0.719 & 0.763 & 0.809 & 0.861 & 0.903 & 0.939 \\ 
\multicolumn{1}{l|}{} & \multicolumn{1}{l|}{2} & 0.601 & 0.658 & 0.710 & 
0.754 & 0.799 & 0.828 & 0.864 & 0.895 & 0.913 & 0.940 \\ 
\multicolumn{1}{l|}{} & \multicolumn{1}{l|}{3} & 0.712 & 0.749 & 0.787 & 
0.820 & 0.843 & 0.874 & 0.893 & 0.913 & 0.930 & 0.954 \\ 
\multicolumn{1}{l|}{} & \multicolumn{1}{l|}{4} & 0.852 & 0.870 & 0.886 & 
0.902 & 0.919 & 0.924 & 0.933 & 0.940 & 0.952 & 0.960 \\ \hline
\multicolumn{1}{l|}{\multirow{2}{*}{400}} & \multicolumn{1}{l|}{1} & 0.502 & 
0.552 & 0.588 & 0.648 & 0.699 & 0.749 & 0.797 & 0.857 & 0.900 & 0.946 \\ 
\multicolumn{1}{l|}{} & \multicolumn{1}{l|}{2} & 0.502 & 0.547 & 0.602 & 
0.661 & 0.720 & 0.771 & 0.813 & 0.861 & 0.908 & 0.944 \\ 
\multicolumn{1}{l|}{} & \multicolumn{1}{l|}{3} & 0.566 & 0.618 & 0.672 & 
0.720 & 0.765 & 0.808 & 0.844 & 0.881 & 0.902 & 0.931 \\ 
\multicolumn{1}{l|}{} & \multicolumn{1}{l|}{4} & 0.617 & 0.663 & 0.708 & 
0.738 & 0.789 & 0.835 & 0.871 & 0.892 & 0.920 & 0.946 \\ \hline
\end{tabular}%
\end{table}
}

\begin{figure}[]
\begin{center}
\begin{tabular}{cccc}
{\scriptsize $p=1$} & {\scriptsize $p=2$} & {\scriptsize $p=3$} & 
{\scriptsize $p=4$ } \\ 
\hskip-20pt \includegraphics[width=.23%
\textwidth,angle=0]{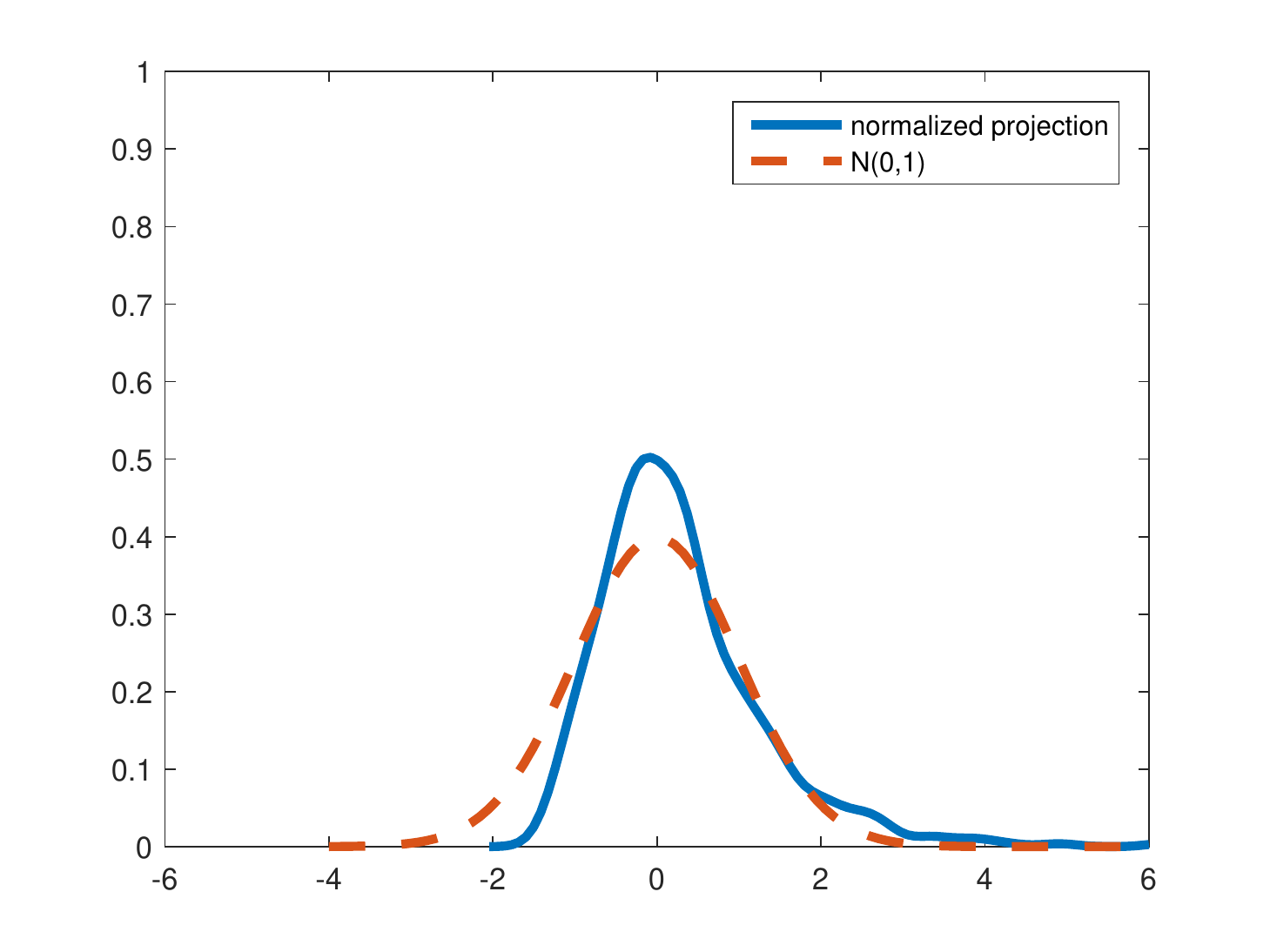} & \hskip-20pt %
\includegraphics[width=.23%
\textwidth,angle=0]{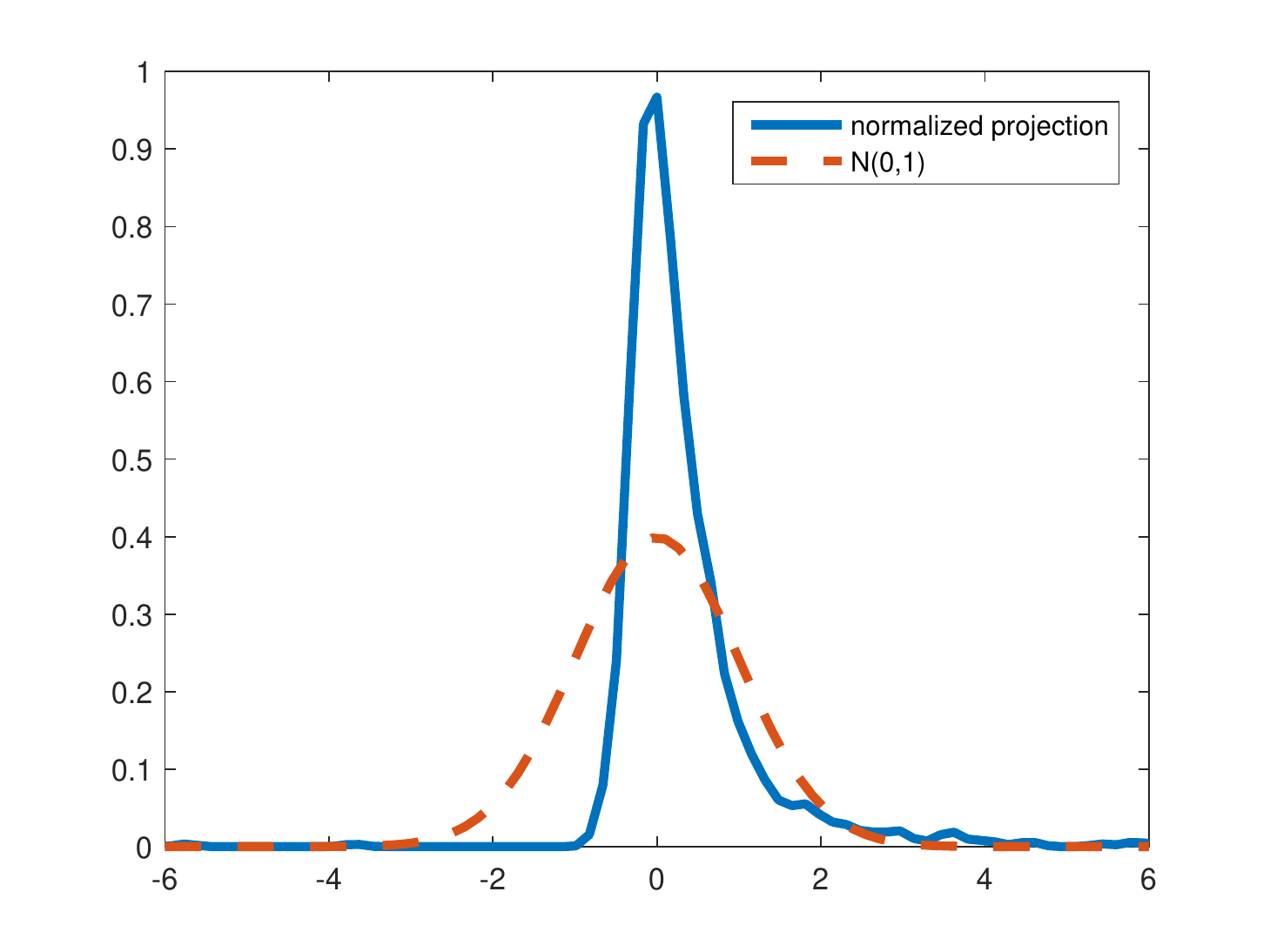} & \hskip-20pt %
\includegraphics[width=.23%
\textwidth,angle=0]{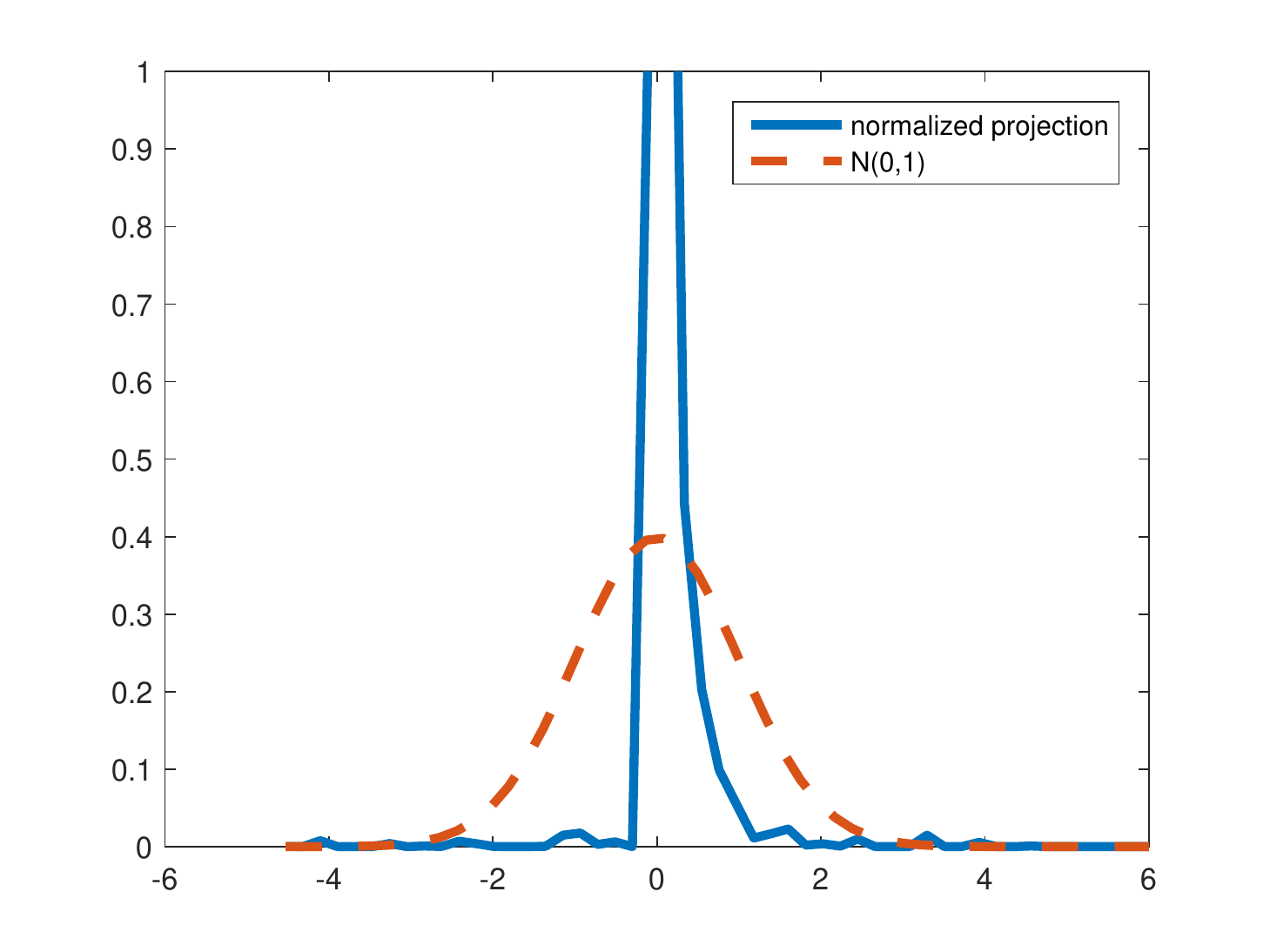} & \hskip-20pt %
\includegraphics[width=.23%
\textwidth,angle=0]{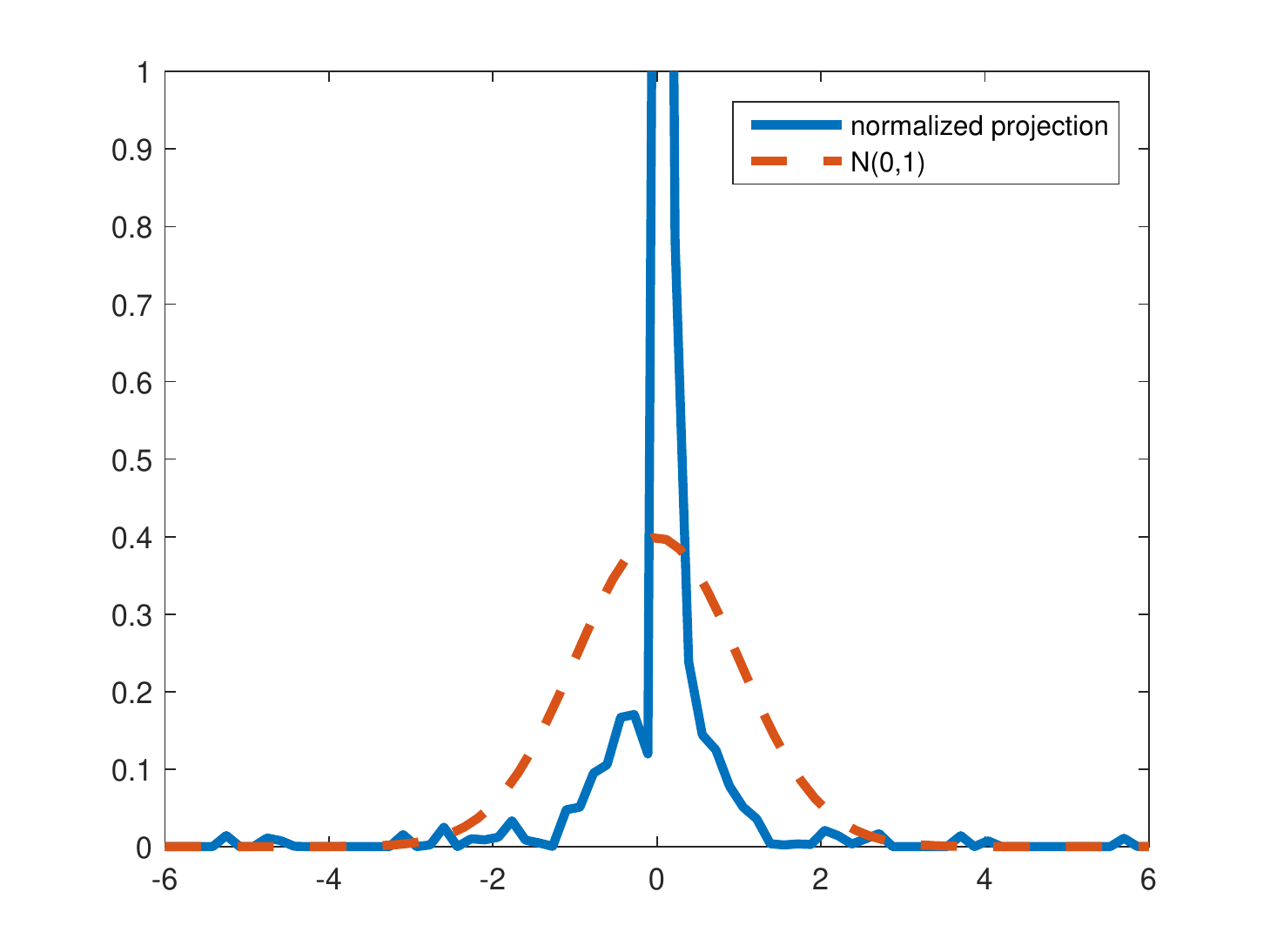} \\[-5pt] 
\hskip-20pt \includegraphics[width=.23%
\textwidth,angle=0]{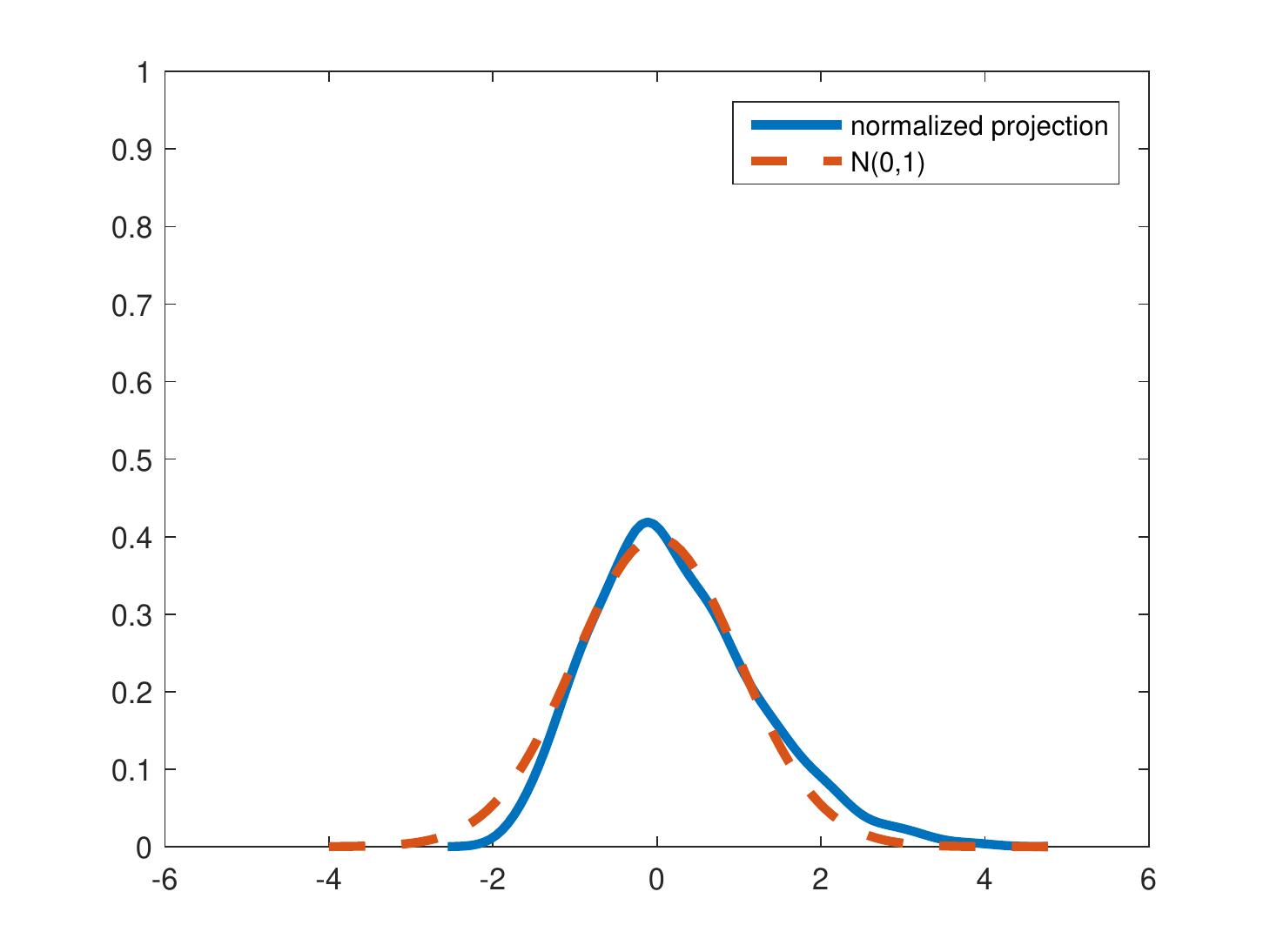} & \hskip-20pt %
\includegraphics[width=.23%
\textwidth,angle=0]{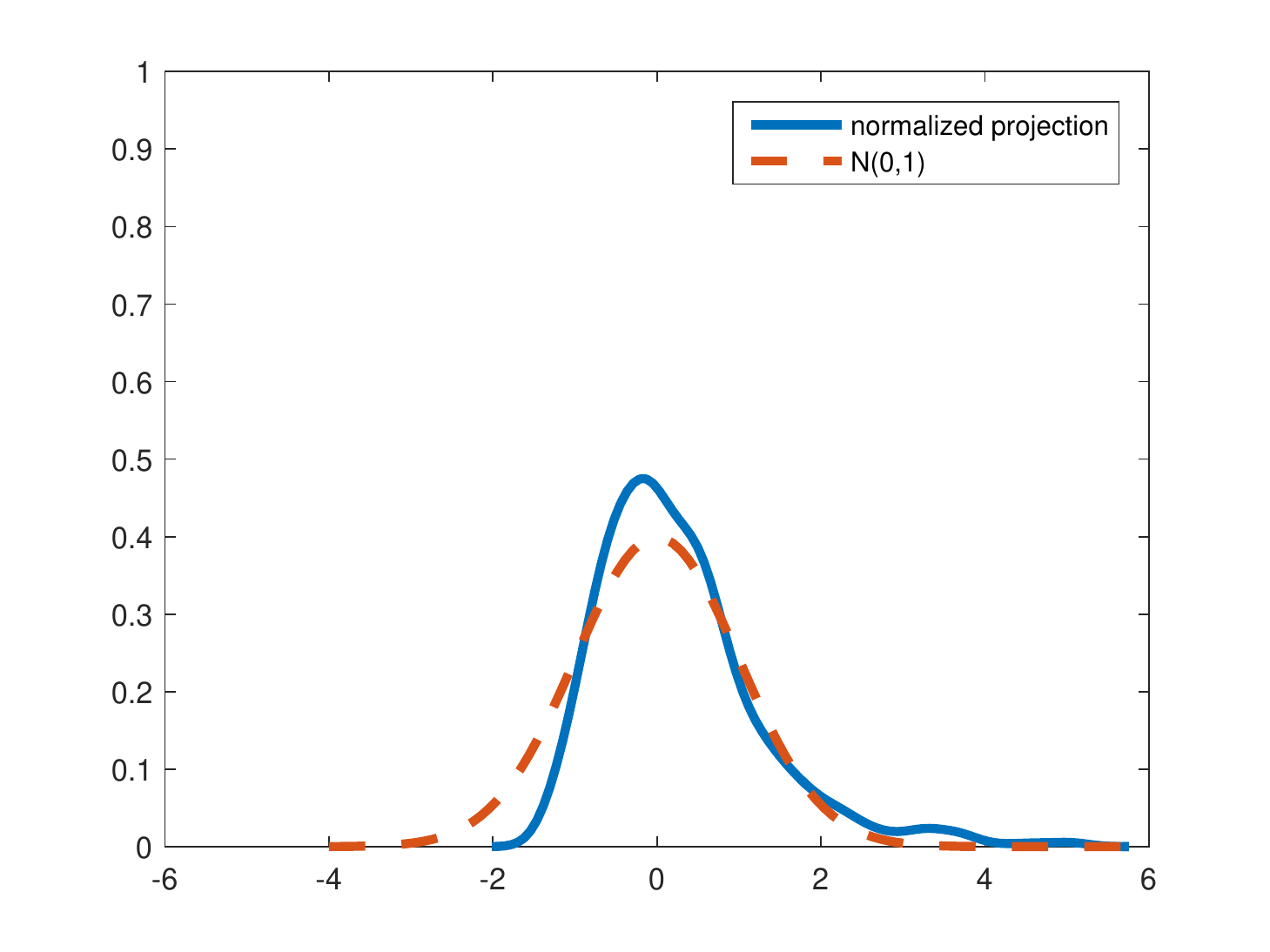} & \hskip-20pt %
\includegraphics[width=.23%
\textwidth,angle=0]{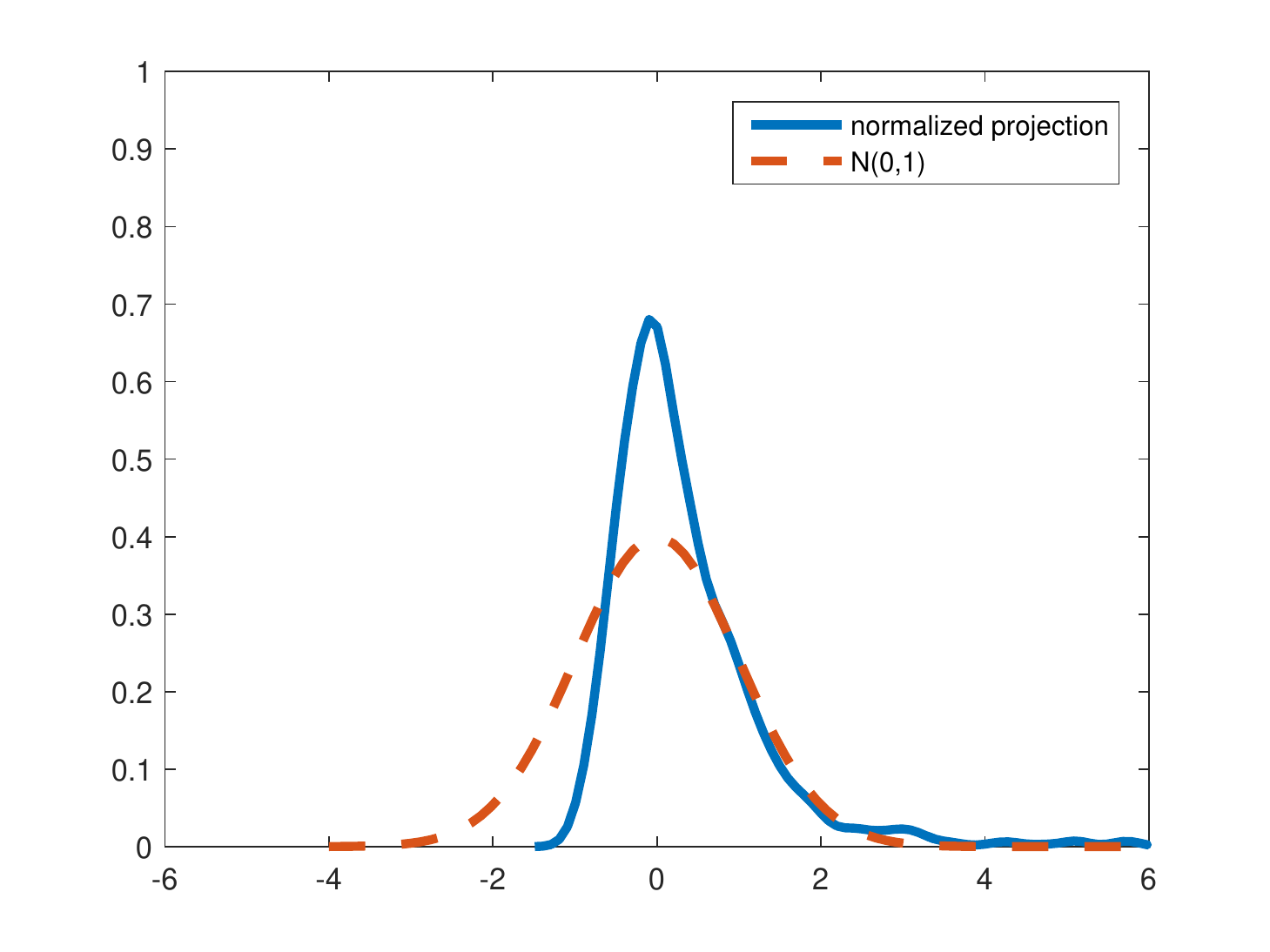} & \hskip-20pt %
\includegraphics[width=.23%
\textwidth,angle=0]{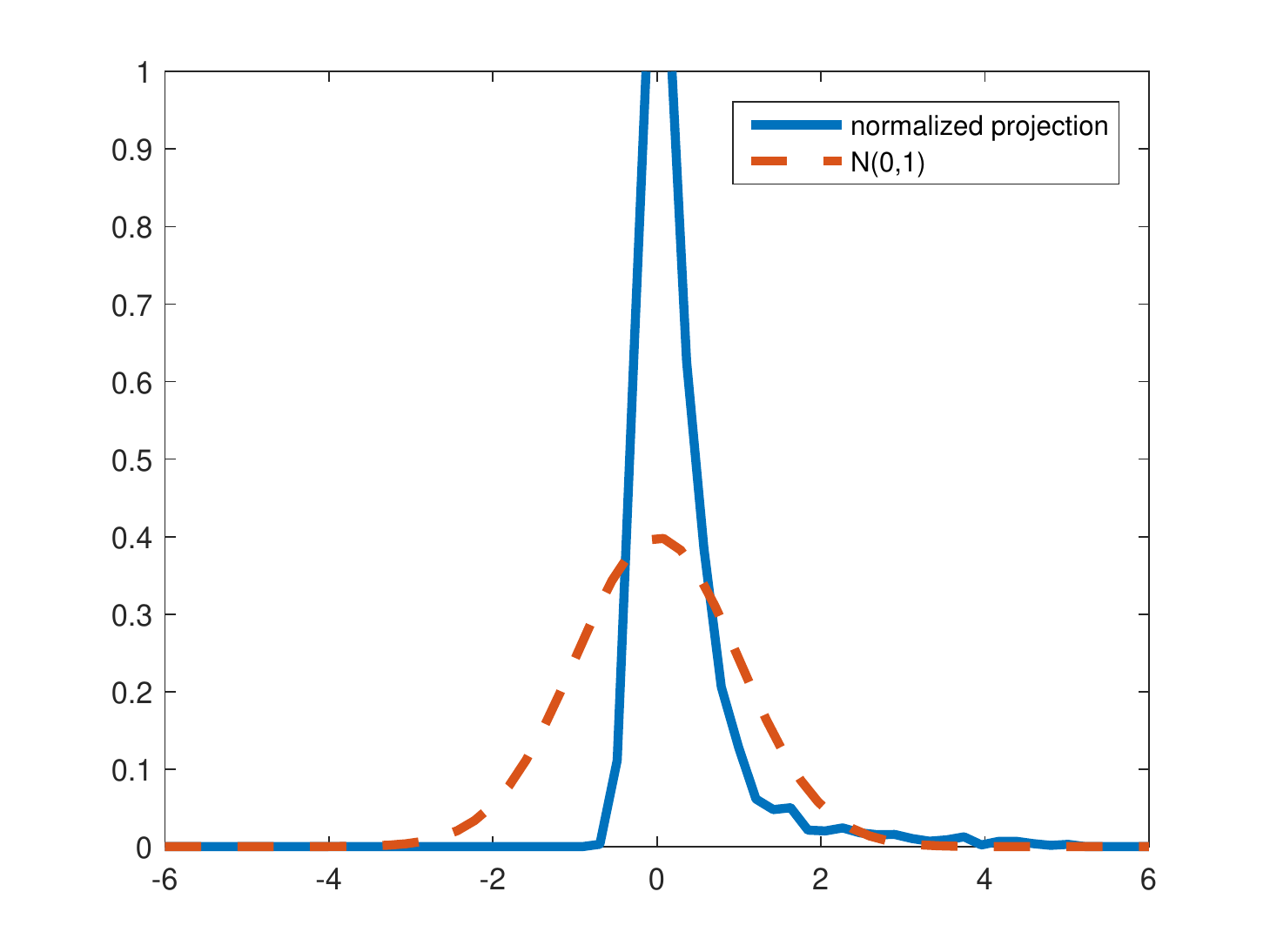} \\[-5pt] 
\hskip-20pt \includegraphics[width=.23%
\textwidth,angle=0]{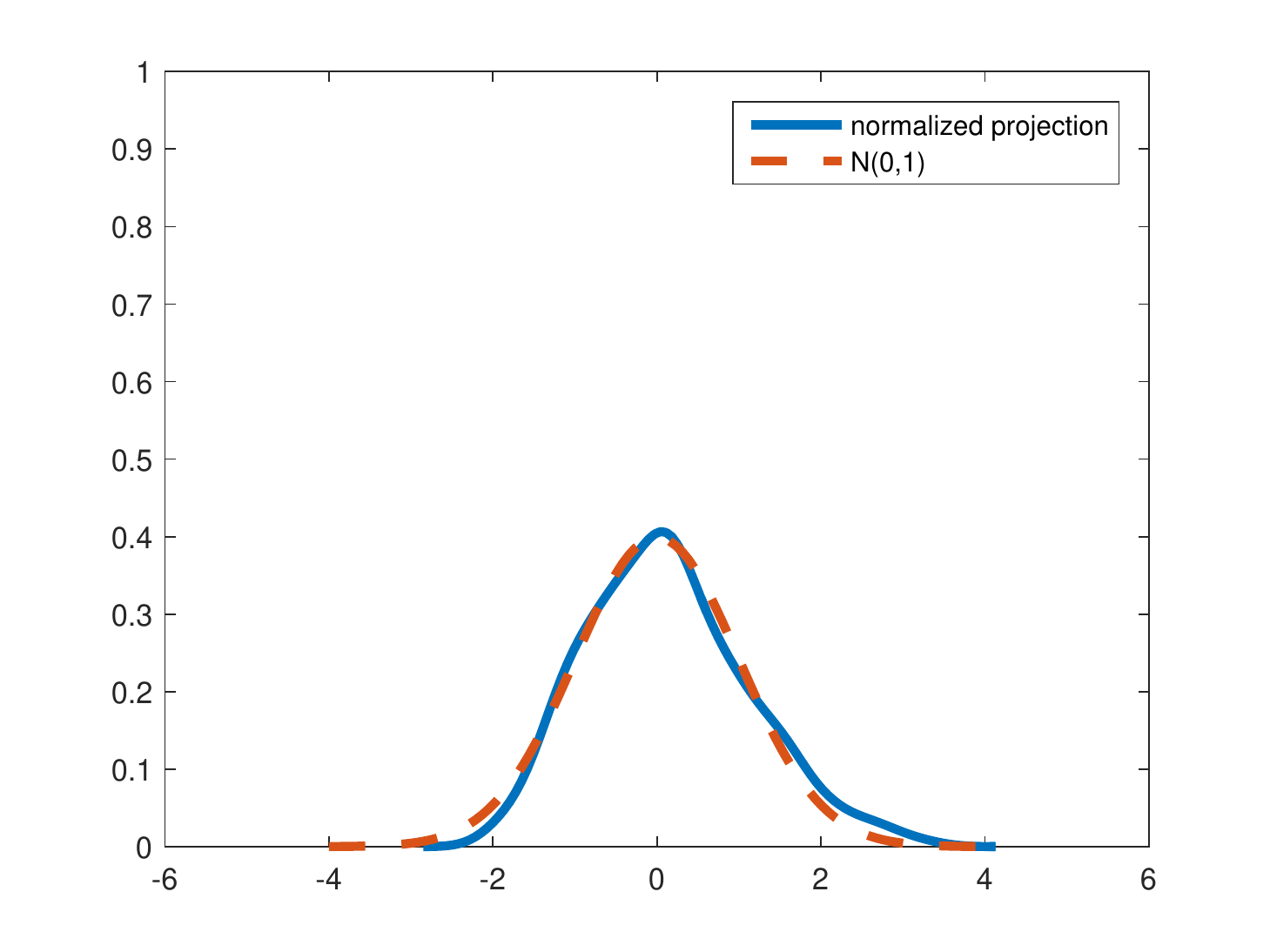} & \hskip-20pt %
\includegraphics[width=.23%
\textwidth,angle=0]{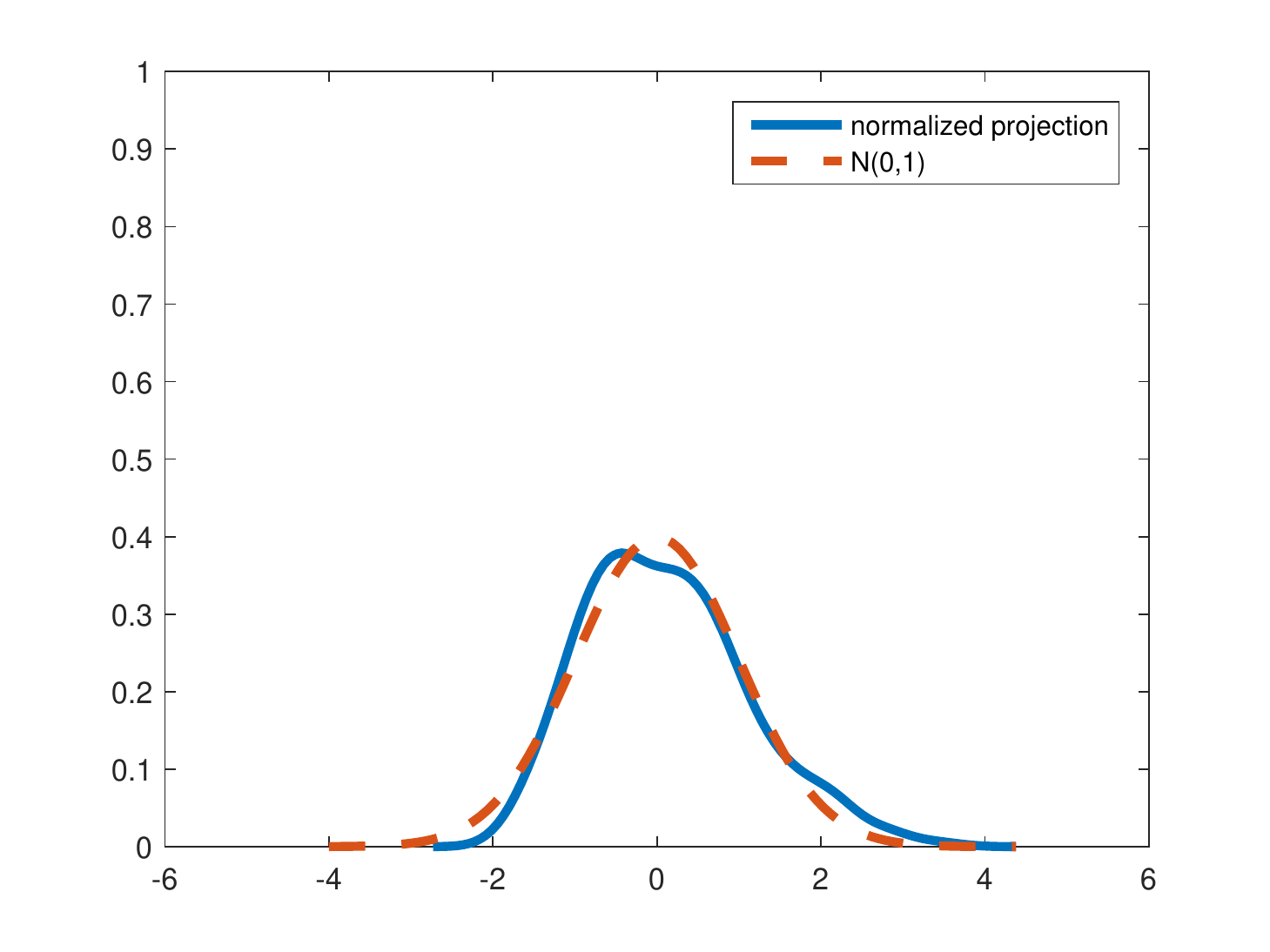} & \hskip-20pt %
\includegraphics[width=.23%
\textwidth,angle=0]{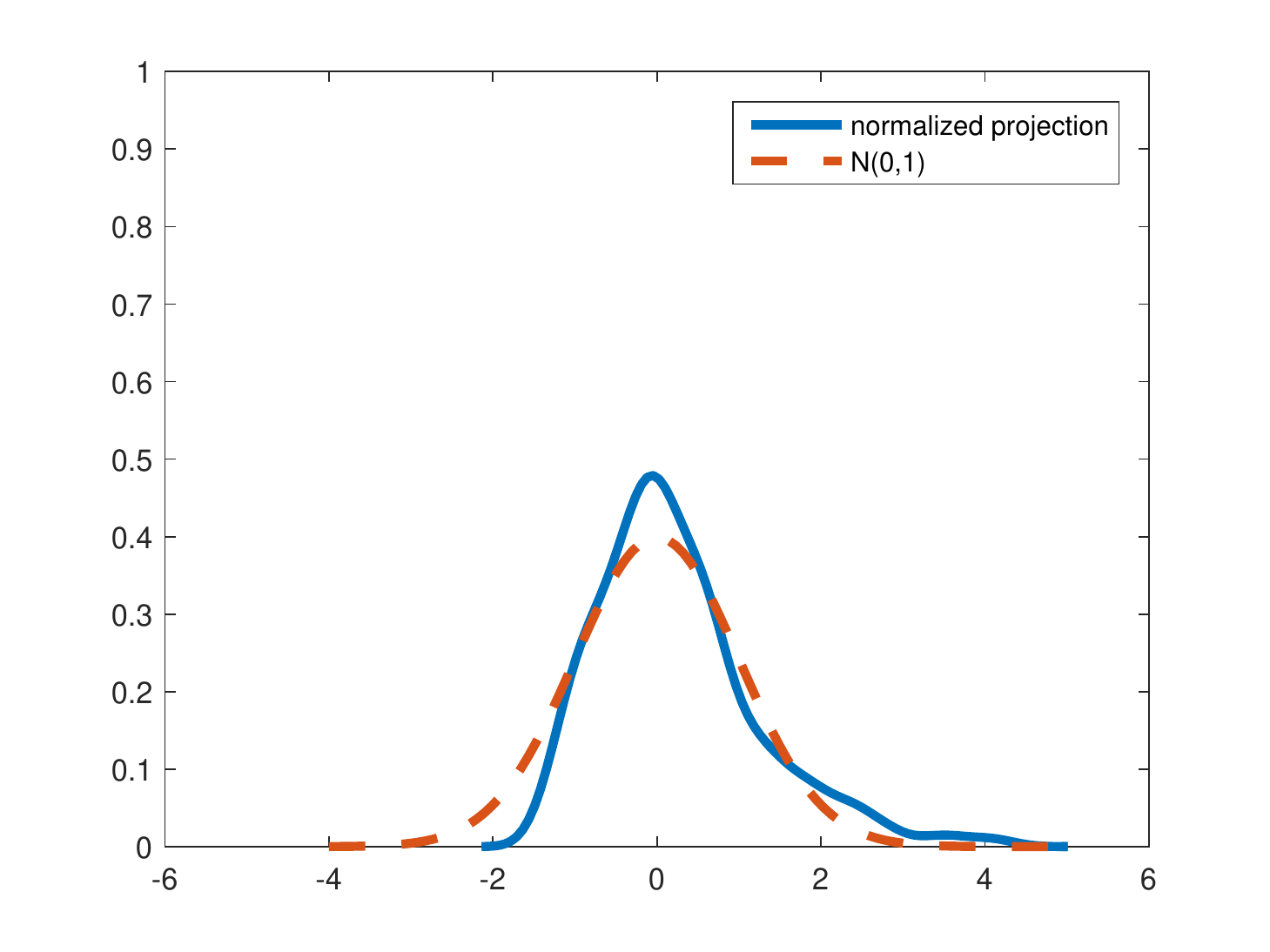} & \hskip-20pt %
\includegraphics[width=.23%
\textwidth,angle=0]{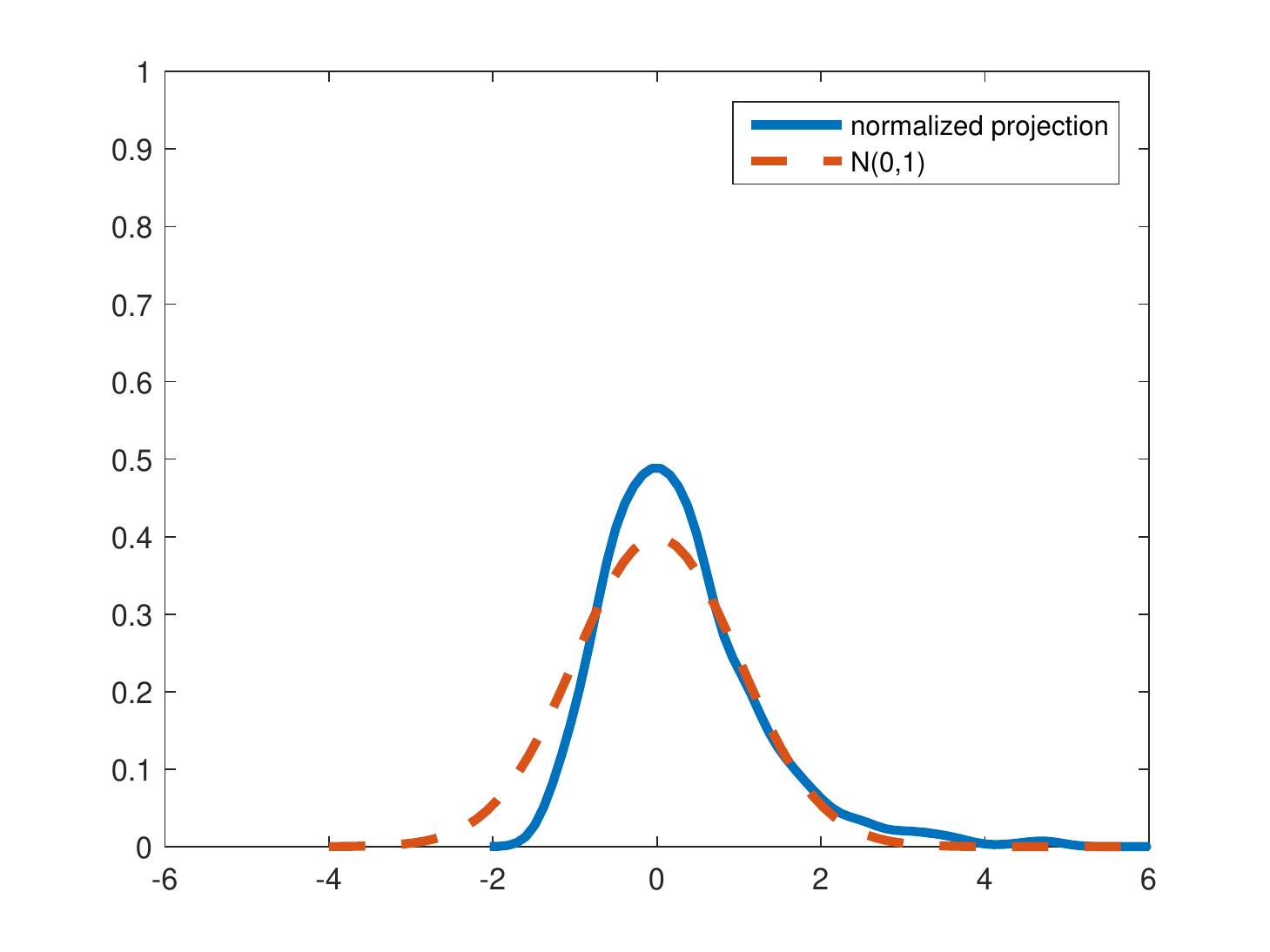} \\[-5pt] 
&  &  &  \\[-15pt] 
&  &  & 
\end{tabular}%
\end{center}
\par
\vspace{-0.5cm} \vspace{-0.5cm}
\caption{{\protect\small Plots of the kernel density estimates of the
normalized estimates (blue) v.s. $N(0,1)$ (red) under the first projection
direction ($n=100,200,400$ from top to bottom).}}
\label{fig:1}
\end{figure}

\begin{figure}[]
\begin{center}
\begin{tabular}{cccc}
{\scriptsize $p=1$} & {\scriptsize $p=2$} & {\scriptsize $p=3$} & 
{\scriptsize $p=4$ } \\ 
\hskip-20pt \includegraphics[width=.23%
\textwidth,angle=0]{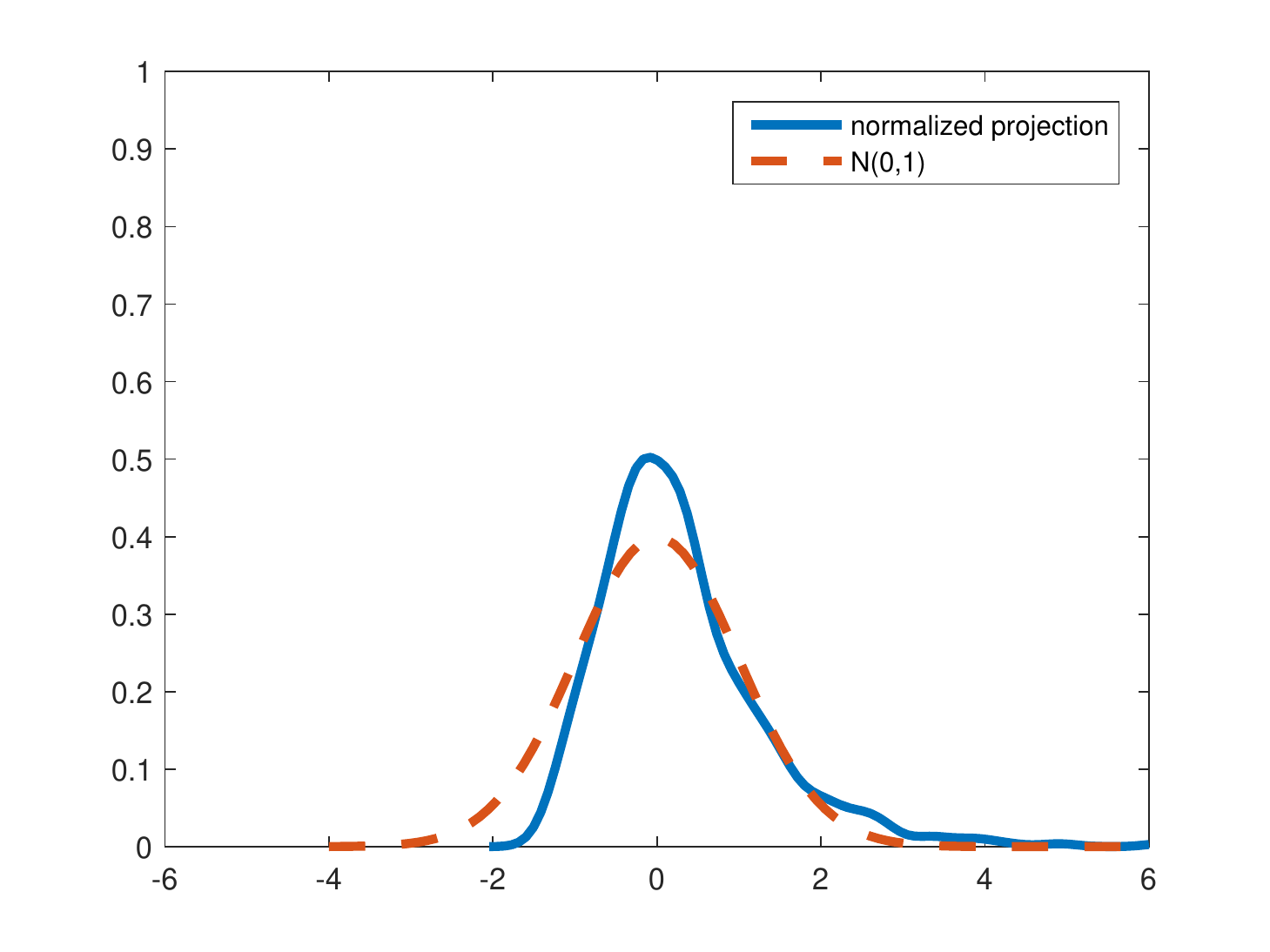} & \hskip-20pt %
\includegraphics[width=.23\textwidth,angle=0]{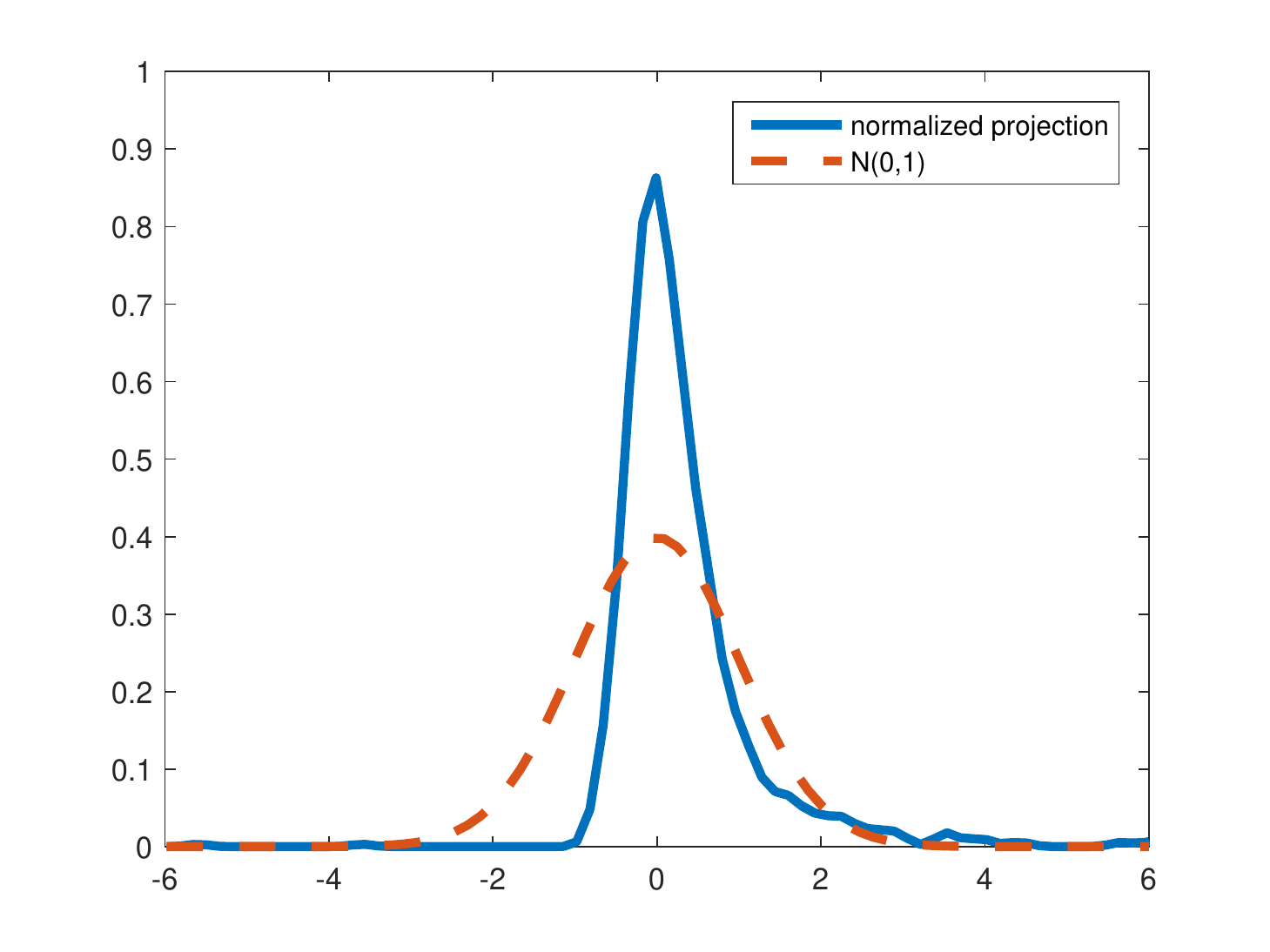}
& \hskip-20pt \includegraphics[width=.23%
\textwidth,angle=0]{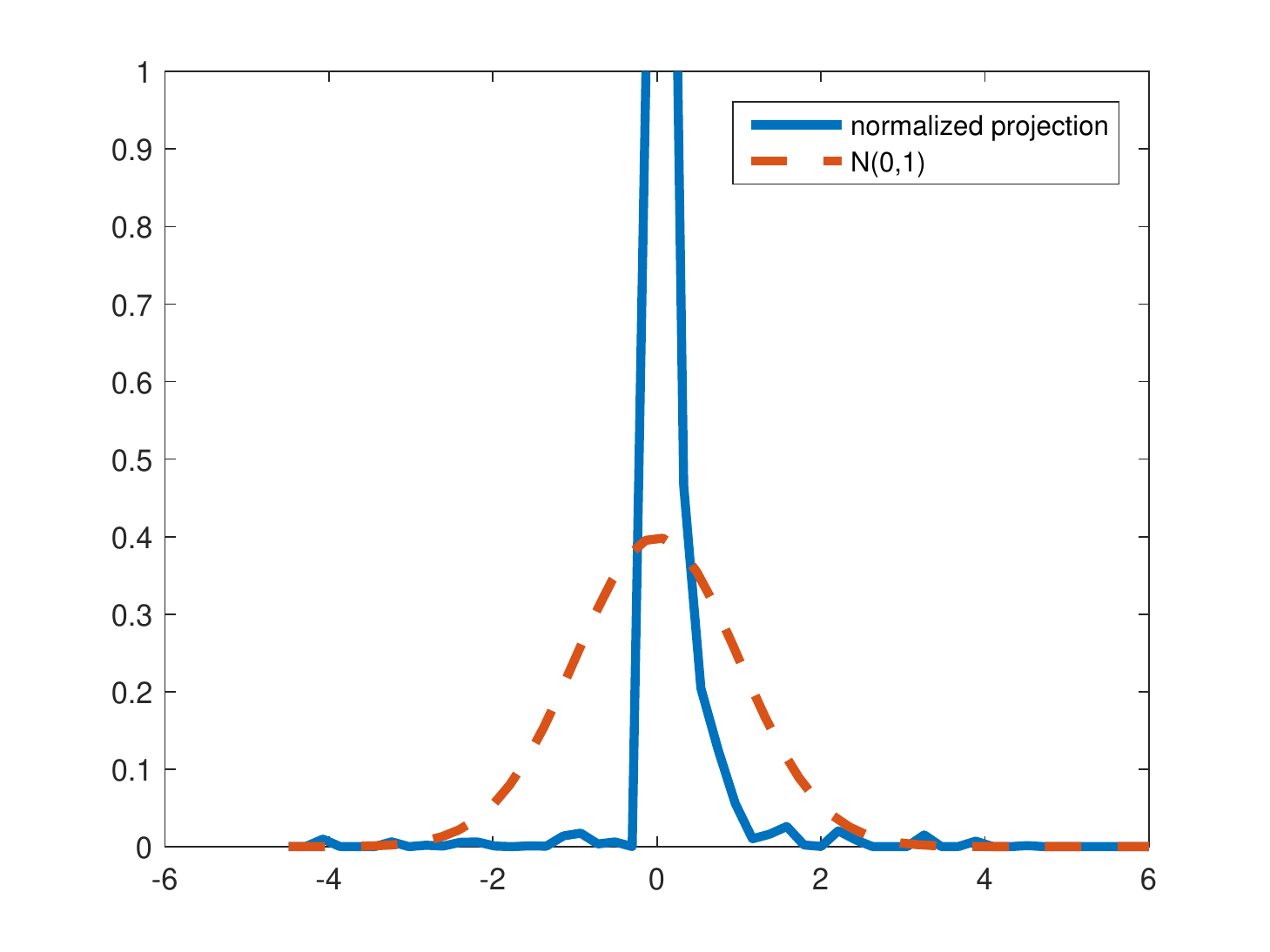} & \hskip-20pt %
\includegraphics[width=.23\textwidth,angle=0]{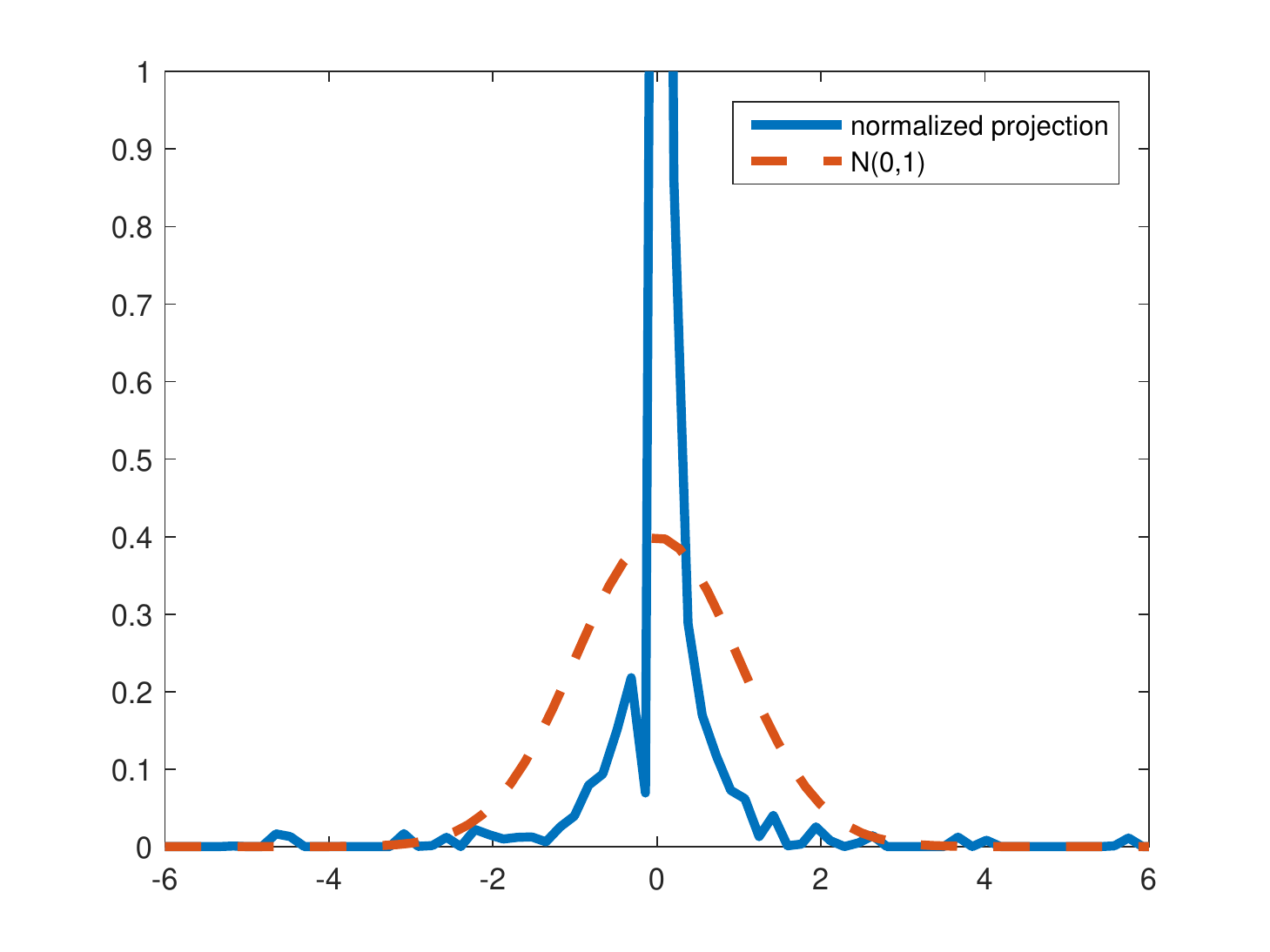}
\\[-5pt] 
\hskip-20pt \includegraphics[width=.23%
\textwidth,angle=0]{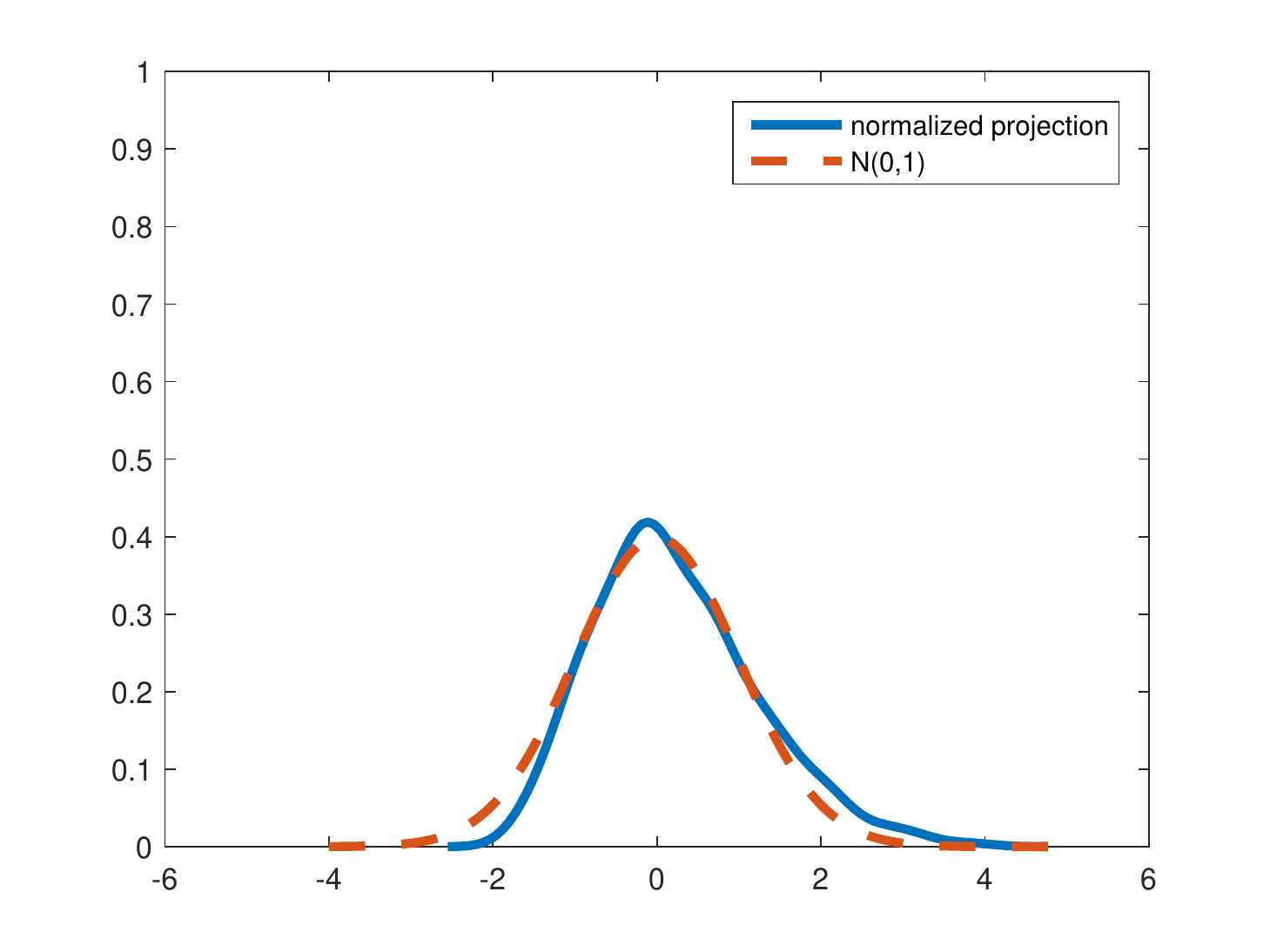} & \hskip-20pt %
\includegraphics[width=.23\textwidth,angle=0]{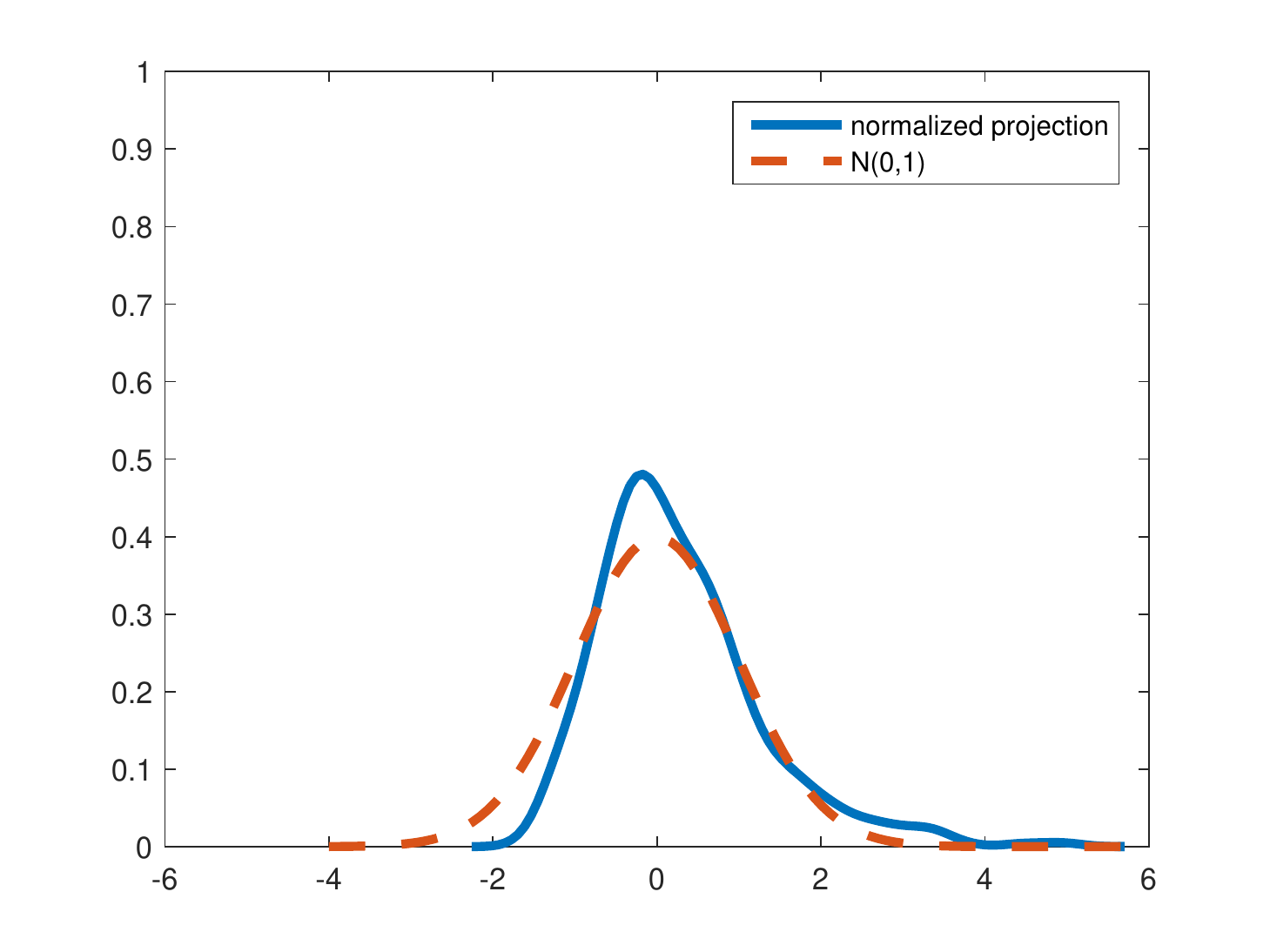}
& \hskip-20pt \includegraphics[width=.23%
\textwidth,angle=0]{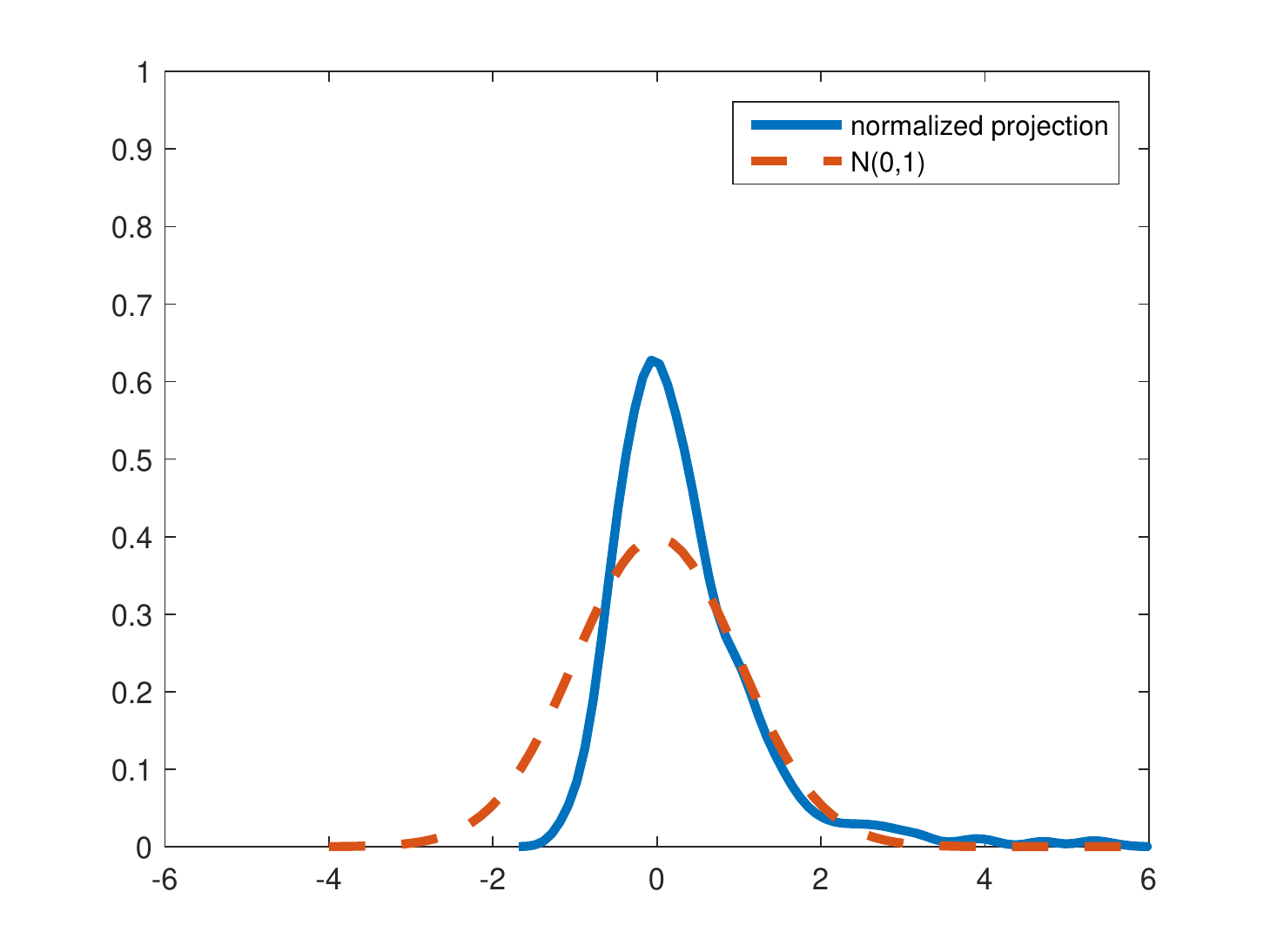} & \hskip-20pt %
\includegraphics[width=.23\textwidth,angle=0]{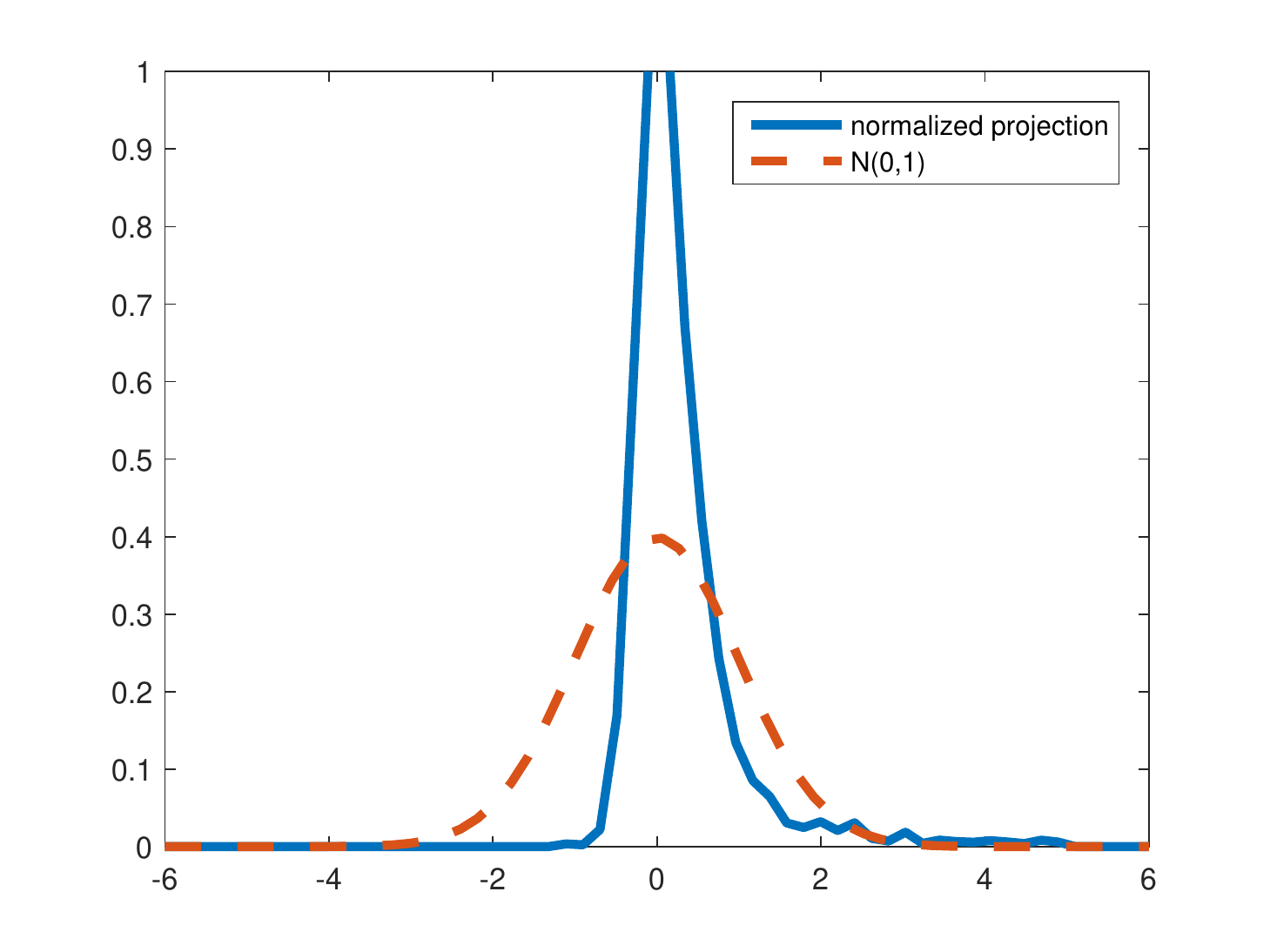}
\\[-5pt] 
\hskip-20pt \includegraphics[width=.23%
\textwidth,angle=0]{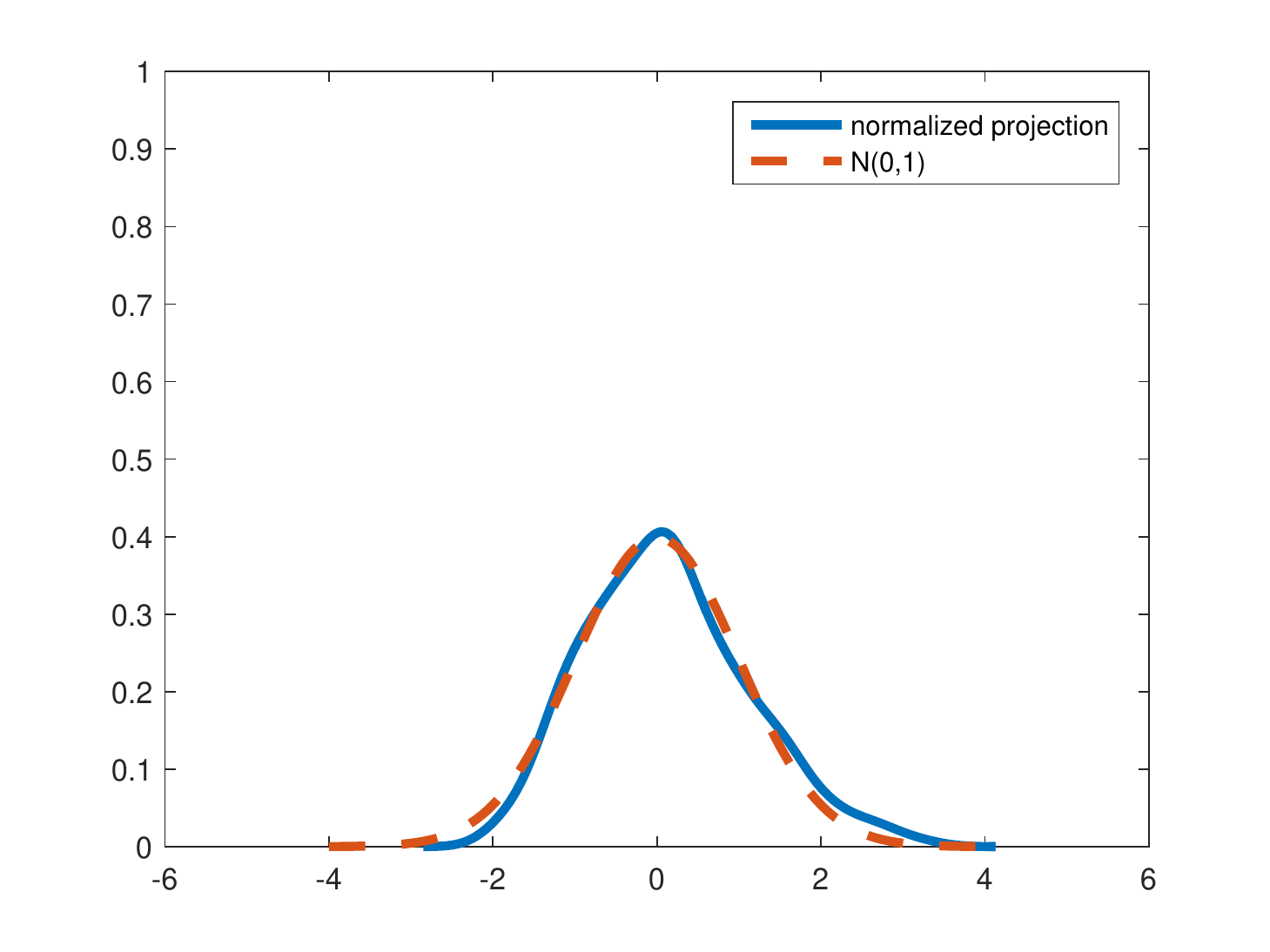} & \hskip-20pt %
\includegraphics[width=.23\textwidth,angle=0]{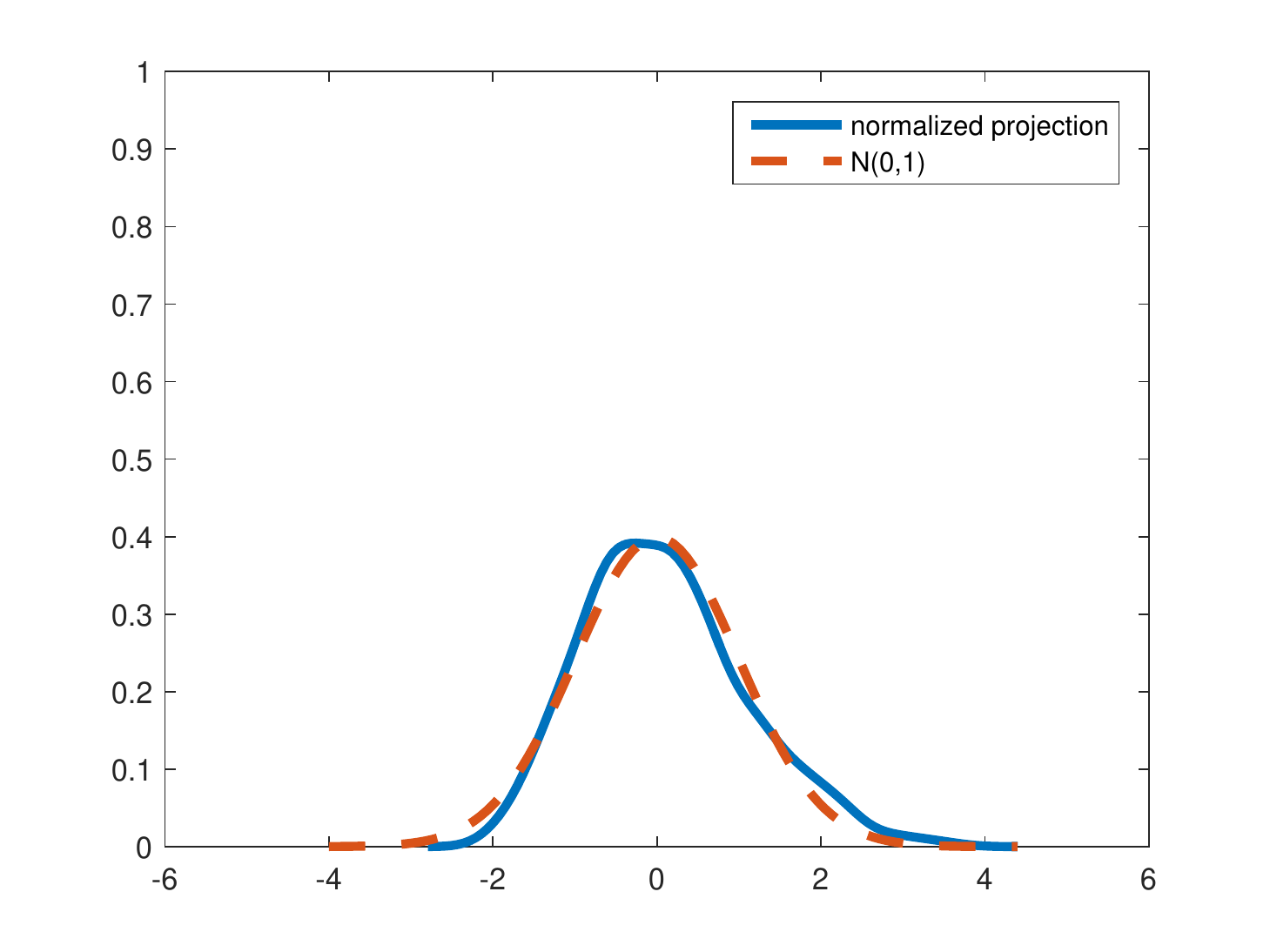}
& \hskip-20pt \includegraphics[width=.23%
\textwidth,angle=0]{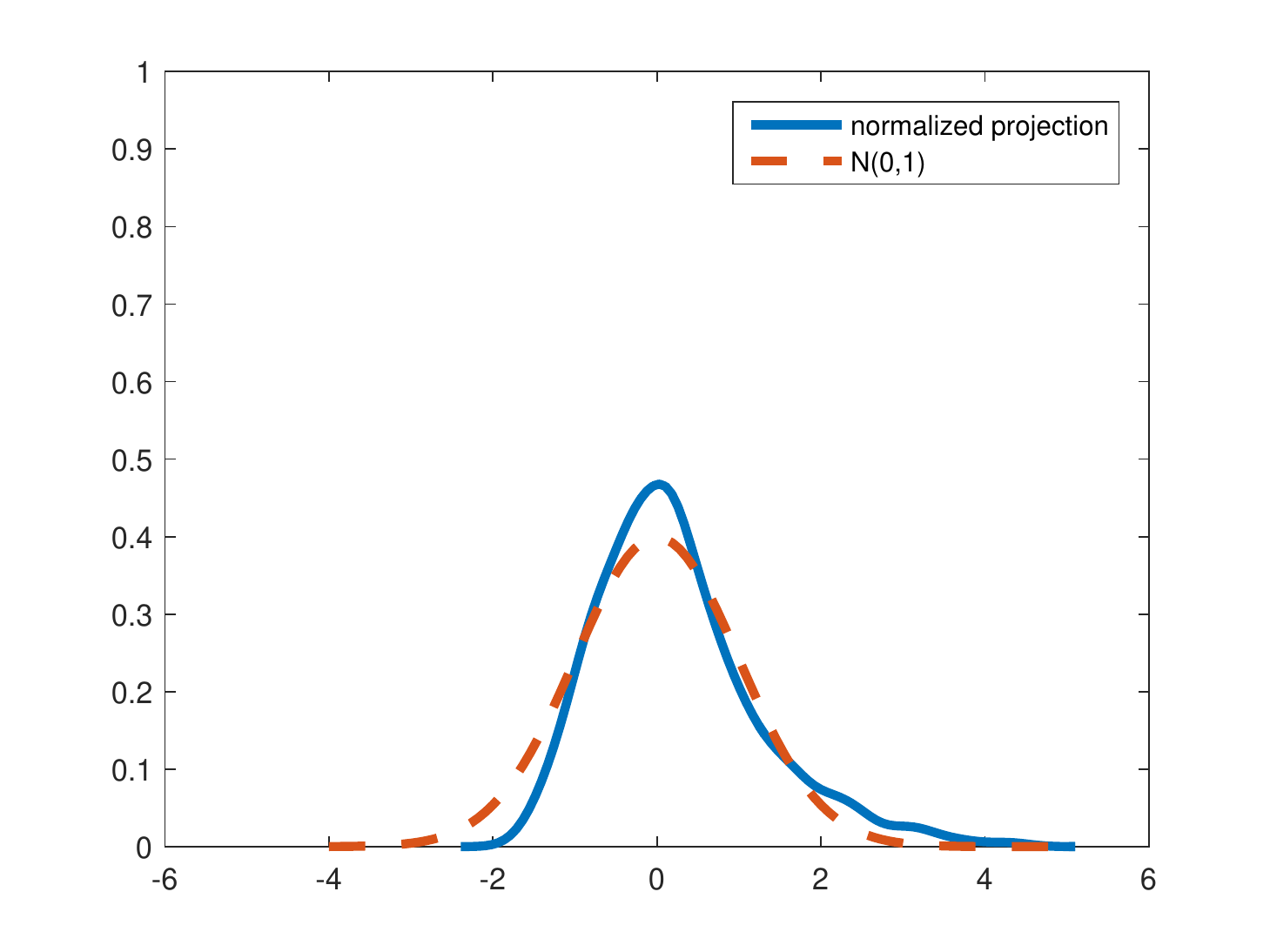} & \hskip-20pt %
\includegraphics[width=.23\textwidth,angle=0]{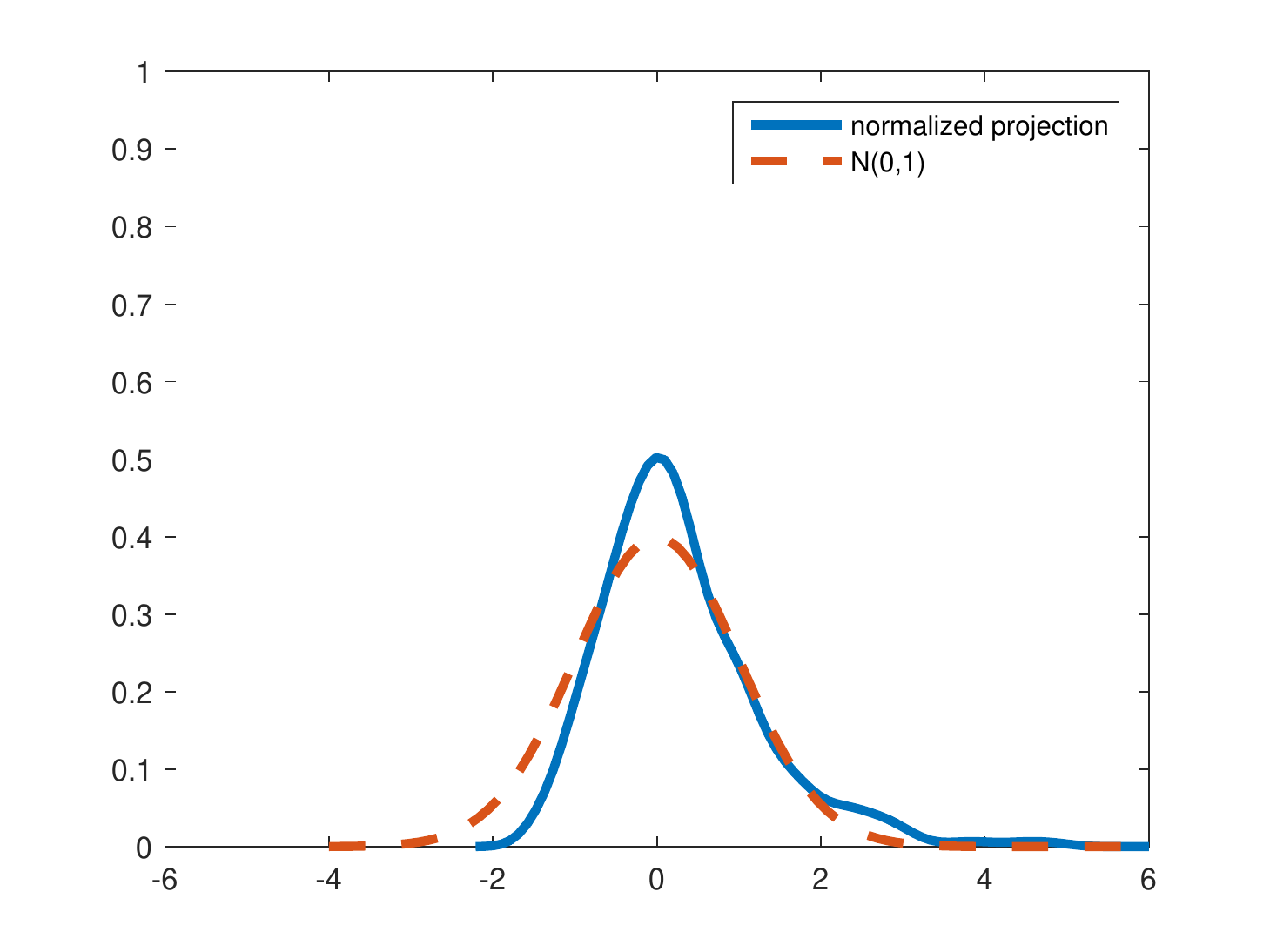}
\\[-5pt] 
&  &  &  \\[-15pt] 
&  &  & 
\end{tabular}%
\end{center}
\par
\vspace{-0.3cm} \vspace{-0.5cm}
\caption{{\protect\small Plots of the kernel density estimates of the
normalized estimates (blue) v.s. $N(0,1)$ (red) under the second projection
direction ($n=100,200,400$ from top to bottom).}}
\label{fig:2}
\end{figure}

\begin{figure}[]
\begin{center}
\begin{tabular}{cccc}
{\scriptsize $p=1$} & {\scriptsize $p=2$} & {\scriptsize $p=3$} & 
{\scriptsize $p=4$ } \\ 
\vspace{-0.1in} \hskip-20pt \includegraphics[width=.23%
\textwidth,angle=0]{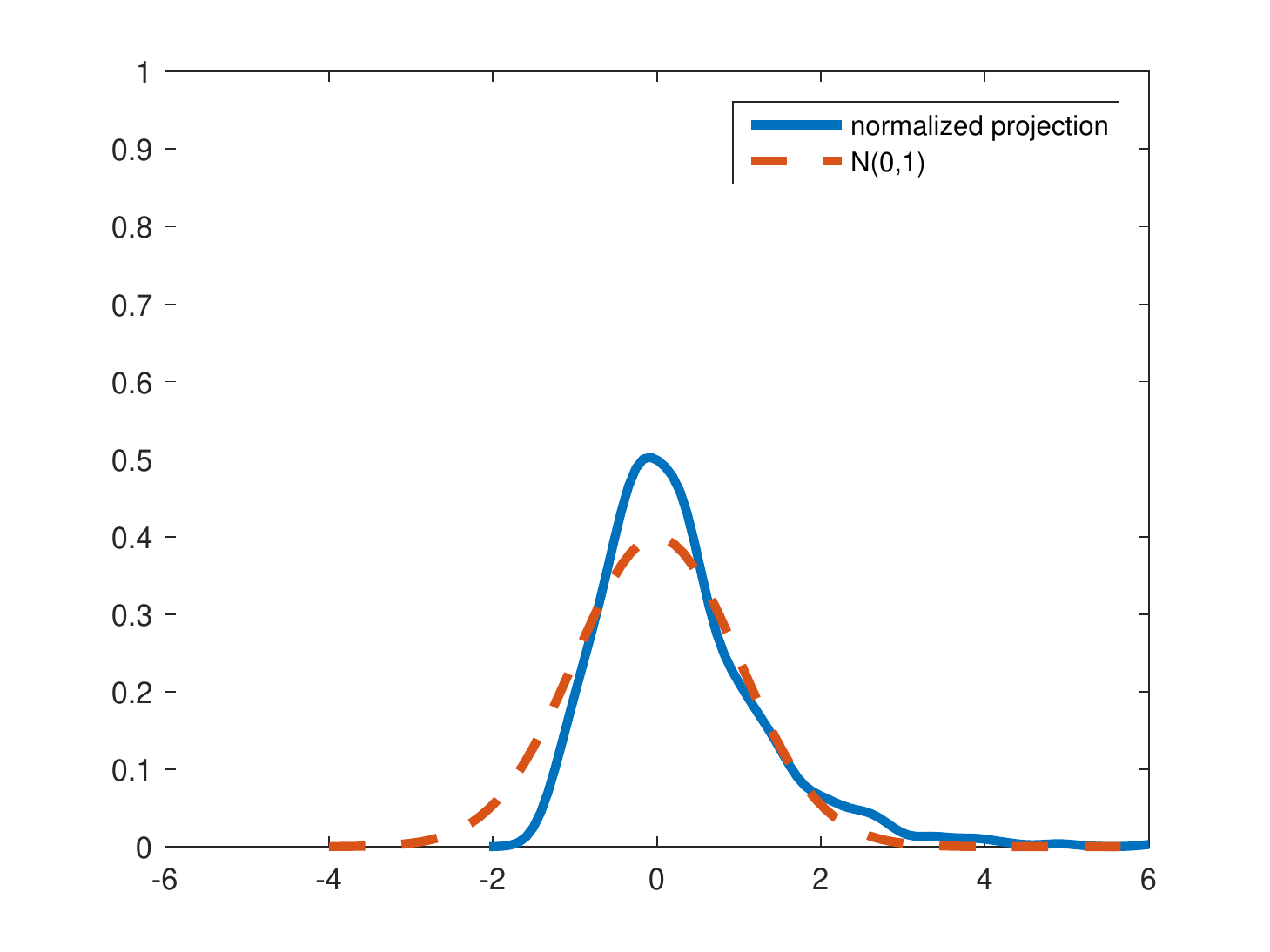} & \hskip-20pt %
\includegraphics[width=.23%
\textwidth,angle=0]{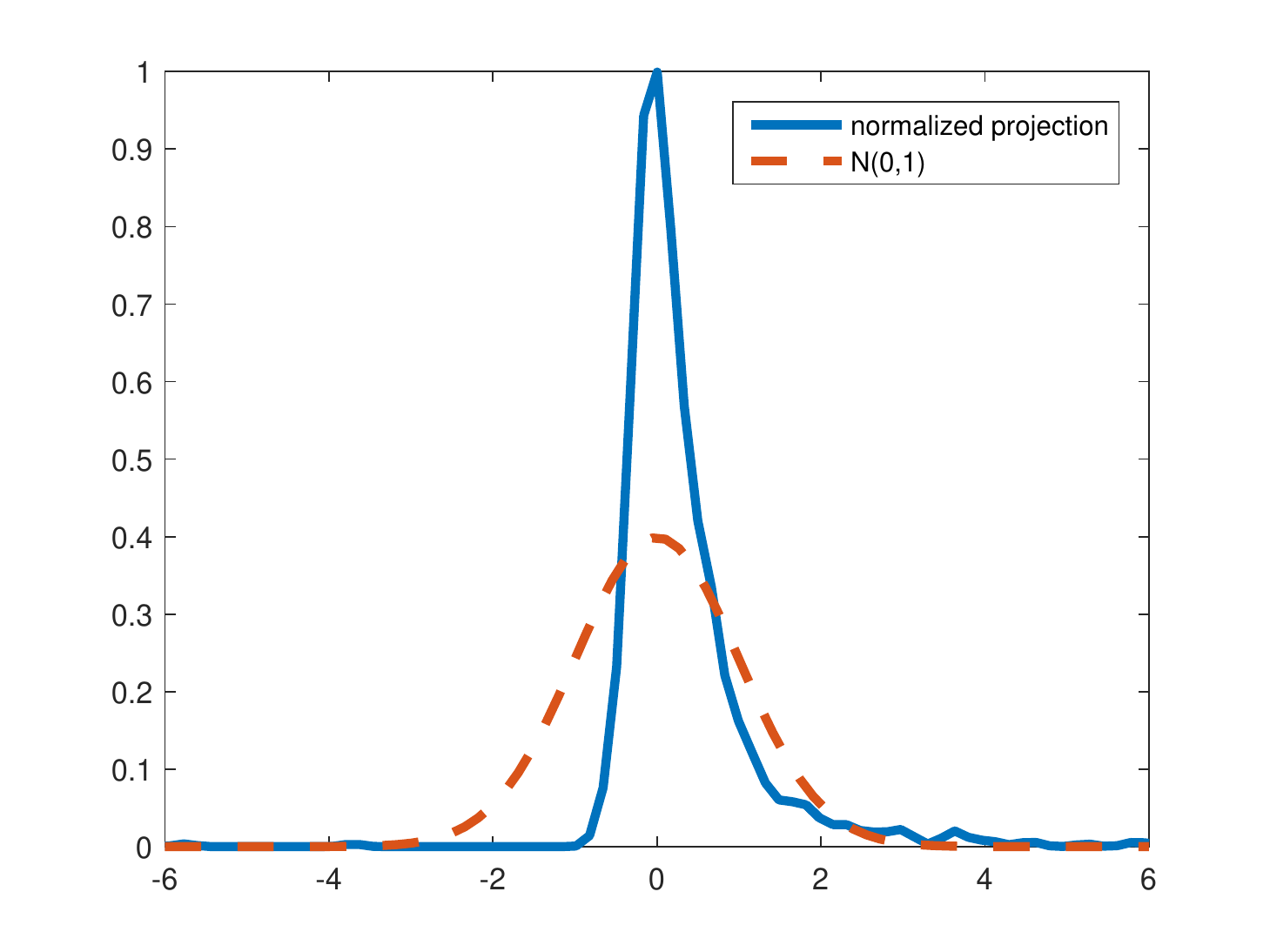} & \hskip-20pt %
\includegraphics[width=.23%
\textwidth,angle=0]{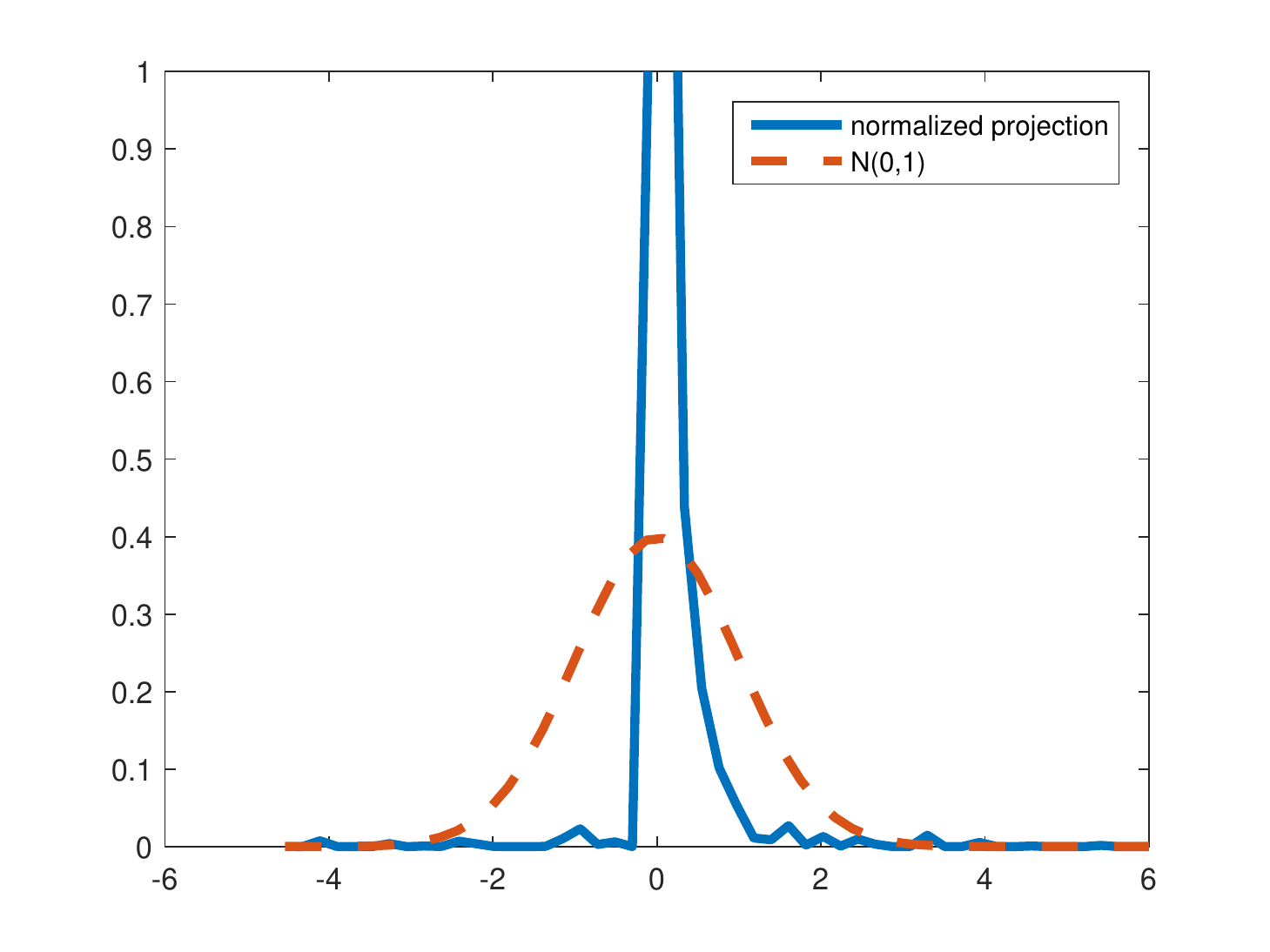} & \hskip-20pt %
\includegraphics[width=.23%
\textwidth,angle=0]{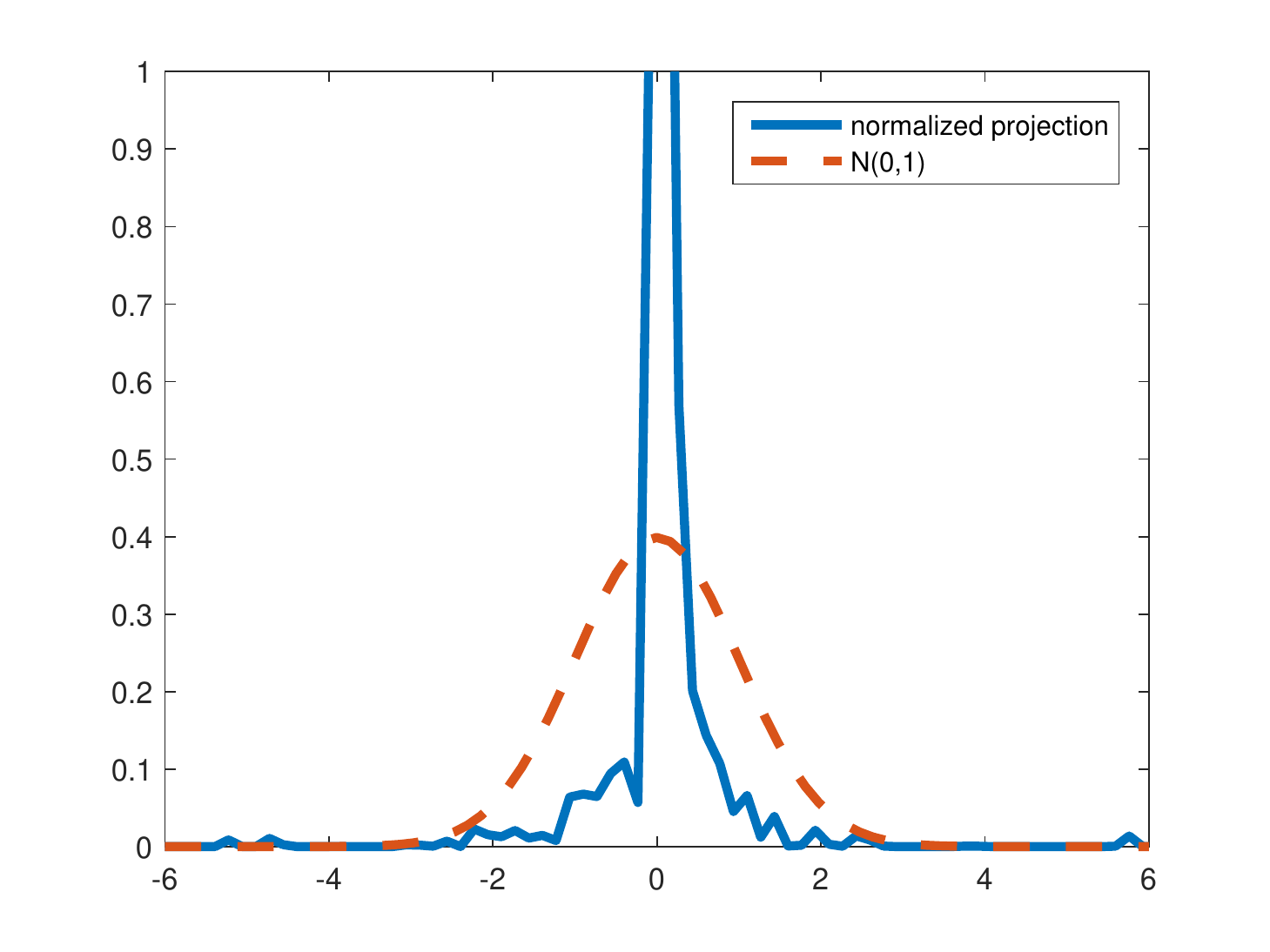} \\[-5pt] 
\hskip-20pt \includegraphics[width=.23%
\textwidth,angle=0]{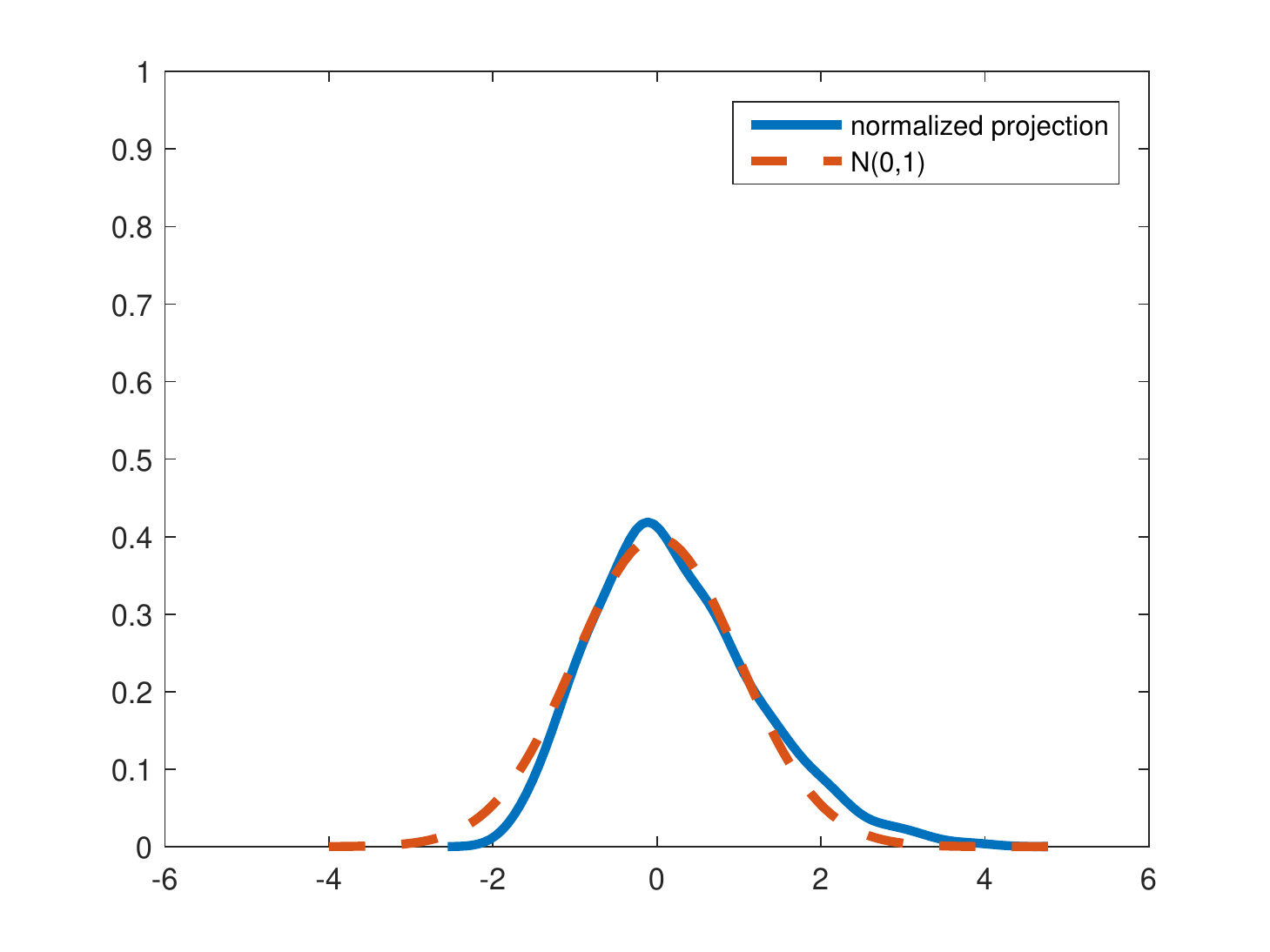} & \hskip-20pt %
\includegraphics[width=.23%
\textwidth,angle=0]{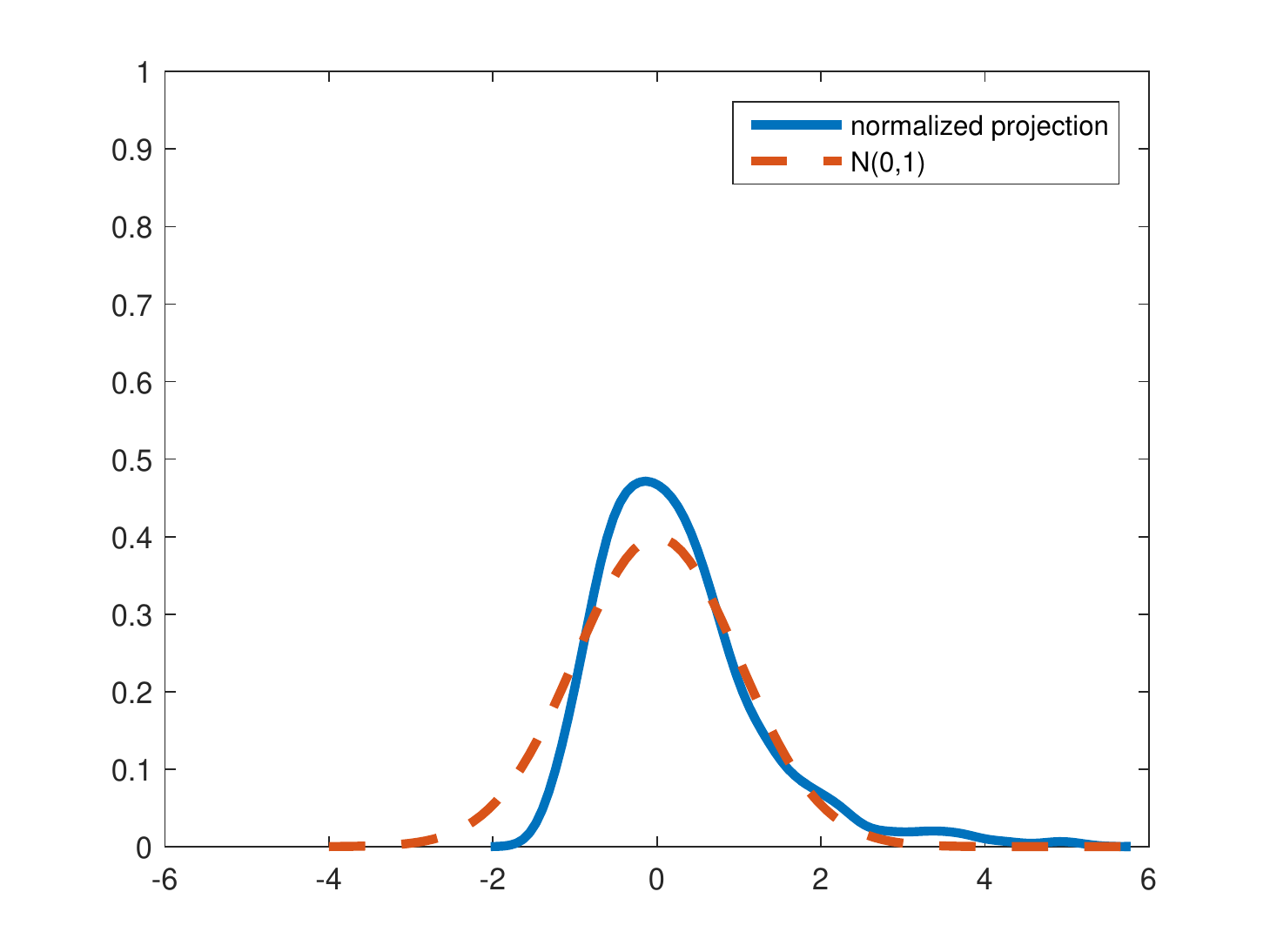} & \hskip-20pt %
\includegraphics[width=.23%
\textwidth,angle=0]{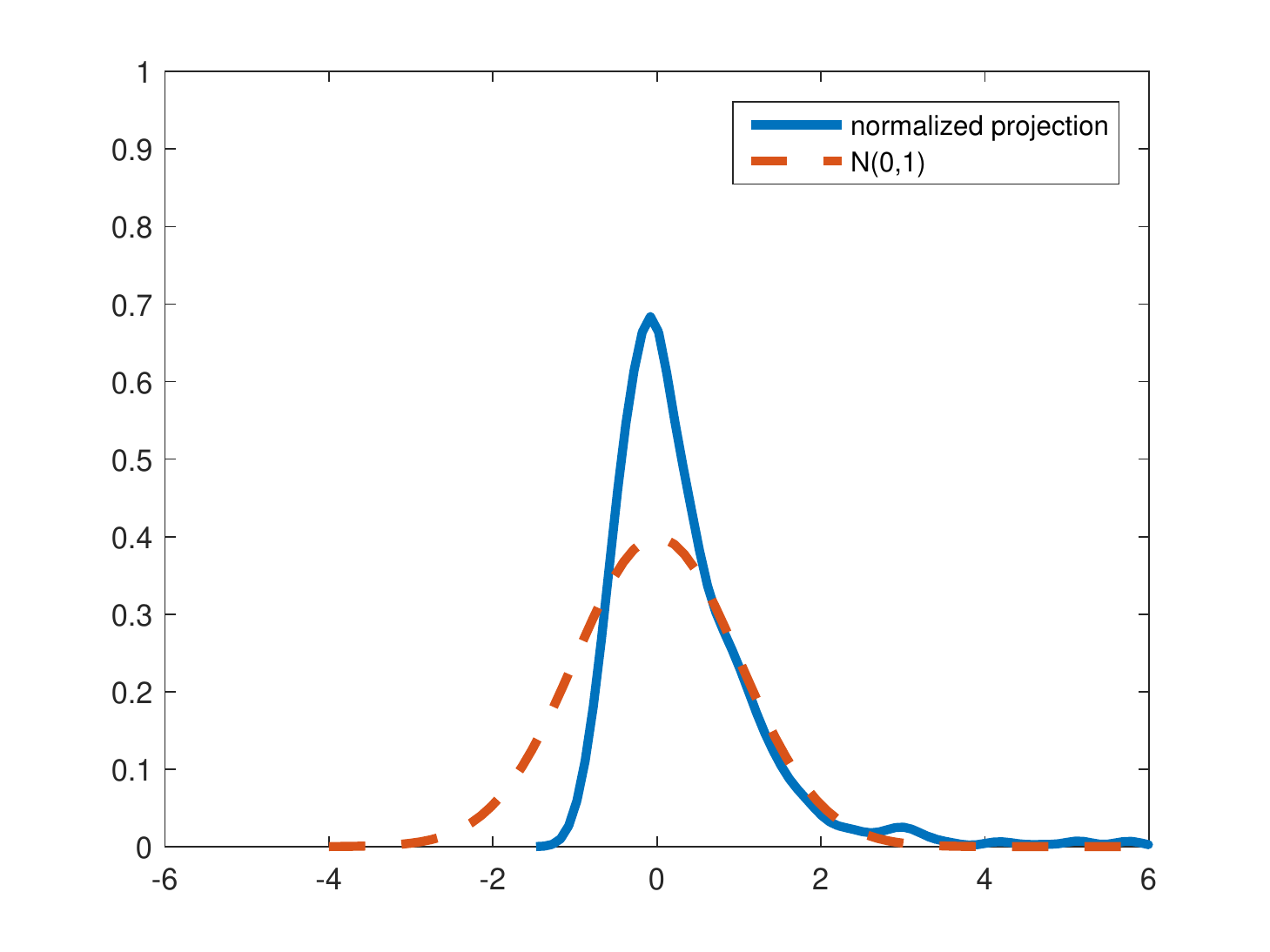} & \hskip-20pt %
\includegraphics[width=.23%
\textwidth,angle=0]{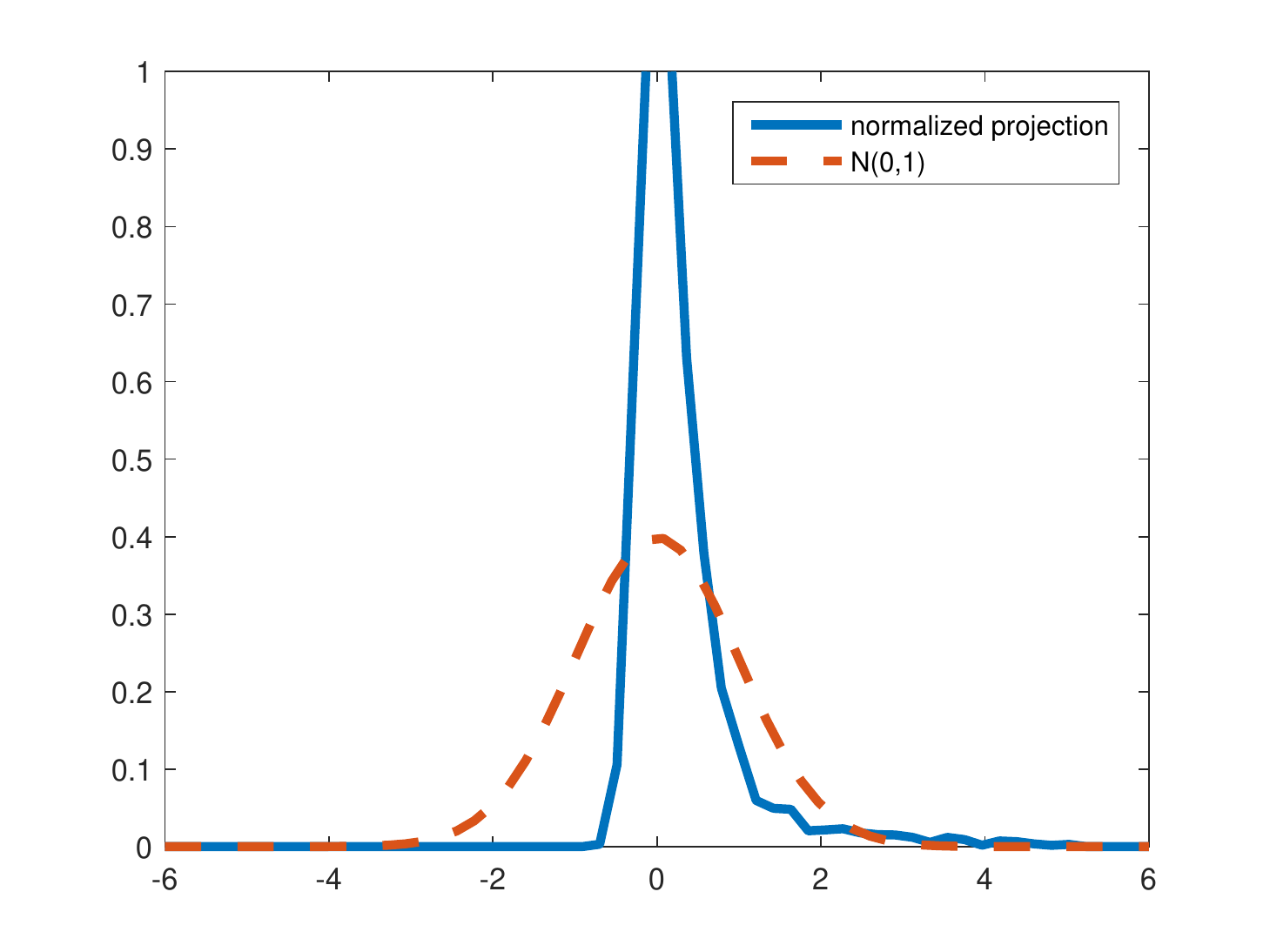} \\[-5pt] 
\hskip-20pt \includegraphics[width=.23%
\textwidth,angle=0]{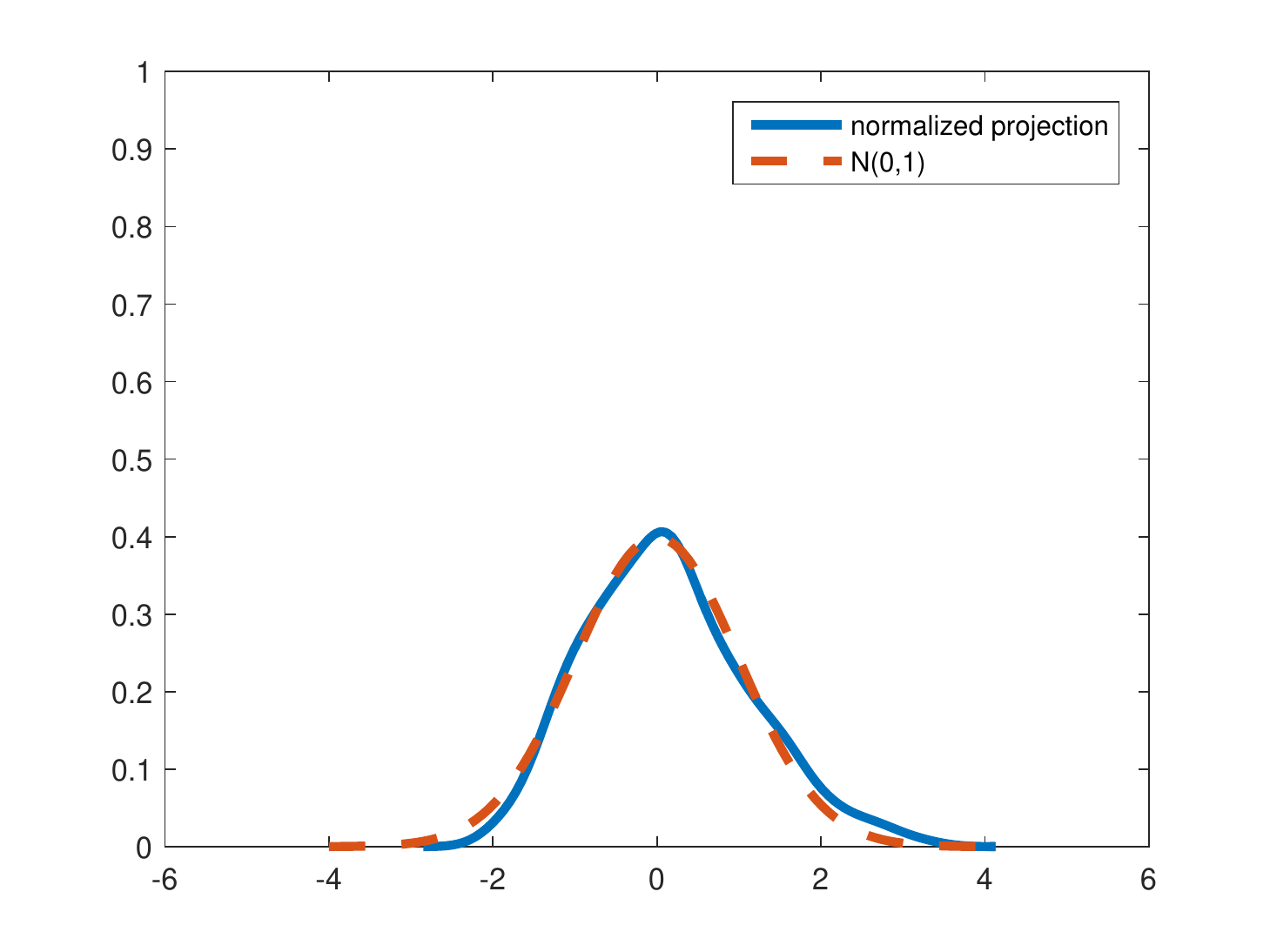} & \hskip-20pt %
\includegraphics[width=.23%
\textwidth,angle=0]{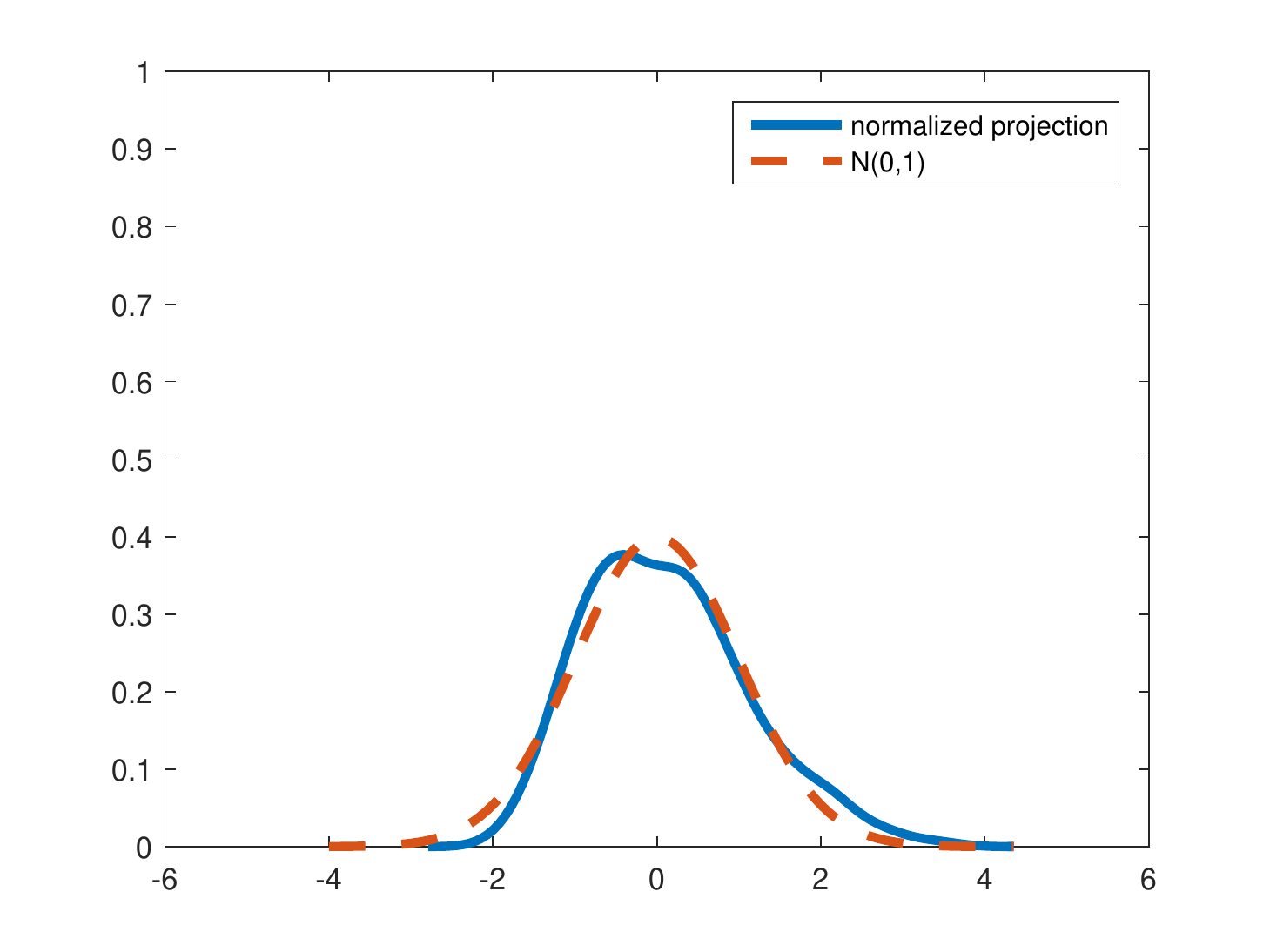} & \hskip-20pt %
\includegraphics[width=.23%
\textwidth,angle=0]{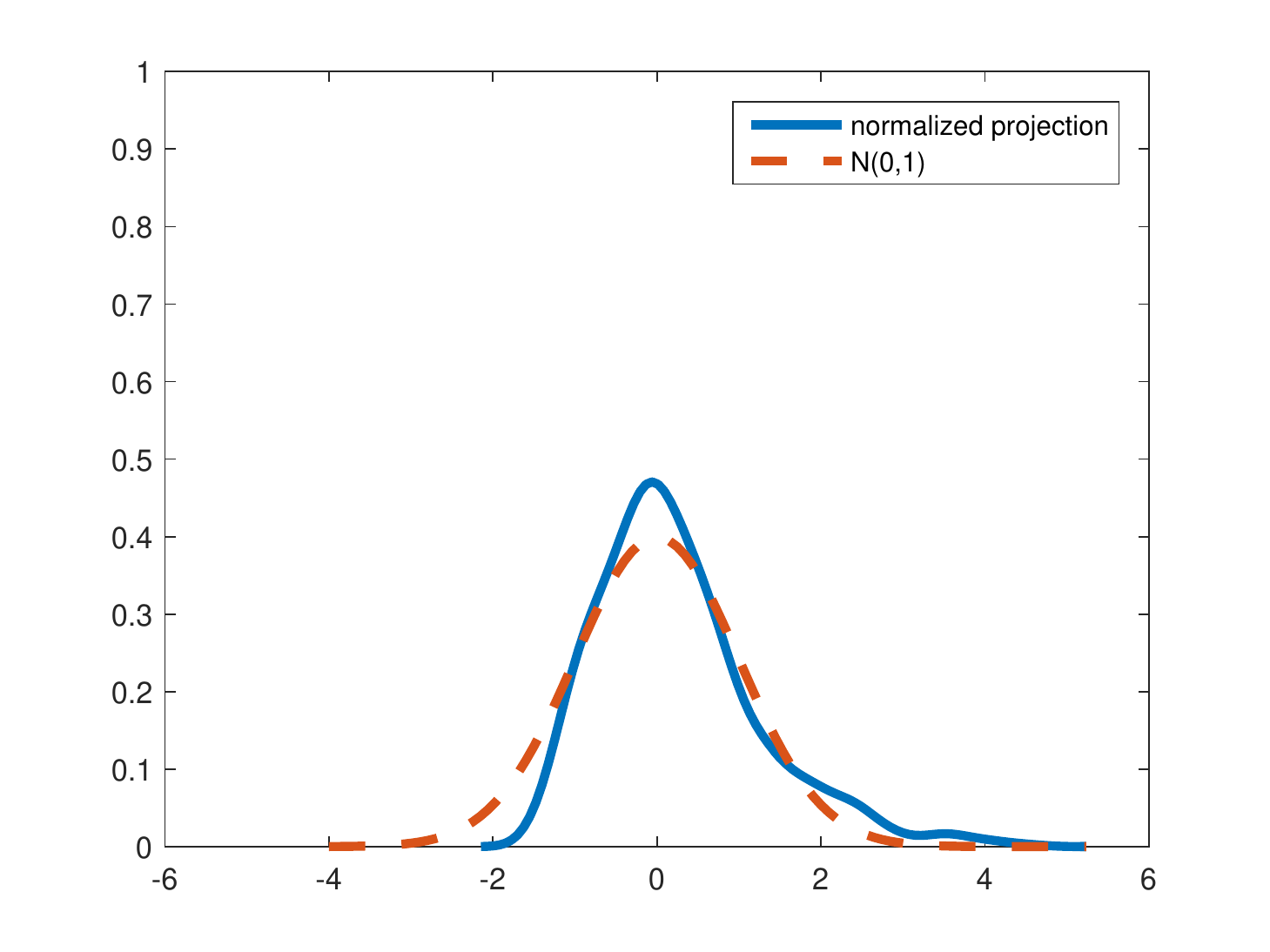} & \hskip-20pt %
\includegraphics[width=.23%
\textwidth,angle=0]{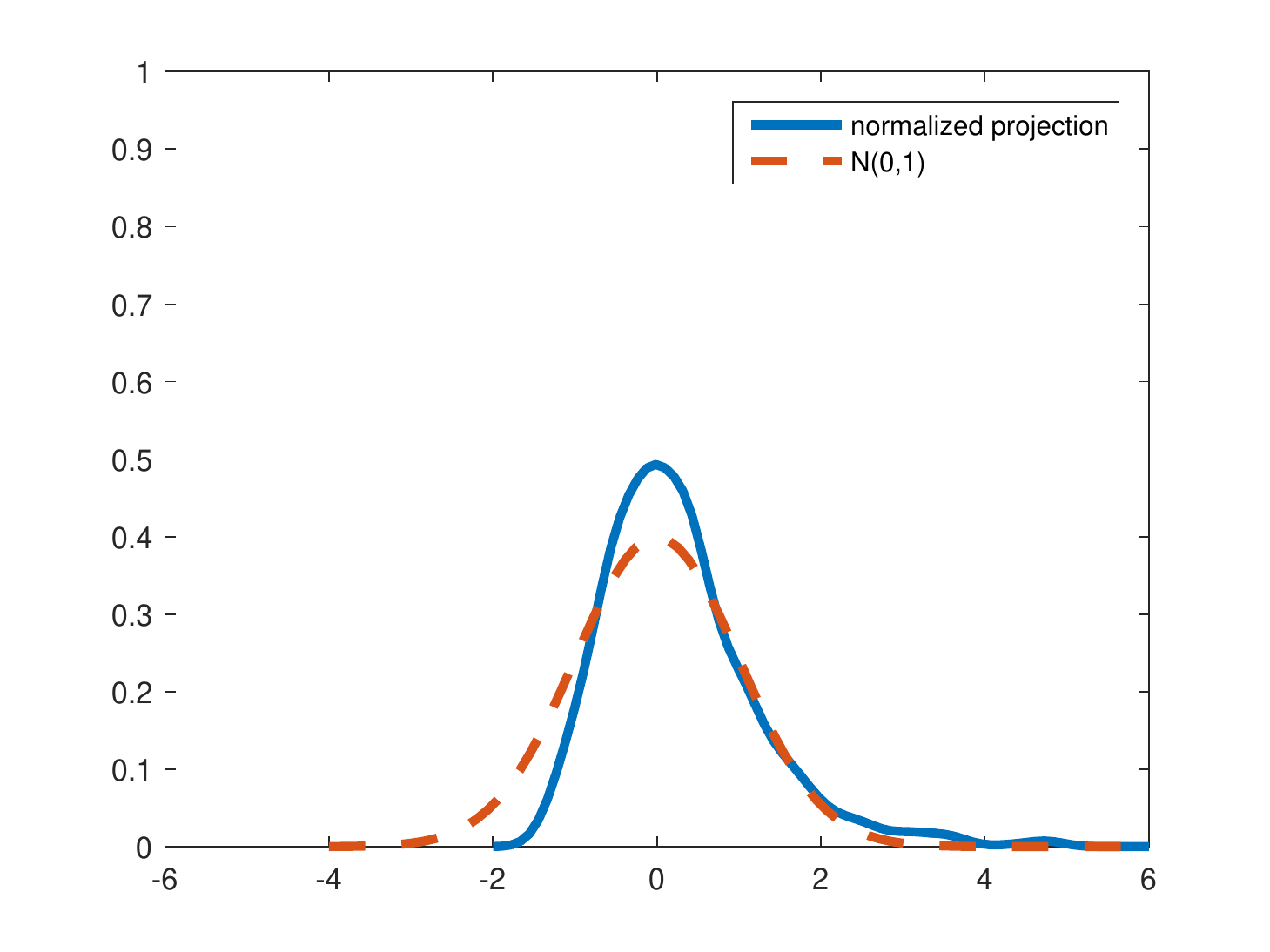} \\[-5pt] 
&  &  &  \\[-15pt] 
&  &  & 
\end{tabular}%
\vspace{-0.5cm}
\end{center}
\caption{{\protect\small Plots of the kernel density estimates of the
normalized estimates (blue) v.s. $N(0,1)$ (red) under the third projection
direction ($n=100,200,400$ from top to bottom).}}
\label{fig:3}
\end{figure}

{\ \renewcommand{\tabcolsep}{3pt} \renewcommand{\arraystretch}{1.1} 
\begin{table}[]
\caption{MAE of the covariance estimator. The results are obtained using
1,000 replications.}
\label{tab:4}\centering
\vspace{0.2cm} \centering
\vspace{0.2cm}
\par
\begin{tabular}{l|llll|llll|llll}
\hline
\multirow{2}{*}{$\epsilon_{n}$} & \multicolumn{4}{c|}{$n=100$} & 
\multicolumn{4}{c|}{$n=200$} & \multicolumn{4}{c}{$n=400$} \\ \cline{2-13}
& \multicolumn{1}{c}{$p=1$} & \multicolumn{1}{c}{$p=2$} & \multicolumn{1}{c}{%
$p=3$} & \multicolumn{1}{c|}{$p=4$} & \multicolumn{1}{c}{$p=1$} & 
\multicolumn{1}{c}{$p=2$} & \multicolumn{1}{c}{$p=3$} & \multicolumn{1}{c|}{$%
p=4$} & \multicolumn{1}{c}{$p=1$} & \multicolumn{1}{c}{$p=2$} & 
\multicolumn{1}{c}{$p=3$} & \multicolumn{1}{c}{$p=4$} \\ \hline
$1.1n^{-1/6}$ & 0.476 & 1.509 & 2.198 & \textbf{3.705} & 0.170 & \textbf{%
0.656} & \textbf{1.208} & \textbf{1.601} & 0.081 & \textbf{0.267} & \textbf{%
0.635} & \textbf{1.101} \\ 
$0.9n^{-1/6}$ & \textbf{0.468} & 1.555 & \textbf{2.197} & 3.786 & \textbf{%
0.160} & 0.663 & 1.269 & 1.695 & 0.073 & 0.275 & 0.671 & 1.144 \\ 
$0.7n^{-1/6}$ & 0.494 & 1.433 & 2.247 & 3.870 & \textbf{0.160} & 0.690 & 
1.257 & 1.755 & \textbf{0.071} & 0.303 & 0.722 & 1.214 \\ 
$0.5n^{-1/6}$ & 0.521 & 1.402 & 2.473 & 3.802 & 0.175 & 0.755 & 1.334 & 1.774
& 0.081 & 0.339 & 0.764 & 1.261 \\ 
$0.3n^{-1/6}$ & 0.503 & \textbf{1.379} & 2.665 & 3.867 & 0.235 & 0.814 & 
1.445 & 1.916 & 0.121 & 0.408 & 0.843 & 1.343 \\ 
$0.1n^{-1/6}$ & 0.657 & 1.464 & 2.962 & 4.762 & 0.329 & 0.874 & 1.452 & 2.161
& 0.201 & 0.475 & 0.869 & 1.379 \\ \hline
\end{tabular}%
\end{table}
}

\end{document}